\pdfoutput=1
\documentclass[11pt]{amsart}
\usepackage{latexsym,exscale,enumerate,enumitem,amsfonts,mathtools}
\usepackage{amsfonts,amscd,textcomp,xcolor}
\usepackage{stmaryrd}
\SetSymbolFont{stmry}{bold}{U}{stmry}{m}{n}
\usepackage{amssymb}
\usepackage[normalem]{ulem}
\usepackage{young,youngtab}
\usepackage{thmtools}
\usepackage{easybmat}
\usepackage{pinlabel}
\usepackage[all]{xy}
\SelectTips{cm}{}

\usepackage{tikz}
\usetikzlibrary{decorations.markings}
\usetikzlibrary{decorations.pathreplacing}
\usetikzlibrary{arrows,shapes,positioning}
\tikzstyle directed=[postaction={decorate,decoration={markings,
    mark=at position #1 with {\arrow{>}}}}]
\tikzstyle rdirected=[postaction={decorate,decoration={markings,
    mark=at position #1 with {\arrow{<}}}}]
\tikzset{anchorbase/.style={baseline={([yshift=-0.5ex]current bounding box.center)}},
}

\newcommand{\qbinn}[2]{\genfrac{[}{]}{0pt}{}{#1}{#2}}
    
\newcommand{\comm}[1]{}
\def\cal#1{\mathcal{#1}}%
\newcommand{\sln}{\mathfrak{sl}_N}
\newcommand{\slnn}[1]{\mathfrak{sl}_{#1}}
\newcommand{\gln}{\mathfrak{gl}_N}
\newcommand{\glm}{\mathfrak{gl}_m}
\newcommand{\glmn}{\mathfrak{gl}_{m|n}}
\newcommand{\glMN}{\mathfrak{gl}_{N|M}}
\newcommand{\glNM}{\mathfrak{gl}_{M|N}}
\newcommand{\glnn}[1]{\mathfrak{gl}_{#1}}
\newcommand{\bV}{\raisebox{0.03cm}{\mbox{\footnotesize$\textstyle{\bigwedge}$}}}

\newcommand{\Gam}{\Gamma_{N}}
\newcommand{\Gami}{\Gamma_{\In}}
\newcommand{\Gams}{\Gamma_{N}^{\mathrm{sort}}}
\newcommand{\Lad}{\Upsilon^{m|n}_{\mathrm{su}}}
\newcommand{\AltSp}{N\text{-}\!\textbf{Web}_{\mathrm{g}}}
\newcommand{\CKM}{N\text{-}\!\textbf{Web}_{\mathrm{CKM}}}
\newcommand{\SymSp}{N\text{-}\!\textbf{Web}_{\mathrm{r}}}
\newcommand{\Sp}[2]{{#1}\text{-}\!\textbf{Web}_{\mathrm{{#2}}}}
\newcommand{\SuSp}{N\text{-}\!\textbf{Web}_{\mathrm{gr}}}
\newcommand{\MNSp}{N|M\text{-}\!\textbf{Web}_{\mathrm{gr}}}
\newcommand{\InSp}{\In\text{-}\!\textbf{Web}_{\mathrm{gr}}}
\newcommand{\InSpf}{\In\text{-}\!\textbf{Web}_{\mathrm{gr}}^f}
\newcommand{\SoSuSp}{N\text{-}\!\textbf{Web}_{\mathrm{gr}}^{\mathrm{sort}}}

\newcommand{\SoMNSp}{N|M\text{-}\!\textbf{Web}_{\mathrm{gr}}^{\mathrm{sort}}}

\newcommand{\Sym}{\mathrm{Sym}}
\newcommand{\Sup}{\bullet}
\newcommand{\Su}{\mathrm{su}}
\newcommand{\tr}{\mathrm{tr}}
\newcommand{\T}{\mathrm{T}}
\newcommand{\triv}{\;\raisebox{-0.02cm}{\mbox{{\tiny$\yng(2)$}}\;}}
\newcommand{\sign}{\;\raisebox{0.05cm}{\mbox{{\tiny\Yvcentermath1$\yng(1,1)$}}\;}}

\newcommand{\Uq}{\textbf{U}_q}
\newcommand{\Ud}{\dot{\textbf{U}}_q}
\newcommand{\Us}{\textbf{U}_q(\glmn)}
\newcommand{\Usu}{\textbf{U}_q(\glMN)}
\newcommand{\Usd}{\dot{\textbf{U}}_q(\glmn)}
\newcommand{\Uun}{\textbf{U}_q(\mathfrak{gl}_N)}

\newcommand{\Hee}{\check{\textbf{H}}}

\newcommand{\Repe}[1]{\mathfrak{gl}_{{#1}}\text{-}\textbf{Mod}_{e}}
\newcommand{\Reps}[1]{\mathfrak{gl}_{{#1}}\text{-}\textbf{Mod}_{s}}
\newcommand{\Repas}{\mathfrak{gl}_{N}\text{-}\textbf{Mod}_{\mathrm{es}}}
\newcommand{\Repsas}{\mathfrak{gl}_{N}\text{-}\textbf{Mod}^{\mathrm{sort}}_{\mathrm{es}}}
\newcommand{\RepMN}{\mathfrak{gl}_{N|M}\text{-}\textbf{Mod}_{\mathrm{es}}}
\newcommand{\RepMNs}{\mathfrak{gl}_{N|M}\text{-}\textbf{Mod}^{\mathrm{sort}}_{\mathrm{es}}}

\newcommand{\repe}[1]{\mathfrak{sl}_{{#1}}\text{-}\textbf{Mod}_{e}}

\newcommand{\repas}{\mathfrak{sl}_{N}\text{-}\textbf{Mod}_{\mathrm{es}}}

\newcommand{\Mat}{\textbf{Mat}}
\newcommand{\Hom}{\mathrm{Hom}}
\newcommand{\End}{\mathrm{End}}

\newcommand{\In}{\infty}

\newcommand{\MNw}{N|M\text{-}\!\textbf{Web}^{\mathrm{sort}}_{m+n}}

\def\C{{\mathbb C}}

\def\Ca{{\mathbb C}_{a,q}}

\def\Z{{\mathbb Z}}

\def\one{\mathrm{1}}

\theoremstyle{plain}
\newtheorem{thm}{Theorem}[section]
\newtheorem{cor}[thm]{Corollary}
\newtheorem{lem}[thm]{Lemma}
\newtheorem{prop}[thm]{Proposition}

\theoremstyle{definition}
\newtheorem{rem}[thm]{Remark}

\newtheorem*{thmm}{Theorem}
\newtheorem*{propp}{Proposition}

\declaretheorem[style=definition,name=Example,qed=$\triangleleft$,numberlike=thm]{ex}
\declaretheorem[style=definition,name=Definition,qed=$\Diamond$,numberlike=thm]{defn}

\definecolor{mycolor}{rgb}{0.9,0,0}
\definecolor{orchid}{RGB}{143,40,194}
\definecolor{lava}{RGB}{207,16,32}

\numberwithin{equation}{section}
\usepackage[final]{hyperref, bookmark}
\usepackage{cleveref}
\hypersetup{
    pdftoolbar=true,        
    pdfmenubar=true,        
    pdffitwindow=false,     
    pdfstartview={FitH},    
    pdftitle={Super \texorpdfstring{\MakeLowercase{q}}{q}-Howe duality and web categories},    
    pdfauthor={Daniel Tubbenhauer, Pedro Vaz and Paul Wedrich},     
    pdfsubject={},   
    pdfcreator={Daniel Tubbenhauer, Pedro Vaz and Paul Wedrich},   
    pdfproducer={Daniel Tubbenhauer, Pedro Vaz and Paul Wedrich}, 
    pdfkeywords={}, 
    pdfnewwindow=true,      
    colorlinks=true,       
    linkcolor=lava,          
    citecolor=teal,        
    filecolor=magenta,      
    urlcolor=orchid,          
    linkbordercolor=lava,
    citebordercolor=teal,
    urlbordercolor=orchid,  
    linktocpage=true
}

\let\fullref\autoref
\def\makeautorefname#1#2{\expandafter\def\csname#1autorefname\endcsname{#2}}
\makeautorefname{equation}{Equation}%
\makeautorefname{footnote}{footnote}%
\makeautorefname{item}{item}%
\makeautorefname{figure}{Figure}%
\makeautorefname{table}{Table}%
\makeautorefname{part}{Part}%
\makeautorefname{appendix}{Appendix}%
\makeautorefname{chapter}{Chapter}%
\makeautorefname{section}{Section}%
\makeautorefname{subsection}{Section}%
\makeautorefname{subsubsection}{Section}%
\makeautorefname{paragraph}{Paragraph}%
\makeautorefname{subparagraph}{Paragraph}%
\makeautorefname{theorem}{Theorem}%
\makeautorefname{theo}{Theorem}%
\makeautorefname{thm}{Theorem}%
\makeautorefname{addendum}{Addendum}%
\makeautorefname{addend}{Addendum}%
\makeautorefname{add}{Addendum}%
\makeautorefname{maintheorem}{Main theorem}%
\makeautorefname{mainthm}{Main theorem}%
\makeautorefname{corollary}{Corollary}%
\makeautorefname{corol}{Corollary}%
\makeautorefname{coro}{Corollary}%
\makeautorefname{cor}{Corollary}%
\makeautorefname{lemma}{Lemma}%
\makeautorefname{lemm}{Lemma}%
\makeautorefname{lem}{Lemma}%
\makeautorefname{sublemma}{Sublemma}%
\makeautorefname{sublem}{Sublemma}%
\makeautorefname{subl}{Sublemma}%
\makeautorefname{proposition}{Proposition}%
\makeautorefname{proposit}{Proposition}%
\makeautorefname{propos}{Proposition}%
\makeautorefname{propo}{Proposition}%
\makeautorefname{prop}{Proposition}%
\makeautorefname{property}{Property}
\makeautorefname{proper}{Property}
\makeautorefname{scholium}{Scholium}%
\makeautorefname{step}{Step}%
\makeautorefname{conjecture}{Conjecture}%
\makeautorefname{conject}{Conjecture}%
\makeautorefname{conj}{Conjecture}%
\makeautorefname{question}{Question}
\makeautorefname{questn}{Question}
\makeautorefname{quest}{Question}
\makeautorefname{ques}{Question}
\makeautorefname{qn}{Question}
\makeautorefname{definition}{Definition}%
\makeautorefname{defin}{Definition}%
\makeautorefname{defi}{Definition}%
\makeautorefname{def}{Definition}%
\makeautorefname{dfn}{Definition}%
\makeautorefname{notation}{Notation}
\makeautorefname{nota}{Notation}
\makeautorefname{notn}{Notation}
\makeautorefname{remark}{Remark}%
\makeautorefname{rema}{Remark}%
\makeautorefname{rem}{Remark}%
\makeautorefname{rmk}{Remark}%
\makeautorefname{rk}{Remark}%
\makeautorefname{remarks}{Remarks}%
\makeautorefname{rems}{Remarks}%
\makeautorefname{rmks}{Remarks}%
\makeautorefname{rks}{Remarks}%
\makeautorefname{example}{Example}%
\makeautorefname{examp}{Example}%
\makeautorefname{exmp}{Example}%
\makeautorefname{exam}{Example}%
\makeautorefname{exa}{Example}%
\makeautorefname{algorithm}{Algorith}%
\makeautorefname{algo}{Algorith}%
\makeautorefname{alg}{Algorith}%
\makeautorefname{axiom}{Axiom}%
\makeautorefname{axi}{Axiom}%
\makeautorefname{ax}{Axiom}%
\makeautorefname{case}{Case}%
\makeautorefname{claim}{Claim}%
\makeautorefname{clm}{Claim}%
\makeautorefname{assumption}{Assumption}%
\makeautorefname{assumpt}{Assumption}%
\makeautorefname{conclusion}{Conclusion}%
\makeautorefname{concl}{Conclusion}%
\makeautorefname{conc}{Conclusion}%
\makeautorefname{condition}{Condition}%
\makeautorefname{condit}{Condition}%
\makeautorefname{cond}{Condition}%
\makeautorefname{construction}{Construction}%
\makeautorefname{construct}{Construction}%
\makeautorefname{const}{Construction}%
\makeautorefname{cons}{Construction}%
\makeautorefname{criterion}{Criterion}%
\makeautorefname{criter}{Criterion}%
\makeautorefname{crit}{Criterion}%
\makeautorefname{exercise}{Exercise}%
\makeautorefname{exer}{Exercise}%
\makeautorefname{exe}{Exercise}%
\makeautorefname{problem}{Problem}%
\makeautorefname{problm}{Problem}%
\makeautorefname{probm}{Problem}%
\makeautorefname{prob}{Problem}%
\makeautorefname{solution}{Solution}%
\makeautorefname{soln}{Solution}%
\makeautorefname{sol}{Solution}%
\makeautorefname{summary}{Summary}%
\makeautorefname{summ}{Summary}%
\makeautorefname{sum}{Summary}%
\makeautorefname{operation}{Operation}%
\makeautorefname{oper}{Operation}%
\makeautorefname{observation}{Observation}%
\makeautorefname{observn}{Observation}%
\makeautorefname{obser}{Observation}%
\makeautorefname{obs}{Observation}%
\makeautorefname{ob}{Observation}%
\makeautorefname{convention}{Convention}%
\makeautorefname{convent}{Convention}%
\makeautorefname{conv}{Convention}%
\makeautorefname{cvn}{Convention}%
\makeautorefname{warning}{Warning}%
\makeautorefname{warn}{Warning}%
\makeautorefname{note}{Note}%
\makeautorefname{fact}{Fact}%

\begin{document}
\vbadness=10001
\hbadness=10001
\title[Super \texorpdfstring{\MakeLowercase{q}}{q}-Howe duality and web categories]{Super \texorpdfstring{\MakeLowercase{q}}{q}-Howe duality and web categories}

\author[D. Tubbenhauer]{Daniel Tubbenhauer}
\address{D.T.: Mathematisches Institut, Universit\"at Bonn, Endenicher Allee 60, Room 1.003, 53115 Bonn, Germany}
\email{dtubben@math.uni-bonn.de}

\author[P. Vaz]{Pedro Vaz}
\address{P.V.: Institut de recherche en math\'{e}matique et physique (IRMP), Universit\'{e} catholique de Louvain, Chemin du Cyclotron 2, building L7.01.02, 1348 Louvain-la-Neuve, Belgium}
\email{pedro.vaz@uclouvain.be}

\author[P. Wedrich]{Paul Wedrich}
\address{P.W.: Imperial College London, Department of Mathematics, 6M50 Huxley Building, South Kensington Campus London SW7 2AZ, United Kingdom}
\email{p.wedrich@gmail.com}

\begin{abstract}
We use super $q$-Howe duality to provide diagrammatic 
presentations of an idempotented form of the Hecke 
algebra and of categories of $\gln$-modules 
(and, more generally, $\glMN$-modules) whose 
objects are tensor generated by exterior and symmetric 
powers of the vector representations. As an application, 
we give a representation theoretic explanation 
and a diagrammatic version of a 
known symmetry of colored HOMFLY--PT polynomials.
\end{abstract}

\maketitle

\tableofcontents

\renewcommand{\theequation}{\thesection-\arabic{equation}}
%
%%%%%%%%%%%%%%%%%%%%%%%%%%%%%%%%%%%%%%%%
%%%                                  %%%
%%%        Introduction              %%%
%%%                                  %%%
%%%%%%%%%%%%%%%%%%%%%%%%%%%%%%%%%%%%%%%%
\section{Introduction}\label{sec-intro}
%%%%%%%%%%%%%%%%%%%%%%%%%%%%%%%%%%%%%%%%%%%%%%%
Let $\Uun$ be the 
quantum enveloping $\C_q=\C(q)$-algebra for 
$\gln$ with $q$ being generic. 
Let $\Repas$ denote the braided monoidal 
category of 
$\Uun$-modules\footnote{We only consider finite-dimensional, left modules 
(of type $1$) throughout the paper.} tensor 
generated by 
\textit{exterior} $\bV^k_q\C_q^N$ \textit{and} 
\textit{symmetric} $\Sym_q^l\C_q^N$ powers and 
$\Uun$-intertwiners between them.

We denote by 
$\Hee$ an \textit{idempotented version} of 
the direct sum of all 
Iwahori--Hecke algebras 
$H_{\In}(q)\!=\!\bigoplus_{K\in\Z_{\geq 0}}\! H_K(q)$ of type $A$. 
Roughly, $\Hee$ is the category obtained from 
the one-object category $H_{\In}(q)$ by adding formal 
Gyoja-Aiston idempotents corresponding to column 
and row Young diagrams as new objects\footnote{Adding only column idempotents, one obtains the type A Schur algebroids introduced in \cite{Wil}.}. By 
\textit{quantum Schur--Weyl duality}, the categories 
$\Repas$ are quotients of $\Hee$ and the 
added idempotents can be thought of as lifts of 
the exterior $\bV^k_q\C_q^N$ and the 
symmetric $\Sym_q^l\C_q^N$ powers.

We construct diagrammatic presentations of $\Hee$ and $\Repas$ 
by using the \textit{green--red} web categories $\InSp$ and $\SuSp$. 
Morphisms in these $\C_q$-linear categories 
are combinations of planar, upward-oriented, 
trivalent graphs with edges labeled by positive integers 
and colored black, green or red\footnote{We use colored diagrams 
in this paper. The colors (black, green and red) are 
important and we recommend to read the paper in color. 
If the reader has a black-and-white version, then green 
will appear lightly shaded and black and red can be 
distinguished since black edges are always labeled $1$.} 
modulo local relations. 
Objects are boundaries of such green--red webs, i.e.\ 
finite sequences of positive integers, each of 
which additionally carries the color black, green or red, 
indicated either by an actual coloring or by a subscript.

An example of a green--red web is:
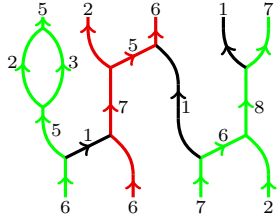
\[
\xy
(0,0)*{
\begin{tikzpicture}[scale=.3]
	\draw [very thick, mycolor, directed=.55] (1,.75) to (1,2);
	\draw [very thick, black] (2,-1) to [out=90,in=320] (1,.75);
	\draw [very thick, black, rdirected=0.65] (2,-1) to (2,-2);
	\draw [very thick, mycolor, directed=0.1] (-1,-2) to (-1,-.25);
	\draw [very thick, mycolor] (-1,-0.25) to [out=150,in=270] (-2,1.5);
	\draw [very thick, mycolor, directed=0.35] (-2,1.5) to (-2,2);
	\draw [very thick, mycolor, directed=.55] (-1,-.25) to (1,.75);
	\draw [very thick, green] (5,-3.25) to (5,-2);
	\draw [very thick, green] (6,-5) to [out=90,in=320] (5,-3.25);
	\draw [very thick, green, rdirected=0.65] (6,-5) to (6,-6);
	\draw [very thick, green, directed=.55] (3,-6) to (3,-4.25);
	\draw [very thick, black] (3,-4.25) to [out=150,in=270] (2,-2.5);
	\draw [very thick, black] (2,-2.5) to (2,-2);
	\draw [very thick, green, directed=.55] (3,-4.25) to (5,-3.25);
	\draw [very thick, mycolor] (-1,-3.25) to (-1,-2);
	\draw [very thick, mycolor] (0,-5) to [out=90,in=320] (-1,-3.25);
	\draw [very thick, mycolor, rdirected=0.65] (0,-5) to (0,-6);
	\draw [very thick, green, directed=.55] (-3,-6) to (-3,-4.25);
	\draw [very thick, green] (-3,-4.25) to [out=150,in=270] (-4,-2.5);
	\draw [very thick, green, directed=0.35] (-4,-2.5) to (-4,-2);
	\draw [very thick, black, directed=.55] (-3,-4.25) to (-1,-3.25);
	\draw [very thick, green, directed=.2] (5,-2) to (5,-.25);
	\draw [very thick, green, directed=.65] (5,-.25) to [out=30,in=270] (6,1.5);
	\draw [very thick, black, directed=.65] (5,-.25) to [out=150,in=270] (4,1.5); 
	\draw [very thick, green] (6,1.5) to (6,2);
	\draw [very thick, black] (4,1.5) to (4,2);
	\draw [very thick, green, directed=.55] (-4,-2) to [out=30,in=330] (-4,1.5);
	\draw [very thick, green, directed=.55] (-4,-2) to [out=150,in=210] (-4,1.5);
	\draw [very thick, green, directed=.95] (-4,1.5) to (-4,2);
	\node at (-4,2.35) {\tiny $5$};
	\node at (-2,2.35) {\tiny $2$};
	\node at (1,2.35) {\tiny $6$};
	\node at (4,2.35) {\tiny $1$};
	\node at (6,2.35) {\tiny $7$};
	\node at (-3,-6.45) {\tiny $6$};
	\node at (0,-6.45) {\tiny $6$};
	\node at (3,-6.45) {\tiny $7$};
	\node at (6,-6.45) {\tiny $2$};
	\node at (-3.4,-3) {\tiny $5$};
	\node at (-0.45,-2) {\tiny $7$};
	\node at (2.35,-2) {\tiny $1$};
	\node at (5.6,-2) {\tiny $8$};
	\node at (-5.35,0) {\tiny $2$};
	\node at (-2.65,0) {\tiny $3$};
	\node at (0,0.75) {\tiny $5$};
	\node at (-2,-3.25) {\tiny $1$};
	\node at (4,-3.25) {\tiny $6$};
\end{tikzpicture}
};
\endxy
\]
A green integer $k$ in a boundary sequence is 
meant to correspond to the $\Uun$-module
$\bV^k_q\C_q^N$, a red integer 
$l$ to $\Sym_q^l\C_q^N$, and sequences of integers 
correspond to tensor products of such. Vertical edges 
are identities on these $\Uun$-modules and trivalent 
vertices encode more interesting $\Uun$-intertwiners. 
The integer $1$ should be 
$\C_q^N\cong\bV^1_q\C_q^N\cong\Sym_q^1\C_q^N$ independent 
of the color green or red, so we color it black.

Our main result is:

\begin{thmm}(\textbf{The diagrammatic presentation})
The additive closures of $\InSp$ and of $\SuSp$ are 
braided monoidally equivalent to $\Hee$ and $\Repas$ respectively.
\end{thmm}

We will see that $\InSp$ admits an involution interchanging the colors 
green and red. An almost direct consequence of this is a symmetry
between the HOMFLY--PT polynomial 
$\cal{P}^{a,q}(\cdot)$ of a link $\cal{L}$ colored with
$\vec{\lambda}=(\lambda^1,\dots,\lambda^d)$ and the HOMFLY--PT polynomial of 
$\cal{L}$ colored with 
$\vec{\lambda}^{\phantom{.}\T}=((\lambda^1)^{\T},\dots,(\lambda^d)^{\T})$:

\begin{propp} (\textbf{The colored HOMFLY--PT symmetry})
We have
\begin{equation}\label{eq-conjecture}
\cal{P}^{a,q}(\cal{L}(\vec{\lambda}))\;=\;(-1)^{c}\;
\cal{P}^{a,q^{-1}}(\cal{L}(\vec{\lambda}^{\phantom{.}\T})).
\end{equation}
Here $c$ is the sum of the number of nodes in the 
Young diagrams $\lambda^i$ for $1\leq i\leq d$.
\end{propp}

Our results might help to understand symmetries observed within the homologies that categorify the colored HOMFLY--PT polynomials, see \cite[Section 5]{gs}.

Moreover, we show that a straightforward 
generalization of our approach also leads to diagrammatic 
presentations for categories $\RepMN$ of $\Usu$-modules 
tensor generated by exterior and symmetric powers of the 
vector representation. The presentations are given by 
quotients $\MNSp$ of $\InSp$, which are obtained by 
killing Gyoja--Aiston idempotents corresponding to box-shaped Young diagrams.

\subsection{The framework}\label{sub-intropart1}

A prototypical diagrammatic presentation result (with roots in the work of Rumer, Teller 
and Weyl \cite{rtw}) states that the 
\textit{Temperley--Lieb category} 
gives a presentation of  
the full subcategory of $\Uq(\slnn{2})$-modules 
tensor generated by the vector representation 
$\comm{\bV_q^1\C_q^2\cong}\C_q^2$.
Kuperberg \cite{kup} extended this 
to all rank $2$ Lie algebras. 
In particular, he described a 
presentation
of the full subcategory 
of $\Uq(\slnn{3})$-modules tensor 
generated by the \textit{exterior 
powers} $\bV_q^1\C_q^3\cong\C_q^3$ and 
$\bV_q^2\C_q^3$. 
More generally, Cautis, Kamnitzer and Morrison \cite{ckm} gave a 
presentation of $\Repe{N}$, the full subcategory of 
$\Uun$-modules tensor generated by the \textit{exterior 
powers} $\bV_q^k\C_q^N$ for $k=0,\dots,N$.

One of their key ideas in \cite{ckm}, 
is the usage of 
\textit{skew quantum Howe duality} (or short, \textit{skew $q$-Howe duality}). 
In order to explain their approach, let $\vec{k}\in\Z^m_{\geq 0}$ be
such that $k_1+\cdots+k_m=K$. 
By skew $q$-Howe duality, the commuting actions of $\Uq(\glm)$ and $\Uq(\gln)$ on
\[
\bV_q^K(\C_q^m\otimes\C_q^N)\cong{\textstyle\bigoplus_{\vec{k}\in\Z^m_{\geq 0}}}\bV_q^{k_1}\C_q^N
\otimes\cdots\otimes\bV_q^{k_m}\C_q^N
\]
give rise to a functor
$\Phi_{\mathrm{skew}}^{m}\colon\Ud(\mathfrak{gl}_m)\to\Repe{N}$, 
where $\Ud(\mathfrak{gl}_m)$ is the idempotented form of 
$\Uq(\mathfrak{gl}_m)$.
Then Cautis, Kamnitzer and Morrison construct a commutative diagram, which takes the following form in our notation
\footnote{We consider $\Repas$ 
instead of $\repas$, see also \fullref{rem-whatwenotdo}.}:
\begin{equation}\label{eq-intro}
\begin{gathered}
\xymatrix{
\Ud(\mathfrak{gl}_m) \ar[r]^{\Phi_{\mathrm{skew}}^{m}} \ar[dr]_{\Upsilon_{\mathrm{skew}}^m} & \Repe{N} \\
& \Sp{N}{g} \ar[u]_{\Gamma}
}
\end{gathered}
\end{equation}
Here $\Upsilon_{\mathrm{skew}}^m$ is 
a certain \textit{ladder functor} realizing an 
action of $\Ud(\mathfrak{gl}_m)$ on 
the diagram category $\Sp{N}{g}$. 
The \textit{presentation functor} $\Gamma$ is 
constructed so that \eqref{eq-intro} commutes. 
The functor 
$\Phi_{\mathrm{skew}}^{m}$ 
is full and
its kernel is generated by
killing $\glm$-weights with 
entries not in $\{0,\dots,N\}$. That $\Gamma$ is 
an equivalence follows since $\Sp{N}{g}$ is defined 
to be the quotient of a ``free'' web category by relations 
coming from $\Ud(\mathfrak{gl}_m)$ (to make the ladder 
functor $\Upsilon_{\mathrm{skew}}^m$ well-defined) and the 
$\Upsilon_{\mathrm{skew}}^m$ image of the kernel of $\Phi_{\mathrm{skew}}^{m}$. $\repe{N}$ can be recovered 
by identifying $\bV_q^k\C_q^N\cong(\bV_q^{N-k}\C_q^N)^*$ as 
$\Uq(\slnn{N})$-modules.

Rose and the first-named author \cite{rt} studied
the situation of \textit{symmetric quantum Howe 
duality} (for short, symmetric $q$-Howe duality) was studied\footnote{In fact, the observations made in the paper \cite{rt} were one of the main motivations to start this project.}. That is, there is 
an analogue of \eqref{eq-intro} where $\Repe{N}$ is replaced by $\Reps{N}$, the 
full subcategory of $\Uun$-modules tensor generated by 
the \textit{symmetric powers} $\Sym_q^l\C_q^N$ for $l\in\Z_{\geq 0}$.
In the $N=2$ case, the kernel of $\Phi_{\mathrm{sym}}^{m}$ is generated by killing $\glm$-weights 
with negative entries and one additional \textit{dumbbell relation}, which 
encodes the relation 
$\C_q^2\otimes \C_q^2\cong \C_q\oplus \Sym_q^2\C_q^2$ in $\Reps{2}$. 
A direct generalization 
for $N>2$ would require additional complicated relations besides killing $\glm$-weights.

In this paper we give a diagrammatic presentation 
of the category $\Repas$, the 
full subcategory of $\Uun$-modules 
tensor generated by both 
exterior \textit{and} symmetric powers of the vector 
representation. This diagrammatic presentation gives a common 
generalization of the web categories of \cite{ckm} (only black--green webs) 
and \cite{rt} (only black--red webs).
We see Cautis, 
Kamnitzer and Morrison's approach 
as a \textit{machine} that takes dualities 
and produces diagrammatic presentations of the related 
representation theoretical categories.
Specifically, we 
start with \textit{super quantum Howe duality} 
(for short, super $q$-Howe duality) 
between the superalgebra $\Us$ and $\Uun$. 
We obtain a full super $q$-Howe functor 
$\Phi_{\Su}^{m|n}$, which we attempt to factor 
as a composite of a ladder functor $\Lad$---mapping 
into an appropriate web category---and a diagrammatic 
presentation functor $\Gam$, to give an analogue of 
\eqref{eq-intro}\footnote{Here the superscript ``$\mathrm{sort}$'' indicates subcategories in which exterior powers are \emph{sorted} to the left of symmetric powers in tensor products. This small technical restriction stems from the use of super q-Howe duality, but will be removed later on.}:
\[
\xymatrix{
\Usd \ar[r]^{\Phi_{\Su}^{m|n}} \ar[dr]_{\Lad} & \Repsas \\
& \SoSuSp \ar[u]_{\Gams}
}
\]

Having decided to 
follow this strategy, the definition of the appropriate web 
category
is already determined. Two aspects are important:

\begin{enumerate}[label=(\Roman*)]
\item In order to make $\Lad$ well-defined, 
the web category needs to satisfy ladder images of $\Usd$ 
relations. Remarkably, it suffices to consider 
relations coming from the subalgebra $\Ud(\glm)\oplus\Ud(\glnn{n})$ 
and only one additional \textit{super commutation} relation $[2]\one_{\vec{k}}=F_mE_m\one_{\vec{k}}+E_mF_m\one_{\vec{k}}$ for $\glmn$-weights with $k_m=k_{m+1}=1$. 
This corresponds to the \textit{dumbbell relation} on webs and to 
$\C_q^N\otimes\C_q^N\cong \bV_q^2\C_q^N\oplus\Sym_q^2\C_q^N$ in 
$\Repas$.

\item In order to make the diagrammatic presentation 
functor an equivalence, we need to impose the ladder 
image of $\ker(\Phi_{\Su}^{m|n})$ as relations in the 
web category. In fact, $\ker(\Phi_{\Su}^{m|n})$ is 
spanned by idempotents corresponding to $\glmn$-weights $\vec{k}=(k_1,\dots,k_{m+n})$ 
with $k_1,\dots,k_m\notin\{0,\dots,N\}$ or 
$k_{m+1},\dots,k_{m+n}\notin\Z_{\geq 0}$. It is 
remarkable that no extra relations, aside 
from killing these $\glmn$-weights, 
are necessary.
\end{enumerate}

We impose the ladder images of 
$\ker(\Phi_{\Su}^{m|n})$ in two steps: first we 
kill all $\glmn$-weights with negative entries by 
allowing only non-negative labels on web edges. This 
produces the web category $\InSp$, which is symmetric 
under exchanging green and red. On this we further quotient by 
setting $\glmn$-weights $\vec{k}=(k_1,\dots,k_{m+n})$ to zero 
if one of $k_1,\dots,k_m$ is greater than $N$. This produces 
the web category $\SuSp$ and in \fullref{thm-equi} we show 
that its additive closure is equivalent to $\Repas$. 
Note that, although our graphical calculus is finer than the one in \cite{ckm} in the sense that it contains more objects, the Karoubi envelopes of these diagrammatic categories agree for each $N$.

In \fullref{thm-equi2} we use \textit{quantum Schur--Weyl duality} 
to derive from \fullref{thm-equi} that $\InSp$ gives a 
diagrammatic presentation of the idempotented 
Iwahori--Hecke algebra $\Hee$ from above.

\begin{rem}\label{rem-whatwenotdo}
We describe $\Repas$ and not 
$\repas$ because of the algebraic form of super $q$-Howe 
duality. In particular, our web categories do not contain 
duality isomorphisms $\bV_q^k\C_q^N\cong(\bV_q^{N-k}\C_q^N)^*$, 
which would be necessary for a diagrammatic presentation of $\repas$. 
In $\Repas$, on the other hand, there are no such hidden duals, 
as we have $\bV_q^k\C_q^N\cong \bV_q^N\C_q^N \otimes(\bV_q^{N-k}\C_q^N)^*$ 
as $\Uun$-modules. Here 
$\bV_q^N\C_q^N \cong L((1,\dots,1))$ is the $\Uun$-module 
of highest weight $\lambda=(1,\dots,1)\in\Z_{\geq 0}^N$.
\end{rem}

Last, but not least, we use the more general \textit{super} 
$q$\textit{-Howe duality between} $\Us$ \textit{and} $\Usu$ 
to describe $\RepMN$. Feeding this duality into the 
``diagrammatic presentation machine'' shows that this 
representation category is equivalent to the quotient 
$\MNSp$ of $\InSp$, which is obtained by killing the Gyoja--Aiston 
idempotent corresponding to the size $(N+1)\times(M+1)$ box-shaped 
Young diagram. This is a generalization, since for $M=0$, $\RepMN$ 
is equivalent to $\Repas$ and $\MNSp$ is equal to $\SuSp$, 
because the box idempotent corresponds exactly to an $(N+1)$-labeled green edge.

This generalizes Grant's \cite{g} 
and Sartori's \cite{sar} presentations of the category 
$\mathfrak{gl}_{1|1}\text{-}\textbf{Mod}_{\mathrm{e}}$, 
and the diagrammatic calculus 
for $\mathfrak{gl}_{N|M}\text{-}\textbf{Mod}_{\mathrm{e}}$ given 
in \cite{qs} (see also \cite{g2}). Compared to the latter, our 
generalization, which also takes the symmetric powers of $\C_q^{N|M}$ 
into account, does not need any extra relations aside from the dumbbell 
relation. In fact, the one extra relation needed to make the 
diagrammatic calculus given in \cite{qs} faithful, 
see \cite[Remark 6.19]{qs}, has a very compact and natural 
description in our green--red web category $\MNSp$.

Finally, we sketch how our presentation 
of $\RepMN$ extends to take duals of exterior and symmetric 
powers into account. This closely follows \cite[Section 6]{qs}. 
The resulting diagrammatic category allows the computation of 
the colored Reshetikhin--Turaev $\glMN$-link invariants. 
In \fullref{cor-supersym}, we interpret the colored 
HOMFLY--PT symmetry \eqref{eq-conjecture} as a stable 
version of a symmetry between colored Reshetikhin--Turaev 
$\glMN$- and $\glNM$-link invariants.

\subsection{Outline of the paper}\label{sub-intropart2}

\fullref{sec-super} is the diagrammatic heart of 
our paper where we introduce 
$\InSp$ and 
its subquotients $\SuSp$, $\AltSp$ and $\SymSp$.

\fullref{sec-proof} contains the proof of our 
main theorems and splits into three subsections: 
We first introduce super $q$-Howe duality. 
Then we show an equivalence between ``sorted'' 
subcategories of $\SuSp$ and $\Repas$. These subcategories are 
induced by the algebraic form of super $q$-Howe duality.
By using the ``sorted'' equivalence and
the fact that the braiding 
gives a way to ``shuffle'' 
the ``sorted'' subcategories, we prove our main theorems.

In \fullref{sec-applications} we discuss 
one application of 
our diagrammatic presentation: we give a procedure to recover the 
colored HOMFLY--PT polynomial from $\InSp$. A direct consequence 
of the green--red symmetry is a symmetry within the 
colored HOMFLY--PT polynomial obtained by 
transposing Young diagrams, see \eqref{eq-conjecture}.
The colored Reshetikhin--Turaev $\sln$-link polynomials can be 
recovered from our approach as well, as we sketch in the last subsection.

Finally, in \fullref{sec-gen} we generalize the diagrammatic presentation of $\Repas$ to the super case $\RepMN$, and we sketch an extension of our diagrammatic calculus to include dual representations.
The required arguments are---mutatis mutandis---contained in the previous sections and in \cite[Section 6]{qs}, which allows a very compact exposition in \fullref{sec-gen}.
\newline\newpage
\noindent \textbf{Acknowledgements:} We especially thank 
Antonio Sartori for a careful reading of a draft version 
of this paper and many helpful comments, and David Rose -- some 
of the ideas underlying this paper came up in the joint 
work between him and D.T. We also thank Jonathan Brundan, 
Jonathan Grant, Jonathan Kujawa, Marco Mackaay, Weiwei Pan, 
Jake Rasmussen, Marko Sto\v{s}i\'{c}, Catharina Stroppel, Geordie 
Williamson and Oded Yacobi for helpful discussion, comments, and 
probing questions, and two referees for some further helpful 
comments. We also like to thank Skype for many useful conversations.

D.T. was supported by a research funding of the ``Deutsche Forschungsgemeinschaft (DFG)'' during the main part of this work. D.T. and P.W. thank the center of excellence grant ``Centre for Quantum Geometry of Moduli Spaces (QGM)'' from the ``Danish National Research Foundation (DNRF)'' for sponsoring a research visit which started this project. P.V. was financially supported by the ``Universit\'{e} catholique de Louvain, Fonds Sp\'{e}ciaux de Recherche (FSR) 12 J.A.''. P.W. was supported by an EPSRC doctoral training grant.
%
%%%%%%%%%%%%%%%%%%%%%%%%%%%%%%%%%%%%%%%%
%%%                                  %%%
%%%       super spider               %%%
%%%                                  %%%
%%%%%%%%%%%%%%%%%%%%%%%%%%%%%%%%%%%%%%%%
\section{The diagrammatic categories}\label{sec-super}
%%%%%%%%%%%%%%%%%%%%%%%%%%%%%%%%%%%%%%%%%%%%%%%
In the present section we 
introduce the category 
$\InSp$ and its quotient  
$\SuSp$.
These provide diagrammatic 
presentations of $\Hee$ 
and its quotient categories $\Repas$ respectively. Other subquotients of $\InSp$ are 
$\AltSp$ and $\SymSp$ (and later in \fullref{sec-gen}, $\MNSp$) which are 
related to categories studied in \cite{ckm} and \cite{rt} respectively.

\subsection{Definition of the category \texorpdfstring{$\InSp$}{inftyWebgr} and its subquotients}\label{sub-defsuper}

We first introduce the \textit{free 
green--red web category} $\InSpf$.
To this end, we denote by $X$ the set
\[
X=X_b\cup X_g\cup X_r=\{0_b,1_b\}\cup\{2_g,3_g,\dots\}\cup \{2_r,3_r,\dots\},
\]
where we think of the elements of $X_b$ as being colored \textit{black}, 
of the elements of $X_g$ as being colored \textit{green} and of the 
elements of $X_r$ as being colored \textit{red}. We usually omit the 
subscripts, since the colors on 
the boundary can be read off from the diagrams.

\begin{defn}\label{defn-fspid} 
The \textit{free green--red web category}, 
which we denote by $\InSpf$, is the category determined by the following data.

\begin{itemize}
\item The objects of $\InSpf$ are finite (possibly empty) 
sequences $\vec{k}\in X^L$ with entries from $X$ 
for some $L\in\Z_{\geq 0}$, together with a zero object. 
We display the entries of $\vec{k}$ ordered from 
left to right according to their appearance in $\vec{k}$.

\item The morphism space $\Hom_{\InSpf}(\vec{k},\vec{l})$ from 
$\vec{k}$ to $\vec{l}$ is the $\C_q$-vector 
space spanned by 
isotopy classes\footnote{We require that 
isotopies preserve the upward orientations and the 
boundary of green--red webs.} of 
planar, upward-oriented, trivalent graphs 
with edges labeled by positive integers 
and colored black, green or red, with bottom boundary $\vec{k}$ and 
top boundary $\vec{l}$. More precisely, 
we only allow webs that can be obtained 
by composition $\circ$ (vertical gluing) 
and taking the monoidal product $\otimes$ 
(horizontal juxtaposition) of the following 
basic pieces (including the empty diagram).

Let $k,l\in\Z_{\geq 2}$, then the generators are
\begin{equation}\label{eq-symgen}
\xy
(0,0)*{
\begin{tikzpicture} [scale=1]
\node at (0,-.15) {\tiny $0$};
\node at (0,1.15) {\tiny $0$};
\end{tikzpicture}
}
\endxy
\quad , \quad
\xy
(0,0)*{
\begin{tikzpicture} [scale=1]
\draw[very thick, directed=.55] (0,0) to (0,1);
\node at (0,-.15) {\tiny $1$};
\node at (0,1.15) {\tiny $1$};
\end{tikzpicture}
}
\endxy
\quad , \quad
\xy
(0,0)*{
\begin{tikzpicture} [scale=1]
\draw[very thick, green, directed=.55] (0,0) to (0,1);
\node at (0,-.15) {\tiny $k$};
\node at (0,1.15) {\tiny $k$};
\end{tikzpicture}
}
\endxy
\quad , \quad
\xy
(0,0)*{
\begin{tikzpicture} [scale=1]
\draw[very thick, mycolor, directed=.55] (0,0) to (0,1);
\node at (0,-.15) {\tiny $k$};
\node at (0,1.15) {\tiny $k$};
\end{tikzpicture}
}
\endxy
\quad , \quad
\xy
(0,0)*{
\begin{tikzpicture}[scale=.3]
	\draw [very thick, green, directed=.55] (0, .75) to (0,2.5);
	\draw [very thick, green, directed=.55] (1,-1) to [out=90,in=330] (0,.75);
	\draw [very thick, green, directed=.55] (-1,-1) to [out=90,in=210] (0,.75); 
	\node at (0, 3) {\tiny $k{+}l$};
	\node at (-1,-1.5) {\tiny $k$};
	\node at (1,-1.5) {\tiny $l$};
\end{tikzpicture}
};
\endxy
\quad , \quad
\xy
(0,0)*{
\begin{tikzpicture}[scale=.3]
	\draw [very thick, green, directed=.55] (0,-1) to (0,.75);
	\draw [very thick, green, directed=.55] (0,.75) to [out=30,in=270] (1,2.5);
	\draw [very thick, green, directed=.55] (0,.75) to [out=150,in=270] (-1,2.5); 
	\node at (0, -1.5) {\tiny $k{+}l$};
	\node at (-1,3) {\tiny $k$};
	\node at (1,3) {\tiny $l$};
\end{tikzpicture}
};
\endxy
\quad , \quad
\xy
(0,0)*{
\begin{tikzpicture}[scale=.3]
	\draw [very thick, mycolor, directed=.55] (0, .75) to (0,2.5);
	\draw [very thick, mycolor, directed=.55] (1,-1) to [out=90,in=330] (0,.75);
	\draw [very thick, mycolor, directed=.55] (-1,-1) to [out=90,in=210] (0,.75); 
	\node at (0, 3) {\tiny $k{+}l$};
	\node at (-1,-1.5) {\tiny $k$};
	\node at (1,-1.5) {\tiny $l$};
\end{tikzpicture}
};
\endxy
\quad , \quad
\xy
(0,0)*{
\begin{tikzpicture}[scale=.3]
	\draw [very thick, mycolor, directed=.55] (0,-1) to (0,.75);
	\draw [very thick, mycolor, directed=.55] (0,.75) to [out=30,in=270] (1,2.5);
	\draw [very thick, mycolor, directed=.55] (0,.75) to [out=150,in=270] (-1,2.5); 
	\node at (0, -1.5) {\tiny $k{+}l$};
	\node at (-1,3) {\tiny $k$};
	\node at (1,3) {\tiny $l$};
\end{tikzpicture}
};
\endxy
\end{equation}
called (from left to right) \textit{empty identity},
\textit{black identity}, \textit{green identity}, 
\textit{red identity}, \textit{green merge},  
\textit{green split}, \textit{red merge} and 
\textit{red split}, together with (here $k,l\in\Z_{\geq 0}$)
\begin{equation}\label{eq-symgen2}
\xy
(0,0)*{
\begin{tikzpicture}[scale=.3]
	\draw [very thick, green, directed=.55] (0, .75) to (0,2.5);
	\draw [very thick, directed=.55] (1,-1) to [out=90,in=330] (0,.75);
	\draw [very thick, green, directed=.55] (-1,-1) to [out=90,in=210] (0,.75); 
	\node at (0, 3) {\tiny $k{+}1$};
	\node at (-1,-1.5) {\tiny $k$};
	\node at (1,-1.5) {\tiny $1$};
\end{tikzpicture}
};
\endxy
\;\; , \;\;
\xy
(0,0)*{
\begin{tikzpicture}[scale=.3]
	\draw [very thick, green, directed=.55] (0,-1) to (0,.75);
	\draw [very thick, directed=.55] (0,.75) to [out=30,in=270] (1,2.5);
	\draw [very thick, green, directed=.55] (0,.75) to [out=150,in=270] (-1,2.5); 
	\node at (0, -1.5) {\tiny $k{+}1$};
	\node at (-1,3) {\tiny $k$};
	\node at (1,3) {\tiny $1$};
\end{tikzpicture}
};
\endxy
\;\; , \;\;
\xy
(0,0)*{
\begin{tikzpicture}[scale=.3]
	\draw [very thick, green, directed=.55] (0, .75) to (0,2.5);
	\draw [very thick, green, directed=.55] (1,-1) to [out=90,in=330] (0,.75);
	\draw [very thick, directed=.55] (-1,-1) to [out=90,in=210] (0,.75); 
	\node at (0, 3) {\tiny $l{+}1$};
	\node at (-1,-1.5) {\tiny $1$};
	\node at (1,-1.5) {\tiny $l$};
\end{tikzpicture}
};
\endxy
\;\; , \;\;
\xy
(0,0)*{
\begin{tikzpicture}[scale=.3]
	\draw [very thick, green, directed=.55] (0,-1) to (0,.75);
	\draw [very thick, green, directed=.55] (0,.75) to [out=30,in=270] (1,2.5);
	\draw [very thick, directed=.55] (0,.75) to [out=150,in=270] (-1,2.5); 
	\node at (0, -1.5) {\tiny $l{+}1$};
	\node at (-1,3) {\tiny $1$};
	\node at (1,3) {\tiny $l$};
\end{tikzpicture}
};
\endxy
\;\; , \;\;
\xy
(0,0)*{
\begin{tikzpicture}[scale=.3]
	\draw [very thick, mycolor, directed=.55] (0, .75) to (0,2.5);
	\draw [very thick, mycolor, directed=.55] (1,-1) to [out=90,in=330] (0,.75);
	\draw [very thick, directed=.55] (-1,-1) to [out=90,in=210] (0,.75); 
	\node at (0, 3) {\tiny $l{+}1$};
	\node at (-1,-1.5) {\tiny $1$};
	\node at (1,-1.5) {\tiny $l$};
\end{tikzpicture}
};
\endxy
\;\; , \;\;
\xy
(0,0)*{
\begin{tikzpicture}[scale=.3]
	\draw [very thick, mycolor, directed=.55] (0,-1) to (0,.75);
	\draw [very thick, mycolor, directed=.55] (0,.75) to [out=30,in=270] (1,2.5);
	\draw [very thick, directed=.55] (0,.75) to [out=150,in=270] (-1,2.5); 
	\node at (0, -1.5) {\tiny $l{+}1$};
	\node at (-1,3) {\tiny $1$};
	\node at (1,3) {\tiny $l$};
\end{tikzpicture}
};
\endxy
\;\; , \;\;
\xy
(0,0)*{
\begin{tikzpicture}[scale=.3]
	\draw [very thick, mycolor, directed=.55] (0, .75) to (0,2.5);
	\draw [very thick, directed=.55] (1,-1) to [out=90,in=330] (0,.75);
	\draw [very thick, mycolor, directed=.55] (-1,-1) to [out=90,in=210] (0,.75); 
	\node at (0, 3) {\tiny $k{+}1$};
	\node at (-1,-1.5) {\tiny $k$};
	\node at (1,-1.5) {\tiny $1$};
\end{tikzpicture}
};
\endxy
\;\; , \;\;
\xy
(0,0)*{
\begin{tikzpicture}[scale=.3]
	\draw [very thick, mycolor, directed=.55] (0,-1) to (0,.75);
	\draw [very thick, directed=.55] (0,.75) to [out=30,in=270] (1,2.5);
	\draw [very thick, mycolor, directed=.55] (0,.75) to [out=150,in=270] (-1,2.5); 
	\node at (0, -1.5) {\tiny $k{+}1$};
	\node at (-1,3) {\tiny $k$};
	\node at (1,3) {\tiny $1$};
\end{tikzpicture}
};
\endxy
\end{equation}
called \textit{mixed merges} and \textit{mixed splits} respectively. 
(We also include versions of these involving edges labeled $0$ which we, 
as in \eqref{eq-symgen}, do not illustrate.)
\end{itemize}
We call webs 
obtained by composition of generators 
with only black and green edges or only black and red edges \textit{monochromatic}, 
cf. \eqref{eq-morethanonecolor}.
\end{defn}

\begin{rem}\label{rem-conventions}
Note the following conventions and properties of $\InSpf$.
\begin{itemize}

\item The category is $\C_q$-linear, i.e.\ the spaces 
$\Hom_{\InSpf}(\vec{k},\vec{l})$ are $\C_q$-vector spaces and
 the composition $\circ$ is $\C_q$-bilinear. Moreover, the category is monoidal by juxtaposition $\otimes$ of objects and morphisms. $\otimes$ is also $\C_q$-bilinear on morphism spaces.

\item It is sometimes convenient 
in illustrations to allow 
green and red edges with label $1$. 
By convention, these edges are to be read as being black: 
\begin{equation}\label{eq-morethanonecolor}
\xy
(0,0)*{
\begin{tikzpicture} [scale=1]
\draw[very thick, directed=.55] (0,0) to (0,1);
\node at (0,-.15) {\tiny $1$};
\node at (0,1.15) {\tiny $1$};
\end{tikzpicture}
}
\endxy
=
\xy
(0,0)*{
\begin{tikzpicture} [scale=1]
\draw[very thick, green, directed=.55] (0,0) to (0,1);
\node at (0,-.15) {\tiny $1$};
\node at (0,1.15) {\tiny $1$};
\end{tikzpicture}
}
\endxy
=
\xy
(0,0)*{
\begin{tikzpicture} [scale=1]
\draw[very thick, mycolor, directed=.55] (0,0) to (0,1);
\node at (0,-.15) {\tiny $1$};
\node at (0,1.15) {\tiny $1$};
\end{tikzpicture}
}
\endxy
\quad\text{and}\quad
\xy
(0,0)*{
\begin{tikzpicture}[scale=.3]
	\draw [very thick, green, directed=.55] (0, .75) to (0,2.5);
	\draw [very thick, directed=.55] (1,-1) to [out=90,in=330] (0,.75);
	\draw [very thick, directed=.55] (-1,-1) to [out=90,in=210] (0,.75); 
	\node at (0, 3) {\tiny $2$};
	\node at (-1,-1.5) {\tiny $1$};
	\node at (1,-1.5) {\tiny $1$};
\end{tikzpicture}
};
\endxy
\quad , \quad
\xy
(0,0)*{
\begin{tikzpicture}[scale=.3]
	\draw [very thick, green, directed=.55] (0,-1) to (0,.75);
	\draw [very thick, directed=.55] (0,.75) to [out=30,in=270] (1,2.5);
	\draw [very thick, directed=.55] (0,.75) to [out=150,in=270] (-1,2.5); 
	\node at (0, -1.5) {\tiny $2$};
	\node at (-1,3) {\tiny $1$};
	\node at (1,3) {\tiny $1$};
\end{tikzpicture}
};
\endxy
\quad , \quad
\xy
(0,0)*{
\begin{tikzpicture}[scale=.3]
	\draw [very thick, mycolor, directed=.55] (0, .75) to (0,2.5);
	\draw [very thick, directed=.55] (1,-1) to [out=90,in=330] (0,.75);
	\draw [very thick, directed=.55] (-1,-1) to [out=90,in=210] (0,.75); 
	\node at (0, 3) {\tiny $2$};
	\node at (-1,-1.5) {\tiny $1$};
	\node at (1,-1.5) {\tiny $1$};
\end{tikzpicture}
};
\endxy
\quad , \quad
\xy
(0,0)*{
\begin{tikzpicture}[scale=.3]
	\draw [very thick, mycolor, directed=.55] (0,-1) to (0,.75);
	\draw [very thick, directed=.55] (0,.75) to [out=30,in=270] (1,2.5);
	\draw [very thick, directed=.55] (0,.75) to [out=150,in=270] (-1,2.5); 
	\node at (0,-1.5) {\tiny $2$};
	\node at (-1,3) {\tiny $1$};
	\node at (1,3) {\tiny $1$};
\end{tikzpicture}
};
\endxy
\end{equation}
are generators.
For example, the diagrams on the right
are obtained by setting $k=1$ or $l=1$ in \eqref{eq-symgen2}.

\item The reading conventions for all 
webs are from \textit{bottom to top} and \textit{left to right}: 
if $u,v$ are two webs, then 
$v \circ u$ is obtained by gluing $v$ on top of $u$ and 
$u\otimes v$ is given by putting $v$ 
to the right of $u$. Moreover, if any of the top boundary labels of $u$ 
differs from the corresponding bottom boundary label of 
$v$, then, by convention, $v\circ u=0$.

\item For $j\in\Z_{\geq 1}$, we define the 
so-called \textit{monochromatic} $F^{(j)}\one_{(k,l)}$ \textit{and} 
$E^{(j)}\one_{(k,l)}$\textit{-ladders} as
\begin{equation}\label{eq-ladders1}
F^{(j)} \one_{(k,l)}=
\xy
(0,0)*{
\begin{tikzpicture}[scale=.3]
	\draw [very thick, green, directed=.65] (1,.75) to (1,2);
	\draw [very thick, green] (2,-1) to [out=90,in=320] (1,.75);
	\draw [very thick, green, rdirected=.1] (2,-1) to (2,-2);
	\draw [very thick, green, directed=.75] (-1,-2) to (-1,-.25);
	\draw [very thick, green] (-1,-0.25) to [out=150,in=270] (-2,1.5);
	\draw [very thick, green, directed=.01] (-2,1.5) to (-2,2);
	\draw [very thick, green, directed=.55] (-1,-.25) to (1,.75);
	\node at (-1,-2.5) {\tiny $k$};
	\node at (2,-2.5) {\tiny $l$};
	\node at (-2,2.5) {\tiny $k{-}j$};
	\node at (1,2.5) {\tiny $l{+}j$};
	\node at (0,.85) {\tiny $j$};
\end{tikzpicture}
};
\endxy\quad,\quad
E^{(j)}\one_{(k,l)}=
\xy
(0,0)*{\reflectbox{
\begin{tikzpicture}[scale=.3]
	\draw [very thick, green, directed=.65] (1,.75) to (1,2);
	\draw [very thick, green] (2,-1) to [out=90,in=320] (1,.75);
	\draw [very thick, green, rdirected=.1] (2,-1) to (2,-2);
	\draw [very thick, green, directed=.75] (-1,-2) to (-1,-.25);
	\draw [very thick, green] (-1,-0.25) to [out=150,in=270] (-2,1.5);
	\draw [very thick, green, directed=.01] (-2,1.5) to (-2,2);
	\draw [very thick, green, directed=.55] (-1,-.25) to (1,.75);
	\node at (-1,-2.5) {\reflectbox{\tiny $l$}};
	\node at (2,-2.5) {\reflectbox{\tiny $k$}};
	\node at (-2,2.5) {\reflectbox{\tiny $l{-}j$}};
	\node at (1,2.5) {\reflectbox{\tiny $k{+}j$}};
	\node at (-0.25,.85) {\reflectbox{\tiny $j$}};
\end{tikzpicture}
}};
\endxy
\end{equation}
and analogously in red. (The notation $\one_{(k,l)}$ 
is motivated by the ``dual side'' as we will see in \fullref{sub-superHowe}. 
For the green--red web calculus it is just
a shorthand to indicated the underlying objects.) Sometimes we 
draw such ladder rungs horizontally. We also have the 
\textit{mixed} $F\one_{(k,l)}$ \textit{and} $E\one_{(k,l)}$\textit{-ladders}
\begin{equation}\label{eq-ladders3}
F\one_{(k,l)}=
\xy
(0,0)*{
\begin{tikzpicture}[scale=.3]
	\draw [very thick, mycolor, directed=.65] (1,.75) to (1,2);
	\draw [very thick, mycolor] (2,-1) to [out=90,in=320] (1,.75);
	\draw [very thick, mycolor, rdirected=.1] (2,-1) to (2,-2);
	\draw [very thick, green, directed=.75] (-1,-2) to (-1,-.25);
	\draw [very thick, green] (-1,-0.25) to [out=150,in=270] (-2,1.5);
	\draw [very thick, green, directed=.01] (-2,1.5) to (-2,2);
	\draw [very thick, directed=.55] (-1,-.25) to (1,.75);
	\node at (-1,-2.5) {\tiny $k$};
	\node at (2,-2.5) {\tiny $l$};
	\node at (-2,2.5) {\tiny $k{-}1$};
	\node at (1,2.5) {\tiny $l{+}1$};
	\node at (0,.75) {\tiny $1$};
\end{tikzpicture}
};
\endxy\quad,\quad
E\one_{(k,l)}=
\xy
(0,0)*{\reflectbox{
\begin{tikzpicture}[scale=.3]
	\draw [very thick, green, directed=.65] (1,.75) to (1,2);
	\draw [very thick, green] (2,-1) to [out=90,in=320] (1,.75);
	\draw [very thick, green, rdirected=.1] (2,-1) to (2,-2);
	\draw [very thick, mycolor, directed=.75] (-1,-2) to (-1,-.25);
	\draw [very thick, mycolor] (-1,-0.25) to [out=150,in=270] (-2,1.5);
	\draw [very thick, mycolor, directed=.01] (-2,1.5) to (-2,2);
	\draw [very thick, directed=.55] (-1,-.25) to (1,.75);
	\node at (-1,-2.5) {\reflectbox{\tiny $l$}};
	\node at (2,-2.5) {\reflectbox{\tiny $k$}};
	\node at (-2,2.5) {\reflectbox{\tiny $l{-}1$}};
	\node at (1,2.5) {\reflectbox{\tiny $k{+}1$}};
	\node at (-0.25,.8) {\reflectbox{\tiny $1$}};
\end{tikzpicture}
}};
\endxy
\end{equation}
and similarly by exchanging green and red. 
Note that the 
ladders from \eqref{eq-ladders1} 
exist for all $j\in\Z_{\geq 1}$, while the mixed 
ladders from \eqref{eq-ladders3} exist only for $j=1$.

\item We usually omit the
object $0$ as well as edges labeled $0$ from illustrations, cf. \eqref{eq-symgen}.

\end{itemize}
\end{rem}

\begin{defn}\label{defn-spid}
The \textit{green--red web category} $\InSp$ 
is the quotient of $\InSpf$ obtained by 
imposing the following local relations on morphisms.
The \textit{monochromatic relations}, which hold for green webs as well as for red webs: 
\textit{(co)associativity}
\begin{equation}\label{eq-frob}
\xy
(0,0)*{
\begin{tikzpicture}[scale=.3]
	\draw [very thick, green, directed=.55] (0,.75) to [out=90,in=220] (1,2.5);
	\draw [very thick, green, directed=.55] (1,-1) to [out=90,in=330] (0,.75);
	\draw [very thick, green, directed=.55] (-1,-1) to [out=90,in=210] (0,.75);
	\draw [very thick, green, directed=.55] (3,-1) to [out=90,in=330] (1,2.5);
	\draw [very thick, green, directed=.55] (1,2.5) to (1,4.25);
	\node at (-1,-1.5) {\tiny $h$};
	\node at (1,-1.5) {\tiny $k$};
	\node at (-1.325,1.5) {\tiny $h{+}k$};
	\node at (3,-1.5) {\tiny $l$};
	\node at (1,4.75) {\tiny $h{+}k{+}l$};
\end{tikzpicture}
};
\endxy=
\xy,
(0,0)*{\reflectbox{
\begin{tikzpicture}[scale=.3]
	\draw [very thick, green, directed=.55] (0,.75) to [out=90,in=220] (1,2.5);
	\draw [very thick, green, directed=.55] (1,-1) to [out=90,in=330] (0,.75);
	\draw [very thick, green, directed=.55] (-1,-1) to [out=90,in=210] (0,.75);
	\draw [very thick, green, directed=.55] (3,-1) to [out=90,in=330] (1,2.5);
	\draw [very thick, green, directed=.55] (1,2.5) to (1,4.25);
	\node at (-1,-1.5) {\tiny \reflectbox{$l$}};
	\node at (1,-1.5) {\tiny \reflectbox{$k$}};
	\node at (-1.325,1.5) {\tiny \reflectbox{$k{+}l$}};
	\node at (3,-1.5) {\tiny \reflectbox{$h$}};
	\node at (1,4.75) {\tiny \reflectbox{$h{+}k{+}l$}};
\end{tikzpicture}
}};
\endxy\quad,\quad
\xy
(0,0)*{\rotatebox{180}{
\begin{tikzpicture}[scale=.3]
	\draw [very thick, green, rdirected=.55] (0,.75) to [out=90,in=220] (1,2.5);
	\draw [very thick, green, rdirected=.55] (1,-1) to [out=90,in=330] (0,.75);
	\draw [very thick, green, rdirected=.55] (-1,-1) to [out=90,in=210] (0,.75);
	\draw [very thick, green, rdirected=.55] (3,-1) to [out=90,in=330] (1,2.5);
	\draw [very thick, green, rdirected=.55] (1,2.5) to (1,4.25);
	\node at (-1,-1.5) {\rotatebox{180}{\tiny $l$}};
	\node at (1,-1.5) {\rotatebox{180}{\tiny $k$}};
	\node at (-1.325,1.5) {\rotatebox{180}{\tiny $k{+}l$}};
	\node at (3,-1.5) {\rotatebox{180}{\tiny $h$}};
	\node at (1,4.75) {\rotatebox{180}{\tiny $h{+}k{+}l$}};
\end{tikzpicture}
}};
\endxy=
\xy
(0,0)*{\reflectbox{\rotatebox{180}{
\begin{tikzpicture}[scale=.3]
	\draw [very thick, green, rdirected=.55] (0,.75) to [out=90,in=220] (1,2.5);
	\draw [very thick, green, rdirected=.55] (1,-1) to [out=90,in=330] (0,.75);
	\draw [very thick, green, rdirected=.55] (-1,-1) to [out=90,in=210] (0,.75);
	\draw [very thick, green, rdirected=.55] (3,-1) to [out=90,in=330] (1,2.5);
	\draw [very thick, green, rdirected=.55] (1,2.5) to (1,4.25);
	\node at (-1,-1.5) {\reflectbox{\rotatebox{180}{\tiny $h$}}};
	\node at (1,-1.5) {\reflectbox{\rotatebox{180}{\tiny $k$}}};
	\node at (-1.325,1.5) {\reflectbox{\rotatebox{180}{\tiny $h{+}k$}}};
	\node at (3,-1.5) {\reflectbox{\rotatebox{180}{\tiny $l$}}};
	\node at (1,4.75) {\reflectbox{\rotatebox{180}{\tiny $h{+}k{+}l$}}};
\end{tikzpicture}
}}};
\endxy
\end{equation} 
where we use the shorthand 
notation from \eqref{eq-morethanonecolor} 
if some of the labels are $1$. 
Next, the \textit{digon removal relations}
\begin{equation}\label{eq-simpler1}
\xy
(0,0)*{
\begin{tikzpicture}[scale=.3]
	\draw [very thick, green, directed=.55] (0,.75) to (0,2.5);
	\draw [very thick, green, directed=.55] (0,-2.75) to [out=30,in=330] (0,.75);
	\draw [very thick, green, directed=.55] (0,-2.75) to [out=150,in=210] (0,.75);
	\draw [very thick, green, directed=.55] (0,-4.5) to (0,-2.75);
	\node at (0,-5) {\tiny $k{+}l$};
	\node at (0,3) {\tiny $k{+}l$};
	\node at (-1.5,-1) {\tiny $k$};
	\node at (1.5,-1) {\tiny $l$};
\end{tikzpicture}
};
\endxy=\qbinn{k+l}{l}
\xy
(0,0)*{
\begin{tikzpicture}[scale=.3]
	\draw [very thick, green, directed=.55] (0,-4.5) to (0,2.5);
	\node at (0,-5) {\tiny $k{+}l$};
	\node at (0,3) {\tiny $k{+}l$};
\end{tikzpicture}
};
\endxy
\end{equation}
for which $k$ and $l$ might be $1$. 
In these relations the 
$s,t$\textit{-quantum binomial} is given by 
\[
\qbinn{s}{t}=\frac{[s][s-1]\cdots[s-t+2][s-t+1]}{[t]!}\in\C_q.
\]
Here $[s]=\frac{q^s-q^{-s}}{q-q^{-1}}\in\C_q$ is the 
\textit{quantum number} and $[t]!=[1][2]\cdots [t]\in\C_q$ is the 
\textit{quantum factorial} for $s\in\Z$ and $t\in\Z_{\geq 0}$. 
Finally, the \textit{square switch} relations
\begin{equation}\label{eq-almostsimpler}
\xy
(0,0)*{
\begin{tikzpicture}[scale=.3]
	\draw [very thick, green, directed=.55] (-2,-4) to (-2,-2);
	\draw [very thick, green, directed=.55] (-2,-2) to (-2,2);
	\draw [very thick, green, directed=.55] (2,-4) to (2,-2);
	\draw [very thick, green, directed=.55] (2,-2) to (2,2);
	\draw [very thick, green, directed=.55] (-2,-2) to (2,-2);
	\draw [very thick, green, directed=.55] (-2,2) to (-2,4);
	\draw [very thick, green, directed=.55] (2,2) to (2,4);
	\draw [very thick, green, rdirected=.55] (-2,2) to (2,2);
	\node at (-2,-4.5) {\tiny $k$};
	\node at (2,-4.5) {\tiny $l$};
	\node at (-2.25,4.5) {\tiny $k{-}j_1{+}j_2$};
	\node at (2.25,4.5) {\tiny $l{+}j_1{-}j_2$};
	\node at (-3.5,0) {\tiny $k{-}j_1$};
	\node at (3.5,0) {\tiny $l{+}j_1$};
	\node at (0,-1.25) {\tiny $j_1$};
	\node at (0,2.75) {\tiny $j_2$};
\end{tikzpicture}
};
\endxy=\sum_{j^{\prime}\geq 0}\qbinn{k-j_1-l+j_2}{j^{\prime}}\xy
(0,0)*{
\begin{tikzpicture}[scale=.3]
	\draw [very thick, green, directed=.55] (-2,-4) to (-2,-2);
	\draw [very thick, green, directed=.55] (-2,-2) to (-2,2);
	\draw [very thick, green, directed=.55] (2,-4) to (2,-2);
	\draw [very thick, green, directed=.55] (2,-2) to (2,2);
	\draw [very thick, green, rdirected=.55] (-2,-2) to (2,-2);
	\draw [very thick, green, directed=.55] (-2,2) to (-2,4);
	\draw [very thick, green, directed=.55] (2,2) to (2,4);
	\draw [very thick, green, directed=.55] (-2,2) to (2,2);
	\node at (-2,-4.5) {\tiny $k$};
	\node at (2,-4.5) {\tiny $l$};
	\node at (-2.25,4.5) {\tiny $k{-}j_1{+}j_2$};
	\node at (2.25,4.5) {\tiny $l{+}j_1{-}j_2$};
	\node at (-4.5,0) {\tiny $k{+}j_2{-}j^{\prime}$};
	\node at (4.5,0) {\tiny $l{-}j_2{+}j^{\prime}$};
	\node at (0,-1.25) {\tiny $j_2{-}j^{\prime}$};
	\node at (0,2.75) {\tiny $j_1{-}j^{\prime}$};
\end{tikzpicture}
};
\endxy
\end{equation}
Here we allow $j_1$ or $j_2$ to be $1$ (we will get 
\textit{mixed} square switch relations, with 
one green and one red side, in \fullref{lem-secondcq}).

To write these relations in a uniform manner, 
we allow negative labels on edges and set webs with such edges equal to zero.

The defining relation between green and red edges is
\begin{equation}\label{eq-dumb}
[2]
\xy
(0,0)*{
\begin{tikzpicture}[scale=.3] 
	\draw [very thick, directed=.55] (1,-2.75) to (1,2.5);
	\draw [very thick, directed=.55] (-1,-2.75) to (-1,2.5);
	\node at (-1,3) {\tiny $1$};
	\node at (1,3) {\tiny $1$};
	\node at (-1,-3.15) {\tiny $1$};
	\node at (1,-3.15) {\tiny $1$};
\end{tikzpicture}
};
\endxy
=
\xy
(0,0)*{
\begin{tikzpicture}[scale=.3]
	\draw [very thick, green, directed=.55] (0,-1) to (0,.75);
	\draw [very thick, directed=.55] (0,.75) to [out=30,in=270] (1,2.5);
	\draw [very thick, directed=.55] (0,.75) to [out=150,in=270] (-1,2.5); 
	\draw [very thick, directed=.55] (1,-2.75) to [out=90,in=330] (0,-1);
	\draw [very thick, directed=.55] (-1,-2.75) to [out=90,in=210] (0,-1);
	\node at (-1,3) {\tiny $1$};
	\node at (1,3) {\tiny $1$};
	\node at (-1,-3.15) {\tiny $1$};
	\node at (1,-3.15) {\tiny $1$};
	\node at (-0.5,0) {\tiny $2$};
\end{tikzpicture}
};
\endxy
+
\xy
(0,0)*{
\begin{tikzpicture}[scale=.3]
	\draw [very thick, mycolor, directed=.55] (0,-1) to (0,.75);
	\draw [very thick, directed=.55] (0,.75) to [out=30,in=270] (1,2.5);
	\draw [very thick, directed=.55] (0,.75) to [out=150,in=270] (-1,2.5); 
	\draw [very thick, directed=.55] (1,-2.75) to [out=90,in=330] (0,-1);
	\draw [very thick, directed=.55] (-1,-2.75) to [out=90,in=210] (0,-1);
	\node at (-1,3) {\tiny $1$};
	\node at (1,3) {\tiny $1$};
	\node at (-1,-3.15) {\tiny $1$};
	\node at (1,-3.15) {\tiny $1$};
	\node at (-0.5,0) {\tiny $2$};
\end{tikzpicture}
};
\endxy
\end{equation}
which we call the \textit{dumbbell relation}.
\end{defn}

\begin{rem}\label{rem-bysymmetry}
The category $\InSp$ is symmetric under exchanging 
green and red. In the following we will often refer to this symmetry to shorten arguments.
\end{rem}

\begin{defn}\label{defn-spid2}
The category $\SuSp$ 
is the quotient category obtained from 
the category $\InSp$ by imposing the 
\textit{exterior relations}, that is,
\begin{equation}\label{eq-alt}
\xy
(0,0)*{
\begin{tikzpicture} [scale=1]
\draw[very thick, green, directed=.55] (0,0) to (0,1);
\node at (0.2,.5) {\tiny $k$};
\end{tikzpicture}
}
\endxy=0\quad,\quad\text{if }k>N.
\end{equation}
The exterior relations hold only for green edges. 
These relations mean that any web $u$ with a green edge 
labeled $k>N$ is zero. In 
contrast, red edges labeled $k>N$ are usually not zero.

The sorted web category $\SoSuSp$ is the full 
(non-monoidal) subcategory of $\SuSp$ whose object set consists of 
$\vec{k}\in X^L$ with no red boundary 
point left of a green boundary point: if $k_i\in X_r$ for some $i$, then $k_{>i}\in X_b\cup X_r$.
\end{defn}

\begin{rem}\label{rem-altsym}
The relations \eqref{eq-alt} are diagrammatic 
versions of
$
\bV_q^{>N}\C_q^N\cong 0.
$
\end{rem}

\begin{defn}\label{defn-spid3} 
The category $\AltSp$ is the subcategory of $\SuSp$ consisting of only 
black and green objects and whose morphism spaces are spanned 
as $\C_q$-vector spaces 
by webs that contain only black or green edges.

Similarly, the category $\SymSp$ 
is the subcategory of $\SuSp$ consisting of only 
black and red objects and whose morphism spaces are spanned 
as $\C_q$-vector spaces 
by webs that contain only black or red edges.

We call these categories \textit{monochromatic}.
\end{defn}

\begin{rem}\label{rem-ckmwebs}
We will see in \fullref{cor-green2} that $\AltSp$ is equivalent to the web category given in \cite[Definition 2.2]{ckm} 
(without tags and downward-pointing arrows). 
The category $\SymSp$ is a generalization of the 
one given in \cite[Definition 1.4]{rt}. In fact, 
\fullref{prop-green} shows that both monochromatic subcategories are full in $\SuSp$.
\end{rem}

\subsection{The diagrammatic super relations}\label{sub-lemmata}

We show in this subsection that 
diagrammatic 
versions of the relations \eqref{eq-super1} 
in the Howe dual quantum group $\Usd$ 
from \fullref{defn-glmn} hold in 
our diagrammatic categories $\InSp$ and $\SuSp$.

\begin{lem}\label{lem-green}
We have the relations
\[
\xy
(0,0)*{
\begin{tikzpicture}[scale=.3]
	\draw [very thick, green, directed=.55] (0,.75) to (0,2.5);
	\draw [very thick, directed=.55] (0,-2.75) to [out=30,in=330] (0,.75);
	\draw [very thick, directed=.55] (0,-2.75) to [out=150,in=210] (0,.75);
	\draw [very thick, mycolor, directed=.55] (0,-4.5) to (0,-2.75);
	\node at (0,-5) {\tiny $k$};
	\node at (0,3) {\tiny $k$};
	\node at (-1.5,-1) {\tiny $1$};
	\node at (1.5,-1) {\tiny $1$};
	\node at (0.1,-1) {\tiny $\cdots$};
\end{tikzpicture}
};
\endxy
=0=
\xy
(0,0)*{
\begin{tikzpicture}[scale=.3]
	\draw [very thick, mycolor, directed=.55] (0,.75) to (0,2.5);
	\draw [very thick, directed=.55] (0,-2.75) to [out=30,in=330] (0,.75);
	\draw [very thick, directed=.55] (0,-2.75) to [out=150,in=210] (0,.75);
	\draw [very thick, green, directed=.55] (0,-4.5) to (0,-2.75);
	\node at (0,-5) {\tiny $k$};
	\node at (0,3) {\tiny $k$};
	\node at (-1.5,-1) {\tiny $1$};
	\node at (1.5,-1) {\tiny $1$};
	\node at (0.1,-1) {\tiny $\cdots$};
\end{tikzpicture}
};
\endxy
\]
where the dots indicate $k$ parallel black edges with label $1$ which split off 
the bottom and merge with the top in any order (the order does not matter 
because of \eqref{eq-frob}).
\end{lem}

\begin{proof}
It suffices by associativity \eqref{eq-frob} to show the statement for $k=2$. We have
\begin{align*}
\xy
(0,0)*{
\begin{tikzpicture}[scale=.3]
	\draw [very thick, green, directed=.55] (0,.75) to (0,8);
	\draw [very thick, directed=.55] (0,-2.75) to [out=30,in=330] (0,.75);
	\draw [very thick, directed=.55] (0,-2.75) to [out=150,in=210] (0,.75);
	\draw [very thick, mycolor, directed=.55] (0,-4.5) to (0,-2.75);
	\node at (0,-5) {\tiny $2$};
	\node at (0,8.5) {\tiny $2$};
	\node at (-1.5,-1) {\tiny $1$};
	\node at (1.5,-1) {\tiny $1$};
\end{tikzpicture}
};
\endxy
\stackrel{\eqref{eq-simpler1}}{=} 
\frac{1}{[2]}\xy
(0,0)*{
\begin{tikzpicture}[scale=.3]
	\draw [very thick, green, directed=.55] (0,.75) to (0,2.5);
	\draw [very thick, directed=.55] (0,-2.75) to [out=30,in=330] (0,.75);
	\draw [very thick, directed=.55] (0,-2.75) to [out=150,in=210] (0,.75);
	\draw [very thick, mycolor, directed=.55] (0,-4.5) to (0,-2.75);
	\draw [very thick, green, directed=.55] (0,6) to (0,8);
	\draw [very thick, directed=.55] (0,2.5) to [out=30,in=330] (0,6);
	\draw [very thick, directed=.55] (0,2.5) to [out=150,in=210] (0,6);
	\node at (0,-5) {\tiny $2$};
	\node at (0,8.5) {\tiny $2$};
	\node at (-1.5,-1) {\tiny $1$};
	\node at (1.5,-1) {\tiny $1$};
	\node at (-1.5,4) {\tiny $1$};
	\node at (1.5,4) {\tiny $1$};
\end{tikzpicture}
};
\endxy
\stackrel{\eqref{eq-dumb}}{=} 
\xy
(0,0)*{
\begin{tikzpicture}[scale=.3]
	\draw [very thick, directed=.55] (.9,-1) to (.9,4.25);
	\draw [very thick, directed=.55] (-.9,-1) to (-.9,4.25);
	\draw [very thick] (0,-2.75) to [out=30,in=270] (.9,-1);
	\draw [very thick] (0,-2.75) to [out=150,in=270] (-.9,-1);
	\draw [very thick, mycolor, directed=.55] (0,-4.5) to (0,-2.75);
	\draw [very thick, green, directed=.55] (0,6) to (0,8);
	\draw [very thick] (.9,4.25) to [out=90,in=330] (0,6);
	\draw [very thick] (-.9,4.25) to [out=90,in=210] (0,6);
	\node at (0,-5) {\tiny $2$};
	\node at (0,8.5) {\tiny $2$};
	\node at (-1.5,1.75) {\tiny $1$};
	\node at (1.5,1.75) {\tiny $1$};
\end{tikzpicture}
};
\endxy
-
\frac{1}{[2]}\xy
(0,0)*{
\begin{tikzpicture}[scale=.3]
	\draw [very thick, mycolor, directed=.55] (0,.75) to (0,2.5);
	\draw [very thick, directed=.55] (0,-2.75) to [out=30,in=330] (0,.75);
	\draw [very thick, directed=.55] (0,-2.75) to [out=150,in=210] (0,.75);
	\draw [very thick, mycolor, directed=.55] (0,-4.5) to (0,-2.75);
	\draw [very thick, green, directed=.55] (0,6) to (0,8);
	\draw [very thick, directed=.55] (0,2.5) to [out=30,in=330] (0,6);
	\draw [very thick, directed=.55] (0,2.5) to [out=150,in=210] (0,6);
	\node at (0,-5) {\tiny $2$};
	\node at (0,8.5) {\tiny $2$};
	\node at (-1.5,-1) {\tiny $1$};
	\node at (1.5,-1) {\tiny $1$};
	\node at (-1.5,4) {\tiny $1$};
	\node at (1.5,4) {\tiny $1$};
\end{tikzpicture}
};
\endxy
\stackrel{\eqref{eq-simpler1}}{=} 
\xy
(0,0)*{
\begin{tikzpicture}[scale=.3]
	\draw [very thick, directed=.55] (.9,-1) to (.9,4.25);
	\draw [very thick, directed=.55] (-.9,-1) to (-.9,4.25);
	\draw [very thick] (0,-2.75) to [out=30,in=270] (.9,-1);
	\draw [very thick] (0,-2.75) to [out=150,in=270] (-.9,-1);
	\draw [very thick, mycolor, directed=.55] (0,-4.5) to (0,-2.75);
	\draw [very thick, green, directed=.55] (0,6) to (0,8);
	\draw [very thick] (.9,4.25) to [out=90,in=330] (0,6);
	\draw [very thick] (-.9,4.25) to [out=90,in=210] (0,6);
	\node at (0,-5) {\tiny $2$};
	\node at (0,8.5) {\tiny $2$};
	\node at (-1.5,1.75) {\tiny $1$};
	\node at (1.5,1.75) {\tiny $1$};
\end{tikzpicture}
};
\endxy -
\xy
(0,0)*{
\begin{tikzpicture}[scale=.3]
	\draw [very thick, mycolor, directed=.55] (0,-4.5) to (0,2.5);
	\draw [very thick, green, directed=.55] (0,6) to (0,8);
	\draw [very thick, directed=.55] (0,2.5) to [out=30,in=330] (0,6);
	\draw [very thick, directed=.55] (0,2.5) to [out=150,in=210] (0,6);
	\node at (0,-5) {\tiny $2$};
	\node at (0,8.5) {\tiny $2$};
	\node at (-1.5,4) {\tiny $1$};
	\node at (1.5,4) {\tiny $1$};
\end{tikzpicture}
};
\endxy
=0.
\end{align*}
The other $k=2$ relation follows by symmetry.
\end{proof}

\begin{lem}\label{lem-secondcq}
\begin{itemize}
\item[(a)] We have for all $k,l\in\Z_{\geq 0}$
\[
\xy
(0,0)*{
\begin{tikzpicture}[scale=.3]
	\draw [very thick, green, directed=.55] (-2,-4) to (-2,-2);
	\draw [very thick, green, directed=.55] (-2,-2) to (-2,2);
	\draw [very thick, mycolor, directed=.55] (2,-4) to (2,-2);
	\draw [very thick, mycolor, directed=.55] (2,-2) to (2,2);
	\draw [very thick, directed=.55] (-2,-2) to (2,-2);
	\draw [very thick, green, directed=.55] (-2,2) to (-2,4);
	\draw [very thick, mycolor, directed=.55] (2,2) to (2,4);
	\draw [very thick, directed=.55] (-2,2) to (2,2);
	\node at (-2,-4.5) {\tiny $k$};
	\node at (2,-4.5) {\tiny $l$};
	\node at (-2.25,4.5) {\tiny $k{-}2$};
	\node at (2.25,4.5) {\tiny $l{+}2$};
	\node at (-3.5,0) {\tiny $k{-}1$};
	\node at (3.5,0) {\tiny $l{+}1$};
	\node at (0,-1.25) {\tiny $1$};
	\node at (0,2.75) {\tiny $1$};
\end{tikzpicture}
};
\endxy
\;
=\;
\xy
(0,0)*{
\begin{tikzpicture}[scale=.3]
	\draw [very thick, mycolor, directed=.55] (-2,-4) to (-2,-2);
	\draw [very thick, mycolor, directed=.55] (-2,-2) to (-2,2);
	\draw [very thick, green, directed=.55] (2,-4) to (2,-2);
	\draw [very thick, green, directed=.55] (2,-2) to (2,2);
	\draw [very thick, directed=.55] (-2,-2) to (2,-2);
	\draw [very thick, mycolor, directed=.55] (-2,2) to (-2,4);
	\draw [very thick, green, directed=.55] (2,2) to (2,4);
	\draw [very thick, directed=.55] (-2,2) to (2,2);
	\node at (-2,-4.5) {\tiny $k$};
	\node at (2,-4.5) {\tiny $l$};
	\node at (-2.25,4.5) {\tiny $k{-}2$};
	\node at (2.25,4.5) {\tiny $l{+}2$};
	\node at (-3.5,0) {\tiny $k{-}1$};
	\node at (3.5,0) {\tiny $l{+}1$};
	\node at (0,-1.25) {\tiny $1$};
	\node at (0,2.75) {\tiny $1$};
\end{tikzpicture}
};
\endxy
\;=\; 0\;=\;
\xy
(0,0)*{
\begin{tikzpicture}[scale=.3]
	\draw [very thick, mycolor, directed=.55] (-2,-4) to (-2,-2);
	\draw [very thick, mycolor, directed=.55] (-2,-2) to (-2,2);
	\draw [very thick, green, directed=.55] (2,-4) to (2,-2);
	\draw [very thick, green, directed=.55] (2,-2) to (2,2);
	\draw [very thick, rdirected=.55] (-2,-2) to (2,-2);
	\draw [very thick, mycolor, directed=.55] (-2,2) to (-2,4);
	\draw [very thick, green, directed=.55] (2,2) to (2,4);
	\draw [very thick, rdirected=.55] (-2,2) to (2,2);
	\node at (-2,-4.5) {\tiny $k$};
	\node at (2,-4.5) {\tiny $l$};
	\node at (-2.25,4.5) {\tiny $k{+}2$};
	\node at (2.25,4.5) {\tiny $l{-}2$};
	\node at (-3.5,0) {\tiny $k{+}1$};
	\node at (3.5,0) {\tiny $l{-}1$};
	\node at (0,-1.25) {\tiny $1$};
	\node at (0,2.75) {\tiny $1$};
\end{tikzpicture}
};
\endxy
\;
=
\;
\xy
(0,0)*{
\begin{tikzpicture}[scale=.3]
	\draw [very thick, green, directed=.55] (-2,-4) to (-2,-2);
	\draw [very thick, green, directed=.55] (-2,-2) to (-2,2);
	\draw [very thick, mycolor, directed=.55] (2,-4) to (2,-2);
	\draw [very thick, mycolor, directed=.55] (2,-2) to (2,2);
	\draw [very thick, rdirected=.55] (-2,-2) to (2,-2);
	\draw [very thick, green, directed=.55] (-2,2) to (-2,4);
	\draw [very thick, mycolor, directed=.55] (2,2) to (2,4);
	\draw [very thick, rdirected=.55] (-2,2) to (2,2);
	\node at (-2,-4.5) {\tiny $k$};
	\node at (2,-4.5) {\tiny $l$};
	\node at (-2.25,4.5) {\tiny $k{+}2$};
	\node at (2.25,4.5) {\tiny $l{-}2$};
	\node at (-3.5,0) {\tiny $k{+}1$};
	\node at (3.5,0) {\tiny $l{-}1$};
	\node at (0,-1.25) {\tiny $1$};
	\node at (0,2.75) {\tiny $1$};
\end{tikzpicture}
};
\endxy
\]

\item[(b)] We have for all $k,l\in\Z_{\geq 0}$
\[
[k+l]
\xy
(0,0)*{
\begin{tikzpicture}[scale=.3]
	\draw [very thick, green, directed=.55] (-2,-4) to (-2,4);
	\draw [very thick, mycolor, directed=.55] (2,-4) to (2,4);
	\node at (-2,-4.5) {\tiny $k$};
	\node at (2,-4.5) {\tiny $l$};
	\node at (-2,4.5) {\tiny $k$};
	\node at (2,4.5) {\tiny $l$};
\end{tikzpicture}
};
\endxy
\;
=
\;
\xy
(0,0)*{
\begin{tikzpicture}[scale=.3]
	\draw [very thick, green, directed=.55] (-2,-4) to (-2,-2);
	\draw [very thick, green, directed=.55] (-2,-2) to (-2,2);
	\draw [very thick, mycolor, directed=.55] (2,-4) to (2,-2);
	\draw [very thick, mycolor, directed=.55] (2,-2) to (2,2);
	\draw [very thick, rdirected=.55] (-2,-2) to (2,-2);
	\draw [very thick, green, directed=.55] (-2,2) to (-2,4);
	\draw [very thick, mycolor, directed=.55] (2,2) to (2,4);
	\draw [very thick, directed=.55] (-2,2) to (2,2);
	\node at (-2,-4.5) {\tiny $k$};
	\node at (2,-4.5) {\tiny $l$};
	\node at (-2,4.5) {\tiny $k$};
	\node at (2,4.5) {\tiny $l$};
	\node at (-3.5,0) {\tiny $k{+}1$};
	\node at (3.5,0) {\tiny $l{-}1$};
	\node at (0,-1.25) {\tiny $1$};
	\node at (0,2.75) {\tiny $1$};
\end{tikzpicture}
};
\endxy
\;
+
\;
\xy
(0,0)*{
\begin{tikzpicture}[scale=.3]
	\draw [very thick, green, directed=.55] (-2,-4) to (-2,-2);
	\draw [very thick, green, directed=.55] (-2,-2) to (-2,2);
	\draw [very thick, mycolor, directed=.55] (2,-4) to (2,-2);
	\draw [very thick, mycolor, directed=.55] (2,-2) to (2,2);
	\draw [very thick, directed=.55] (-2,-2) to (2,-2);
	\draw [very thick, green, directed=.55] (-2,2) to (-2,4);
	\draw [very thick, mycolor, directed=.55] (2,2) to (2,4);
	\draw [very thick, rdirected=.55] (-2,2) to (2,2);
	\node at (-2,-4.5) {\tiny $k$};
	\node at (2,-4.5) {\tiny $l$};
	\node at (-2,4.5) {\tiny $k$};
	\node at (2,4.5) {\tiny $l$};
	\node at (-3.5,0) {\tiny $k{-}1$};
	\node at (3.5,0) {\tiny $l{+}1$};
	\node at (0,-1.25) {\tiny $1$};
	\node at (0,2.75) {\tiny $1$};
\end{tikzpicture}
};
\endxy
\]
and similarly for exchanged roles of green and red.

\item[(c)] We have for all $k,l\in\Z_{\geq 0}$
\begin{align*}
\raisebox{-.1cm}{$[2]
\xy
(0,0)*{
\begin{tikzpicture}[scale=.3]
	\draw [very thick, green, directed=.25, directed=.75] (-4.5,-2) to (-4.5,6);
	\draw [very thick, green, directed=.125, directed=.375, directed=.625, directed=.875] (-1.5,-2) to (-1.5,6);
	\draw [very thick, mycolor, directed=.125, directed=.375, directed=.625, directed=.875] (1.5,-2) to (1.5,6);
	\draw [very thick, mycolor, directed=.25, directed=.75] (4.5,-2) to (4.5,6);
	\draw [very thick, directed=.55] (-1.5,0) to (1.5,0);
	\draw [very thick, green, directed=.55] (-4.5,2) to (-1.5,2);
	\draw [very thick, mycolor, directed=.55] (1.5,2) to (4.5,2);
	\draw [very thick, directed=.55] (-1.5,4) to (1.5,4);
	\node at (-4.5,-2.5) {\tiny $k_1$};
	\node at (-1.5,-2.5) {\tiny $k_2$};
	\node at (1.5,-2.5) {\tiny $k_3$};
	\node at (4.5,-2.5) {\tiny $k_4$};
	\node at (0,0.75) {\tiny $1$};
	\node at (0,4.75) {\tiny $1$};
	\node at (-3,2.75) {\tiny $1$};
	\node at (3,2.75) {\tiny $1$};
	\node at (-4.5,6.5) {\tiny $k_1{-}1$};
	\node at (-1.5,6.5) {\tiny $k_2{-}1$};
	\node at (1.5,6.5) {\tiny $k_3{+}1$};
	\node at (4.5,6.5) {\tiny $k_4{+}1$};
\end{tikzpicture}
};
\endxy$} =&
\xy
(0,0)*{
\begin{tikzpicture}[scale=.3]
	\draw [very thick, green, directed=.125, directed=.875] (-4.5,-2) to (-4.5,6);
	\draw [very thick, green, directed=.125, directed=.365, directed=.625, directed=.875] (-1.5,-2) to (-1.5,6);
	\draw [very thick, mycolor, directed=.125, directed=.51, directed=.7, directed=.875] (1.5,-2) to (1.5,6);
	\draw [very thick, mycolor, directed=.125, directed=.875] (4.5,-2) to (4.5,6);
	\draw [very thick, directed=.55] (-1.5,1.333) to (1.5,1.333);
	\draw [very thick, mycolor, directed=.55] (1.5,2.666) to (4.5,2.666);
	\draw [very thick, green, directed=.55] (-4.5,0) to (-1.5,0);
	\draw [very thick, directed=.55] (-1.5,4) to (1.5,4);
	\node at (-4.5,-2.5) {\tiny $k_1$};
	\node at (-1.5,-2.5) {\tiny $k_2$};
	\node at (1.5,-2.5) {\tiny $k_3$};
	\node at (4.5,-2.5) {\tiny $k_4$};
	\node at (0,2.09) {\tiny $1$};
	\node at (0,4.75) {\tiny $1$};
	\node at (-3,0.75) {\tiny $1$};
	\node at (3,3.42) {\tiny $1$};
	\node at (-4.5,6.5) {\tiny $k_1{-}1$};
	\node at (-1.5,6.5) {\tiny $k_2{-}1$};
	\node at (1.5,6.5) {\tiny $k_3{+}1$};
	\node at (4.5,6.5) {\tiny $k_4{+}1$};
\end{tikzpicture}
};
\endxy
+
\xy
(0,0)*{
\begin{tikzpicture}[scale=.3]
	\draw [very thick, green, directed=.125, directed=.875] (-4.5,-2) to (-4.5,6);
	\draw [very thick, green, directed=.125, directed=.425, directed=.7, directed=.875] (-1.5,-2) to (-1.5,6);
	\draw [very thick, mycolor, directed=.125, directed=.36, directed=.51, directed=.875] (1.5,-2) to (1.5,6);
	\draw [very thick, mycolor, directed=.125, directed=.875] (4.5,-2) to (4.5,6);
	\draw [very thick, directed=.55] (-1.5,0) to (1.5,0);
	\draw [very thick, green, directed=.55] (-4.5,4) to (-1.5,4);
	\draw [very thick, mycolor, directed=.55] (1.5,1.333) to (4.5,1.333);
	\draw [very thick, directed=.55] (-1.5,2.666) to (1.5,2.666);
	\node at (-4.5,-2.5) {\tiny $k_1$};
	\node at (-1.5,-2.5) {\tiny $k_2$};
	\node at (1.5,-2.5) {\tiny $k_3$};
	\node at (4.5,-2.5) {\tiny $k_4$};
	\node at (0,0.75) {\tiny $1$};
	\node at (0,3.42) {\tiny $1$};
	\node at (-3,4.75) {\tiny $1$};
	\node at (3,2.09) {\tiny $1$};
	\node at (-4.5,6.5) {\tiny $k_1{-}1$};
	\node at (-1.5,6.5) {\tiny $k_2{-}1$};
	\node at (1.5,6.5) {\tiny $k_3{+}1$};
	\node at (4.5,6.5) {\tiny $k_4{+}1$};
\end{tikzpicture}
};
\endxy
\\
&+
\xy
(0,0)*{
\begin{tikzpicture}[scale=.3]
	\draw [very thick, green, directed=.125, directed=.875] (-4.5,-2) to (-4.5,6);
	\draw [very thick, green, directed=.125, directed=.36, directed=.51, directed=.875] (-1.5,-2) to (-1.5,6);
	\draw [very thick, mycolor, directed=.125, directed=.425, directed=.7, directed=.875] (1.5,-2) to (1.5,6);
	\draw [very thick, mycolor, directed=.125, directed=.875] (4.5,-2) to (4.5,6);
	\draw [very thick, directed=.55] (-1.5,0) to (1.5,0);
	\draw [very thick, mycolor, directed=.55] (1.5,4) to (4.5,4);
	\draw [very thick, green, directed=.55] (-4.5,1.333) to (-1.5,1.333);
	\draw [very thick, directed=.55] (-1.5,2.666) to (1.5,2.666);
	\node at (-4.5,-2.5) {\tiny $k_1$};
	\node at (-1.5,-2.5) {\tiny $k_2$};
	\node at (1.5,-2.5) {\tiny $k_3$};
	\node at (4.5,-2.5) {\tiny $k_4$};
	\node at (0,0.75) {\tiny $1$};
	\node at (0,3.42) {\tiny $1$};
	\node at (-3,2.09) {\tiny $1$};
	\node at (3,4.75) {\tiny $1$};
	\node at (-4.5,6.5) {\tiny $k_1{-}1$};
	\node at (-1.5,6.5) {\tiny $k_2{-}1$};
	\node at (1.5,6.5) {\tiny $k_3{+}1$};
	\node at (4.5,6.5) {\tiny $k_4{+}1$};
\end{tikzpicture}
};
\endxy+
\xy
(0,0)*{
\begin{tikzpicture}[scale=.3]
	\draw [very thick, green, directed=.125, directed=.875] (-4.5,-2) to (-4.5,6);
	\draw [very thick, green, directed=.125, directed=.51, directed=.7, directed=.875] (-1.5,-2) to (-1.5,6);
	\draw [very thick, mycolor, directed=.125, directed=.365, directed=.625, directed=.875] (1.5,-2) to (1.5,6);
	\draw [very thick, mycolor, directed=.125, directed=.875] (4.5,-2) to (4.5,6);
	\draw [very thick, directed=.55] (-1.5,1.333) to (1.5,1.333);
	\draw [very thick, green, directed=.55] (-4.5,2.666) to (-1.5,2.666);
	\draw [very thick, mycolor, directed=.55] (1.5,0) to (4.5,0);
	\draw [very thick, directed=.55] (-1.5,4) to (1.5,4);
	\node at (-4.5,-2.5) {\tiny $k_1$};
	\node at (-1.5,-2.5) {\tiny $k_2$};
	\node at (1.5,-2.5) {\tiny $k_3$};
	\node at (4.5,-2.5) {\tiny $k_4$};
	\node at (0,2.09) {\tiny $1$};
	\node at (0,4.75) {\tiny $1$};
	\node at (-3,3.42) {\tiny $1$};
	\node at (3,0.75) {\tiny $1$};
	\node at (-4.5,6.5) {\tiny $k_1{-}1$};
	\node at (-1.5,6.5) {\tiny $k_2{-}1$};
	\node at (1.5,6.5) {\tiny $k_3{+}1$};
	\node at (4.5,6.5) {\tiny $k_4{+}1$};
\end{tikzpicture}
};
\endxy
\end{align*}
and similarly for 
exchanged roles of green and red and flipped 
horizontal orientations.
\end{itemize}
\end{lem}

\begin{proof}
\textbf{(a)}: 
This follows directly from \eqref{eq-frob}, \fullref{lem-green} and symmetry.

\textbf{(b)}: Let $u$ and $v$ denote the two webs on the 
right-hand side of (b) above.
Using \eqref{eq-almostsimpler} for the edges labeled $k+1$ 
and $l+1$ in $u$ respectively $v$, we get
\[
u=
\xy
(0,0)*{
\begin{tikzpicture}[scale=.3]
	\draw [very thick, green, directed=.55] (-2,-4) to (-2,-2);
	\draw [very thick, green, directed=.55] (-2,-2) to (-2,2);
	\draw [very thick, mycolor, directed=.55] (2,-4) to (2,-2);
	\draw [very thick, mycolor, directed=.55] (2,-2) to (2,2);
	\draw [very thick, directed=.55] (-2,-2) to (0,-2);
	\draw [very thick, rdirected=.55] (0,-2) to (2,-2);
	\draw [very thick, green, directed=.55] (-2,2) to (-2,4);
	\draw [very thick, mycolor, directed=.55] (2,2) to (2,4);
	\draw [very thick, rdirected=.55] (-2,2) to (0,2);
	\draw [very thick, directed=.55] (0,2) to (2,2);
	\draw [very thick, green, directed=.55] (0,-2) to (0,2);
	\node at (-2,-4.5) {\tiny $k$};
	\node at (2,-4.5) {\tiny $l$};
	\node at (-2,4.5) {\tiny $k$};
	\node at (2,4.5) {\tiny $l$};
	\node at (-3.5,0) {\tiny $k{-}1$};
	\node at (3.5,0) {\tiny $l{-}1$};
	\node at (-1,-1.25) {\tiny $1$};
	\node at (-1,2.75) {\tiny $1$};
	\node at (1,-1.25) {\tiny $1$};
	\node at (1,2.75) {\tiny $1$};
	\node at (-0.6,0) {\tiny $2$};
\end{tikzpicture}
};
\endxy
-[k-1][l]
\xy
(0,0)*{
\begin{tikzpicture}[scale=.3]
	\draw [very thick, green, directed=.55] (-2,-4) to (-2,4);
	\draw [very thick, mycolor, directed=.55] (2,-4) to (2,4);
	\node at (-2,-4.5) {\tiny $k$};
	\node at (2,-4.5) {\tiny $l$};
	\node at (-2,4.5) {\tiny $k$};
	\node at (2,4.5) {\tiny $l$};
\end{tikzpicture}
};
\endxy
\quad,\quad
v=
\xy
(0,0)*{
\begin{tikzpicture}[scale=.3]
	\draw [very thick, green, directed=.55] (-2,-4) to (-2,-2);
	\draw [very thick, green, directed=.55] (-2,-2) to (-2,2);
	\draw [very thick, mycolor, directed=.55] (2,-4) to (2,-2);
	\draw [very thick, mycolor, directed=.55] (2,-2) to (2,2);
	\draw [very thick, directed=.55] (-2,-2) to (0,-2);
	\draw [very thick, rdirected=.55] (0,-2) to (2,-2);
	\draw [very thick, green, directed=.55] (-2,2) to (-2,4);
	\draw [very thick, mycolor, directed=.55] (2,2) to (2,4);
	\draw [very thick, rdirected=.55] (-2,2) to (0,2);
	\draw [very thick, directed=.55] (0,2) to (2,2);
	\draw [very thick, mycolor, directed=.55] (0,-2) to (0,2);
	\node at (-2,-4.5) {\tiny $k$};
	\node at (2,-4.5) {\tiny $l$};
	\node at (-2,4.5) {\tiny $k$};
	\node at (2,4.5) {\tiny $l$};
	\node at (-3.5,0) {\tiny $k{-}1$};
	\node at (3.5,0) {\tiny $l{-}1$};
	\node at (-1,-1.25) {\tiny $1$};
	\node at (-1,2.75) {\tiny $1$};
	\node at (1,-1.25) {\tiny $1$};
	\node at (1,2.75) {\tiny $1$};
	\node at (-0.6,0) {\tiny $2$};
\end{tikzpicture}
};
\endxy
+[k][1-l]
\xy
(0,0)*{
\begin{tikzpicture}[scale=.3]
	\draw [very thick, green, directed=.55] (-2,-4) to (-2,4);
	\draw [very thick, mycolor, directed=.55] (2,-4) to (2,4);
	\node at (-2,-4.5) {\tiny $k$};
	\node at (2,-4.5) {\tiny $l$};
	\node at (-2,4.5) {\tiny $k$};
	\node at (2,4.5) {\tiny $l$};
\end{tikzpicture}
};
\endxy
\]
after collapsing appearing digons. By 
using \eqref{eq-dumb} on the central 
vertical edges in the expansions, we 
see that $u+v=s\cdot\mathrm{id}_{(k,l)}$. The scalar is 
$s=[2][k][l]+[k][1-l]-[k-1][l]=[k+l]$. 
The other cases follow by symmetry.

\textbf{(c)}: We start with the web 
on the left-hand side and first use \eqref{eq-dumb} on 
the middle two horizontal edges. Thus, we 
obtain (our drawings are simplified and the 
orientations pointing down could be isotoped to point up)
\[
[2]
\xy
(0,0)*{
\begin{tikzpicture}[scale=.3]
	\draw [very thick, green, directed=.25, directed=.75] (-4.5,-2) to (-4.5,6);
	\draw [very thick, green, directed=.125, directed=.375, directed=.625, directed=.875] (-1.5,-2) to (-1.5,6);
	\draw [very thick, mycolor, directed=.125, directed=.375, directed=.625, directed=.875] (1.5,-2) to (1.5,6);
	\draw [very thick, mycolor, directed=.25, directed=.75] (4.5,-2) to (4.5,6);
	\draw [very thick, directed=.55] (-1.5,0) to (1.5,0);
	\draw [very thick, green, directed=.55] (-4.5,2) to (-1.5,2);
	\draw [very thick, mycolor, directed=.55] (1.5,2) to (4.5,2);
	\draw [very thick, directed=.55] (-1.5,4) to (1.5,4);
	\node at (-4.5,-2.5) {\tiny $k_1$};
	\node at (-1.5,-2.5) {\tiny $k_2$};
	\node at (1.5,-2.5) {\tiny $k_3$};
	\node at (4.5,-2.5) {\tiny $k_4$};
	\node at (0,0.75) {\tiny $1$};
	\node at (0,4.75) {\tiny $1$};
	\node at (-3,2.75) {\tiny $1$};
	\node at (3,2.75) {\tiny $1$};
	\node at (-4.5,6.5) {\tiny $k_1{-}1$};
	\node at (-1.5,6.5) {\tiny $k_2{-}1$};
	\node at (1.5,6.5) {\tiny $k_3{+}1$};
	\node at (4.5,6.5) {\tiny $k_4{+}1$};
\end{tikzpicture}
};
\endxy
=
\xy
(0,0)*{
\begin{tikzpicture}[scale=.3]
	\draw [very thick, green, directed=.25, directed=.75] (-4.5,-2) to (-4.5,6);
	\draw [very thick, green, directed=.125, directed=.375, directed=.625, directed=.875] (-1.5,-2) to (-1.5,6);
	\draw [very thick, mycolor, directed=.125, directed=.375, directed=.625, directed=.875] (1.5,-2) to (1.5,6);
	\draw [very thick, mycolor, directed=.25, directed=.75] (4.5,-2) to (4.5,6);
	\draw [very thick, green, directed=.55] (-4.5,2) to (-1.5,2);
	\draw [very thick, mycolor, directed=.55] (1.5,2) to (4.5,2);
	\draw [very thick, directed=.65] (-1.5,0) to (-0.5,2);
	\draw [very thick, directed=.55] (0.5,2) to (1.5,0);
	\draw [very thick, directed=.65] (-1.5,4) to (-0.5,2);
	\draw [very thick, directed=.55] (0.5,2) to (1.5,4);
	\draw [very thick, green, directed=.7] (-0.5,2) to (0.5,2);
	\draw [dashed] (-2,-.25) rectangle (-0.25,4.25);
	\node at (-4.5,-2.5) {\tiny $k_1$};
	\node at (-1.5,-2.5) {\tiny $k_2$};
	\node at (1.5,-2.5) {\tiny $k_3$};
	\node at (4.5,-2.5) {\tiny $k_4$};
	\node at (0,2.7) {\tiny $2$};
	\node at (-3,2.75) {\tiny $1$};
	\node at (3,2.75) {\tiny $1$};
	\node at (-4.5,6.5) {\tiny $k_1{-}1$};
	\node at (-1.5,6.5) {\tiny $k_2{-}1$};
	\node at (1.5,6.5) {\tiny $k_3{+}1$};
	\node at (4.5,6.5) {\tiny $k_4{+}1$};
\end{tikzpicture}
};
\endxy
+
\xy
(0,0)*{
\begin{tikzpicture}[scale=.3]
	\draw [very thick, green, directed=.25, directed=.75] (-4.5,-2) to (-4.5,6);
	\draw [very thick, green, directed=.125, directed=.375, directed=.625, directed=.875] (-1.5,-2) to (-1.5,6);
	\draw [very thick, mycolor, directed=.125, directed=.375, directed=.625, directed=.875] (1.5,-2) to (1.5,6);
	\draw [very thick, mycolor, directed=.25, directed=.75] (4.5,-2) to (4.5,6);
	\draw [very thick, green, directed=.55] (-4.5,2) to (-1.5,2);
	\draw [very thick, mycolor, directed=.55] (1.5,2) to (4.5,2);
	\draw [very thick, directed=.65] (-1.5,0) to (-0.5,2);
	\draw [very thick, directed=.55] (0.5,2) to (1.5,0);
	\draw [very thick, directed=.65] (-1.5,4) to (-0.5,2);
	\draw [very thick, directed=.55] (0.5,2) to (1.5,4);
	\draw [very thick, mycolor, directed=.7] (-0.5,2) to (0.5,2);
	\draw [dashed] (0.25,-.25) rectangle (2,4.25);
	\node at (-4.5,-2.5) {\tiny $k_1$};
	\node at (-1.5,-2.5) {\tiny $k_2$};
	\node at (1.5,-2.5) {\tiny $k_3$};
	\node at (4.5,-2.5) {\tiny $k_4$};
	\node at (0,2.7) {\tiny $2$};
	\node at (-3,2.75) {\tiny $1$};
	\node at (3,2.75) {\tiny $1$};
	\node at (-4.5,6.5) {\tiny $k_1{-}1$};
	\node at (-1.5,6.5) {\tiny $k_2{-}1$};
	\node at (1.5,6.5) {\tiny $k_3{+}1$};
	\node at (4.5,6.5) {\tiny $k_4{+}1$};
\end{tikzpicture}
};
\endxy
\]
The two marked parts above are monochromatic squares, which can be switched to give:
\[
\xy
(0,0)*{
\begin{tikzpicture}[scale=.3]
	\draw [very thick, green, directed=.55] (-2,-4) to (-2,-2);
	\draw [very thick, green, directed=.55] (-2,-2) to (-2,2);
	\draw [very thick, green, directed=.55] (2,-4) to (2,-2);
	\draw [very thick, green, directed=.55] (2,-2) to (2,2);
	\draw [very thick, green, rdirected=.55] (-2,-2) to (2,-2);
	\draw [very thick, green, directed=.55] (-2,2) to (-2,4);
	\draw [very thick, green, directed=.55] (2,2) to (2,4);
	\draw [very thick, green, directed=.55] (-2,2) to (2,2);
	\node at (-2,-4.5) {\tiny $1$};
	\node at (2,-4.5) {\tiny $k_2$};
	\node at (-2.25,4.5) {\tiny $k_2{-}1$};
	\node at (2,4.5) {\tiny $2$};
	\node at (-2.85,0) {\tiny $k_2$};
	\node at (2.6,0) {\tiny $1$};
	\node at (0,-1.25) {\tiny $k_2{-}1$};
	\node at (0,2.75) {\tiny $1$};
\end{tikzpicture}
};
\endxy
=
\xy
(0,0)*{
\begin{tikzpicture}[scale=.3]
	\draw [very thick, green, directed=.55] (-2,-4) to (-2,-2);
	\draw [very thick, green, directed=.55] (2,-4) to (2,-2);
	\draw [very thick, green, directed=.55] (2,-2) to (2,2);
	\draw [very thick, green, directed=.55] (-2,-2) to (2,-2);
	\draw [very thick, green, directed=.55] (-2,2) to (-2,4);
	\draw [very thick, green, directed=.55] (2,2) to (2,4);
	\draw [very thick, green, rdirected=.55] (-2,2) to (2,2);
	\node at (-2,-4.5) {\tiny $1$};
	\node at (2,-4.5) {\tiny $k_2$};
	\node at (-2.25,4.5) {\tiny $k_2{-}1$};
	\node at (2,4.5) {\tiny $2$};
	\node at (0.4,0) {\tiny $k_2{+}1$};
	\node at (0,-1.25) {\tiny $1$};
	\node at (0,2.75) {\tiny $k_2{-}1$};
\end{tikzpicture}
};
\endxy
+
\xy
(0,0)*{
\begin{tikzpicture}[scale=.3]
	\draw [very thick, green, directed=.55] (-2,-4) to (-2,-2);
	\draw [very thick, green] (-2,-2) to (-2,2);
	\draw [very thick, green, directed=.55] (2,-4) to (2,-2);
	\draw [very thick, green] (2,-2) to (2,2);
	\draw [very thick, green, directed=.55] (-2,2) to (-2,4);
	\draw [very thick, green, directed=.55] (2,2) to (2,4);
	\draw [very thick, green, rdirected=.55] (-2,2) to (2,2);
	\node at (-2,-4.5) {\tiny $1$};
	\node at (2,-4.5) {\tiny $k_2$};
	\node at (-2.25,4.5) {\tiny $k_2{-}1$};
	\node at (2,4.5) {\tiny $2$};
	\node at (0,2.75) {\tiny $k_2{-}2$};
\end{tikzpicture}
};
\endxy\quad
,
\quad
\xy
(0,0)*{
\begin{tikzpicture}[scale=.3]
	\draw [very thick, mycolor, directed=.55] (-2,-4) to (-2,-2);
	\draw [very thick, mycolor, directed=.55] (-2,-2) to (-2,2);
	\draw [very thick, mycolor, directed=.55] (2,-4) to (2,-2);
	\draw [very thick, mycolor, directed=.55] (2,-2) to (2,2);
	\draw [very thick, mycolor, directed=.55] (-2,-2) to (2,-2);
	\draw [very thick, mycolor, directed=.55] (-2,2) to (-2,4);
	\draw [very thick, mycolor, directed=.55] (2,2) to (2,4);
	\draw [very thick, mycolor, rdirected=.55] (-2,2) to (2,2);
	\node at (-2,-4.5) {\tiny $2$};
	\node at (2,-4.5) {\tiny $k_3$};
	\node at (-2.25,4.5) {\tiny $k_3{+}1$};
	\node at (2,4.5) {\tiny $1$};
	\node at (-2.6,0) {\tiny $1$};
	\node at (3.7,0) {\tiny $k_3{+}1$};
	\node at (0,-1.25) {\tiny $1$};
	\node at (0,2.75) {\tiny $k_3$};
\end{tikzpicture}
};
\endxy
=
\xy
(0,0)*{
\begin{tikzpicture}[scale=.3]
	\draw [very thick, mycolor, directed=.55] (-2,-4) to (-2,-2);
	\draw [very thick, mycolor, directed=.55] (-2,-2) to (-2,2);
	\draw [very thick, mycolor, directed=.55] (2,-4) to (2,-2);
	\draw [very thick, mycolor, rdirected=.55] (-2,-2) to (2,-2);
	\draw [very thick, mycolor, directed=.55] (-2,2) to (-2,4);
	\draw [very thick, mycolor, directed=.55] (2,2) to (2,4);
	\draw [very thick, mycolor, directed=.55] (-2,2) to (2,2);
	\node at (-2,-4.5) {\tiny $2$};
	\node at (2,-4.5) {\tiny $k_3$};
	\node at (-2.25,4.5) {\tiny $k_3{+}1$};
	\node at (2,4.5) {\tiny $1$};
	\node at (-0.3,0) {\tiny $k_3{+}2$};
	\node at (0,-1.25) {\tiny $k_3$};
	\node at (0,2.75) {\tiny $1$};
\end{tikzpicture}
};
\endxy
+
\xy
(0,0)*{
\begin{tikzpicture}[scale=.3]
	\draw [very thick, mycolor, directed=.55] (-2,-4) to (-2,-2);
	\draw [very thick, mycolor] (-2,-2) to (-2,2);
	\draw [very thick, mycolor, directed=.55] (2,-4) to (2,-2);
	\draw [very thick, mycolor] (2,-2) to (2,2);
	\draw [very thick, mycolor, rdirected=.55] (-2,-2) to (2,-2);
	\draw [very thick, mycolor, directed=.55] (-2,2) to (-2,4);
	\draw [very thick, mycolor, directed=.55] (2,2) to (2,4);
	\node at (-2,-4.5) {\tiny $2$};
	\node at (2,-4.5) {\tiny $k_3$};
	\node at (-2.25,4.5) {\tiny $k_3{+}1$};
	\node at (2,4.5) {\tiny $1$};
	\node at (0,-1.25) {\tiny $k_3{-}1$};
\end{tikzpicture}
};
\endxy
\]
Plugging these four terms back in, we get the four 
webs from the right-hand side 
of the equation in (c) (in the indicated order) which 
can be seen by using \eqref{eq-frob} as for example
\[
\xy
(0,0)*{
\begin{tikzpicture}[scale=.3]
	\draw [very thick, green, directed=.125, directed=.875] (-4.5,-2) to (-4.5,6);
	\draw [very thick, green, directed=.125, directed=.365, directed=.625, directed=.875] (-1.5,-2) to (-1.5,6);
	\draw [very thick, mycolor, directed=.125, directed=.51, directed=.7, directed=.875] (1.5,-2) to (1.5,6);
	\draw [very thick, mycolor, directed=.125, directed=.875] (4.5,-2) to (4.5,6);
	\draw [very thick, directed=.65] (0,1.333) to (1.5,1.333);
	\draw [very thick, green, directed=.65] (-1.5,1.333) to (0,1.333);
	\draw [very thick, mycolor, directed=.55] (1.5,2.666) to (4.5,2.666);
	\draw [very thick, green, directed=.55] (-4.5,0) to (-1.5,0);
	\draw [very thick, directed=.65] (0,4) to (1.5,4);
	\draw [very thick, directed=.55] (0,1.333) to (0,4);
	\node at (-4.5,-2.5) {\tiny $k_1$};
	\node at (-1.5,-2.5) {\tiny $k_2$};
	\node at (1.5,-2.5) {\tiny $k_3$};
	\node at (4.5,-2.5) {\tiny $k_4$};
	\node at (-0.75,2.09) {\tiny $2$};
	\node at (0.75,2.09) {\tiny $1$};
	\node at (0.75,4.75) {\tiny $1$};
	\node at (-3,0.75) {\tiny $1$};
	\node at (3,3.42) {\tiny $1$};
	\node at (-4.5,6.5) {\tiny $k_1{-}1$};
	\node at (-1.5,6.5) {\tiny $k_2{-}1$};
	\node at (1.5,6.5) {\tiny $k_3{+}1$};
	\node at (4.5,6.5) {\tiny $k_4{+}1$};
\end{tikzpicture}
};
\endxy
=
\xy
(0,0)*{
\begin{tikzpicture}[scale=.3]
	\draw [very thick, green, directed=.125, directed=.875] (-4.5,-2) to (-4.5,6);
	\draw [very thick, green, directed=.125, directed=.365, directed=.625, directed=.875] (-1.5,-2) to (-1.5,6);
	\draw [very thick, mycolor, directed=.125, directed=.51, directed=.7, directed=.875] (1.5,-2) to (1.5,6);
	\draw [very thick, mycolor, directed=.125, directed=.875] (4.5,-2) to (4.5,6);
	\draw [very thick, directed=.55] (-1.5,1.333) to (1.5,1.333);
	\draw [very thick, mycolor, directed=.55] (1.5,2.666) to (4.5,2.666);
	\draw [very thick, green, directed=.55] (-4.5,0) to (-1.5,0);
	\draw [very thick, directed=.55] (-1.5,4) to (1.5,4);
	\node at (-4.5,-2.5) {\tiny $k_1$};
	\node at (-1.5,-2.5) {\tiny $k_2$};
	\node at (1.5,-2.5) {\tiny $k_3$};
	\node at (4.5,-2.5) {\tiny $k_4$};
	\node at (0,2.09) {\tiny $1$};
	\node at (0,4.75) {\tiny $1$};
	\node at (-3,0.75) {\tiny $1$};
	\node at (3,3.42) {\tiny $1$};
	\node at (-4.5,6.5) {\tiny $k_1{-}1$};
	\node at (-1.5,6.5) {\tiny $k_2{-}1$};
	\node at (1.5,6.5) {\tiny $k_3{+}1$};
	\node at (4.5,6.5) {\tiny $k_4{+}1$};
\end{tikzpicture}
};
\endxy
\]
The other three
cases in (c) follow by symmetry.
\end{proof}

\subsection{Green and red clasps}
\label{sub-jw}

We show now that our calculus 
contains web analogues of the 
\textit{Jones--Wenzl projectors} of the Temperley--Lieb algebra. We call them \textit{clasps},
following \cite{kup}.

From now on, we denote by capital vectors as 
$\vec{K}\in X^K$ special objects of $\InSp$ 
of the form $\vec{K}=(1_b,\dots,1_b)$ with $K$ entries equal $1_b$ 
and no other entries.

\begin{defn}\label{defn-symJW}
Let $K\in\Z_{>0}$. We define the \textit{$K^{\text{th}}$ green clasp}
$\cal{CL}^g_K\in \End_{\InSp}(\vec{K})$ recursively: 
$\cal{CL}^g_1$ is the black identity strand and for $K\in\Z_{>1}$ set
\[
\xy
(0,0)*{
\begin{tikzpicture}[scale=.3]
	\draw [very thick] (-2,-8) to (-2,-6);
	\draw [very thick] (0,-8) to (0,-6);
	\draw [very thick] (2,-8) to (2,-6);
	\draw [very thick] (4,-8) to (4,-6);
	\draw [very thick, ->] (-2,-4) to (-2,-2);
	\draw [very thick, ->] (0,-4) to (0,-2);
	\draw [very thick, ->] (2,-4) to (2,-2);
	\draw [very thick, ->] (4,-4) to (4,-2);
	\draw [very thick, green] (-2.3,-6) rectangle (4.3,-4);
	\node at (-2,-8.45) {\tiny $1$};
	\node at (0,-8.45) {\tiny $1$};
	\node at (2,-8.45) {\tiny $1$};
	\node at (4,-8.45) {\tiny $1$};
	\node at (-2,-1.55) {\tiny $1$};
	\node at (0,-1.55) {\tiny $1$};
	\node at (2,-1.55) {\tiny $1$};
	\node at (4,-1.55) {\tiny $1$};
	\node at (-1,-2.75) {\tiny $\cdots$};
	\node at (-1,-7.5) {\tiny $\cdots$};
	\node at (1,-5) {\tiny $\cal{CL}^g_K$};
\end{tikzpicture}
};
\endxy
=
\xy
(0,0)*{
\begin{tikzpicture}[scale=.3]
	\draw [very thick] (-2,-8) to (-2,-6);
	\draw [very thick] (0,-8) to (0,-6);
	\draw [very thick] (2,-8) to (2,-6);
	\draw [very thick, ->] (-2,-4) to (-2,-2);
	\draw [very thick, ->] (0,-4) to (0,-2);
	\draw [very thick, ->] (2,-4) to (2,-2);
	\draw [very thick, ->] (4,-8) to (4,-2);
	\draw [very thick, green] (-2.3,-6) rectangle (2.3,-4);
	\node at (-2,-8.45) {\tiny $1$};
	\node at (0,-8.45) {\tiny $1$};
	\node at (2,-8.45) {\tiny $1$};
	\node at (4,-8.45) {\tiny $1$};
	\node at (-2,-1.55) {\tiny $1$};
	\node at (0,-1.55) {\tiny $1$};
	\node at (2,-1.55) {\tiny $1$};
	\node at (4,-1.55) {\tiny $1$};
	\node at (-1,-2.75) {\tiny $\cdots$};
	\node at (-1,-7.5) {\tiny $\cdots$};
	\node at (0,-5) {\tiny $\cal{CL}^g_K$};
\end{tikzpicture}
};
\endxy
-\frac{[K-1]}{[K]}
\xy
(0,0)*{
\begin{tikzpicture}[scale=.3]
	\draw [very thick] (-2,-8) to (-2,-7);
	\draw [very thick] (0,-8) to (0,-7);
	\draw [very thick] (2,-8) to (2,-7);
	\draw [very thick] (4,-8) to (4,-5.85);
	\draw [very thick, ->] (-2,-3) to (-2,-2);
	\draw [very thick, ->] (0,-3) to (0,-2);
	\draw [very thick, ->] (2,-3) to (2,-2);
	\draw [very thick, ->] (4,-4.15) to (4,-2);
	\draw [very thick] (-2,-5.85) to (-2,-4.15);
	\draw [very thick] (0,-5.85) to (0,-4.15);
	\draw [very thick] (3,-4.5) to (2,-4.15);
	\draw [very thick] (2,-5.85) to (3,-5.5);
	\draw [very thick, mycolor, ->] (3,-5.5) to (3,-4.5);
	\draw [very thick] (4,-5.85) to (3,-5.5);
	\draw [very thick] (3,-4.5) to (4,-4.15);
	\draw [very thick, green] (-2.3,-7) rectangle (2.3,-5.85);
	\draw [very thick, green] (-2.3,-4.15) rectangle (2.3,-3);
	\node at (-2,-8.45) {\tiny $1$};
	\node at (0,-8.45) {\tiny $1$};
	\node at (2,-8.45) {\tiny $1$};
	\node at (4,-8.45) {\tiny $1$};
	\node at (-2,-1.55) {\tiny $1$};
	\node at (0,-1.55) {\tiny $1$};
	\node at (2,-1.55) {\tiny $1$};
	\node at (4,-1.55) {\tiny $1$};
	\node at (2.75,-3.8) {\tiny $1$};
	\node at (2.75,-6.2) {\tiny $1$};
	\node at (3.65,-5) {\tiny $2$};
	\node at (-2.35,-5) {\tiny $1$};
	\node at (0.35,-5) {\tiny $1$};
	\node at (-1,-2.75) {\tiny $\cdots$};
	\node at (-1,-5) {\tiny $\cdots$};
	\node at (-1,-7.5) {\tiny $\cdots$};
	\node at (0,-3.63) {\tiny $\cal{CL}^g_K$};
	\node at (0,-6.475) {\tiny $\cal{CL}^g_K$};
\end{tikzpicture}
};
\endxy
\]
and similarly for 
the \textit{red clasp} $\cal{CL}^r_K$ by exchanging green and red.
\end{defn}

The following lemma identifies the clasps avoiding the recursive definition.

\begin{lem}\label{lem-recursion}
We have for all $K\in\Z_{>0}$
\[
\cal{CL}^g_K
=\frac{1}{[K]!}
\xy
(0,0)*{
\begin{tikzpicture}[scale=.3]
	\draw [very thick, green, directed=.55] (0,-1) to (0,.75);
	\draw [very thick, directed=.55] (0,.75) to [out=30,in=270] (1,2.5);
	\draw [very thick, directed=.55] (0,.75) to [out=150,in=270] (-1,2.5); 
	\draw [very thick, directed=.55] (1,-2.75) to [out=90,in=330] (0,-1);
	\draw [very thick, directed=.55] (-1,-2.75) to [out=90,in=210] (0,-1);
	\node at (-1,3) {\tiny $1$};
	\node at (0.1,3) {$\cdots$};
	\node at (1,3) {\tiny $1$};
	\node at (-1,-3.15) {\tiny $1$};
	\node at (0.1,-3.15) {$\cdots$};
	\node at (1,-3.15) {\tiny $1$};
	\node at (-0.65,0) {\tiny $K$};
\end{tikzpicture}
};
\endxy\quad,\quad
\cal{CL}^r_K
=\frac{1}{[K]!}
\xy
(0,0)*{
\begin{tikzpicture}[scale=.3]
	\draw [very thick, mycolor, directed=.55] (0,-1) to (0,.75);
	\draw [very thick, directed=.55] (0,.75) to [out=30,in=270] (1,2.5);
	\draw [very thick, directed=.55] (0,.75) to [out=150,in=270] (-1,2.5); 
	\draw [very thick, directed=.55] (1,-2.75) to [out=90,in=330] (0,-1);
	\draw [very thick, directed=.55] (-1,-2.75) to [out=90,in=210] (0,-1);
	\node at (-1,3) {\tiny $1$};
	\node at (0.1,3) {$\cdots$};
	\node at (1,3) {\tiny $1$};
	\node at (-1,-3.15) {\tiny $1$};
	\node at (0.1,-3.15) {$\cdots$};
	\node at (1,-3.15) {\tiny $1$};
	\node at (-0.65,0) {\tiny $K$};
\end{tikzpicture}
};
\endxy
\]
where we repeatedly split an edge labeled $K$ until all of 
the top and bottom edges are black.
\end{lem}

\begin{proof}
Up to signs and drawing conventions as in \cite[Lemma 2.12]{rt} and left to the reader.
\end{proof}

\begin{cor}\label{cor-green}
For all $K\in\Z_{>0}$: the projector $\cal{CL}^g_K$ can be expressed 
as a linear combination of webs with 
only black and red edges of label $2$, and 
similarly for $\cal{CL}^r_K$.
\end{cor}

\begin{proof}
This follows directly from \eqref{eq-dumb} and \fullref{lem-recursion}.
\end{proof}

\begin{ex}\label{ex-JW}
The projector $\cal{CL}_1^r$ is just the black 
identity strand, the projector $\cal{CL}_2^r$ is 
$\frac{1}{[2]}$ times the red dumbbell as in \eqref{eq-dumb}
and
\[
\cal{CL}^r_3 \! =\!
\frac{1}{[3]!}
\xy
(0,0)*{
\begin{tikzpicture}[scale=.3]
	\draw [very thick, mycolor, directed=.55] (0,-1) to (0,.75);
	\draw [very thick, directed=.55] (0,.75) to [out=30,in=270] (1,2.5);
	\draw [very thick, directed=.55] (0,.75) to [out=150,in=270] (-1,2.5); 
	\draw [very thick, directed=.55] (1,-2.75) to [out=90,in=330] (0,-1);
	\draw [very thick, directed=.55] (-1,-2.75) to [out=90,in=210] (0,-1);
	\node at (-1,3) {\tiny $1$};
	\node at (0.1,3) {$\cdots$};
	\node at (1,3) {\tiny $1$};
	\node at (-1,-3.15) {\tiny $1$};
	\node at (0.1,-3.15) {$\cdots$};
	\node at (1,-3.15) {\tiny $1$};
	\node at (-0.6,0) {\tiny $3$};
\end{tikzpicture}
};
\endxy
=
\xy
(0,0)*{
\begin{tikzpicture}[scale=.3] 
	\draw [very thick, directed=.55] (1,-2.75) to (1,2.5);
	\draw [very thick, directed=.55] (-1,-2.75) to (-1,2.5);
	\draw [very thick, directed=.55] (-3,-2.75) to (-3,2.5);
	\node at (-3,3) {\tiny $1$};
	\node at (-1,3) {\tiny $1$};
	\node at (1,3) {\tiny $1$};
	\node at (-3,-3.15) {\tiny $1$};
	\node at (-1,-3.15) {\tiny $1$};
	\node at (1,-3.15) {\tiny $1$};
\end{tikzpicture}
};
\endxy
-\frac{[2]}{[3]}
\xy
(0,0)*{
\begin{tikzpicture}[scale=.3] 
	\draw [very thick, green, directed=.55] (0,-1) to (0,.75);
	\draw [very thick, directed=.55] (0,.75) to [out=30,in=270] (1,2.5);
	\draw [very thick, directed=.55] (0,.75) to [out=150,in=270] (-1,2.5);
	\draw [very thick, directed=.55] (1,-2.75) to [out=90,in=330] (0,-1);
	\draw [very thick, directed=.55] (-1,-2.75) to [out=90,in=210] (0,-1); 
	\draw [very thick, directed=.55] (-3,-2.75) to (-3,2.5);
	\node at (-3,3) {\tiny $1$};
	\node at (-1,3) {\tiny $1$};
	\node at (1,3) {\tiny $1$};
	\node at (-3,-3.15) {\tiny $1$};
	\node at (-1,-3.15) {\tiny $1$};
	\node at (1,-3.15) {\tiny $1$};
	\node at (-0.5,0) {\tiny $2$};
\end{tikzpicture}
};
\endxy
+\frac{1}{[3]}
\left(
\xy
(0,0)*{
\begin{tikzpicture}[scale=.3] 
	\draw [very thick, green, directed=.65] (0,.625) to (0,1.75);
	\draw [very thick, directed=.55] (1,-0.125) to (0,.625);
	\draw [very thick, directed=.55] (-1,-0.125) to (0,.625);
	\draw [very thick, directed=.65] (0,1.75) to (1,2.5);
	\draw [very thick, directed=.65] (0,1.75) to (-1,2.5);
	\draw [very thick, green, directed=.65] (-2,-2) to (-2,-0.875);
	\draw [very thick] (-2,-0.875) to (-1,-0.125);
	\draw [very thick] (-2,-0.875) to (-3,-0.125);
	\draw [very thick, directed=.55] (-1,-2.75) to (-2,-2);
	\draw [very thick, directed=.55] (-3,-2.75) to (-2,-2);
	\draw [very thick, directed=.55] (1,-2.75) to (1,-0.125);
	\draw [very thick, directed=.55] (-3,-0.125) to (-3,2.5);
	\node at (-3,3) {\tiny $1$};
	\node at (-1,3) {\tiny $1$};
	\node at (1,3) {\tiny $1$};
	\node at (-3,-3.15) {\tiny $1$};
	\node at (-1,-3.15) {\tiny $1$};
	\node at (1,-3.15) {\tiny $1$};
	\node at (-1.3,0.125) {\tiny $1$};
	\node at (-0.55,1.1875) {\tiny $2$};
	\node at (-2.55,-1.4375) {\tiny $2$};
\end{tikzpicture}
};
\endxy
+
\reflectbox{
\xy
(0,0)*{
\begin{tikzpicture}[scale=.3] 
	\draw [very thick, green, directed=.65] (0,.625) to (0,1.75);
	\draw [very thick, directed=.55] (1,-0.125) to (0,.625);
	\draw [very thick, directed=.55] (-1,-0.125) to (0,.625);
	\draw [very thick, directed=.65] (0,1.75) to (1,2.5);
	\draw [very thick, directed=.65] (0,1.75) to (-1,2.5);
	\draw [very thick, green, directed=.65] (-2,-2) to (-2,-0.875);
	\draw [very thick] (-2,-0.875) to (-1,-0.125);
	\draw [very thick] (-2,-0.875) to (-3,-0.125);
	\draw [very thick, directed=.55] (-1,-2.75) to (-2,-2);
	\draw [very thick, directed=.55] (-3,-2.75) to (-2,-2);
	\draw [very thick, directed=.55] (1,-2.75) to (1,-0.125);
	\draw [very thick, directed=.55] (-3,-0.125) to (-3,2.5);
	\node at (-3,3) {\tiny \reflectbox{$1$}};
	\node at (-1,3) {\tiny \reflectbox{$1$}};
	\node at (1,3) {\tiny \reflectbox{$1$}};
	\node at (-3,-3.15) {\tiny \reflectbox{$1$}};
	\node at (-1,-3.15) {\tiny \reflectbox{$1$}};
	\node at (1,-3.15) {\tiny \reflectbox{$1$}};
	\node at (-1.3,0.125) {\tiny \reflectbox{$1$}};
	\node at (0.45,1.1875) {\tiny \reflectbox{$2$}};
	\node at (-1.55,-1.4375) {\tiny \reflectbox{$2$}};
\end{tikzpicture}
};
\endxy
}
\right)
-\frac{1}{[2][3]}
\xy
(0,0)*{
\begin{tikzpicture}[scale=.3] 
	\draw [very thick, green, directed=.75] (-2,1.25) to (-2,2);
	\draw [very thick, directed=.15] (-1,0.75) to (-2,1.25);
	\draw [very thick] (-3,0.75) to (-2,1.25);
	\draw [very thick, directed=.65] (-2,2) to (-1,2.5);
	\draw [very thick, directed=.65] (-2,2) to (-3,2.5);
	\draw [very thick, green, directed=.75] (-2,-2.25) to (-2,-1.5);
	\draw [very thick, ->] (-2,-1.5) to (-1,-1);
	\draw [very thick] (-2,-1.5) to (-3,-1);
	\draw [very thick, directed=.55] (-1,-2.75) to (-2,-2.25);
	\draw [very thick, directed=.55] (-3,-2.75) to (-2,-2.25);
	\draw [very thick, green, directed=.75] (0,-0.5) to (0,.25);
	\draw [very thick] (1,-1) to (0,-.5);
	\draw [very thick] (-1,-1) to (0,-.5);
	\draw [very thick] (0,.25) to (1,0.75);
	\draw [very thick] (0,.25) to (-1,0.75);
	\draw [very thick, directed=.55] (1,0.75) to (1,2.5);
	\draw [very thick, rdirected=.55] (1,-1) to (1,-2.75);
	\draw [very thick, rdirected=.55] (-3,0.75) to (-3,-1);
	\node at (-3,3) {\tiny $1$};
	\node at (-1,3) {\tiny $1$};
	\node at (1,3) {\tiny $1$};
	\node at (-3,-3.15) {\tiny $1$};
	\node at (-1,-3.15) {\tiny $1$};
	\node at (1,-3.15) {\tiny $1$};
	\node at (-2.55,1.625) {\tiny $2$};
	\node at (-2.55,-1.875) {\tiny $2$};
	\node at (-.55,-0.125) {\tiny $2$};
	\node at (-3.5,-0.25) {\tiny $1$};
	\node at (-0.65,1.25) {\tiny $1$};
	\node at (-0.65,-1.3) {\tiny $1$};
\end{tikzpicture}
};
\endxy
\]
Note that all edges appearing on the right-hand side are black or green with label $2$.
\end{ex}

\begin{prop}\label{prop-green}
Let $\vec{k}$ and $\vec{l}$ be sequences of black and green boundary points. 
Every web $u\in\Hom_{\InSp}(\vec{k},\vec{l})$ can
be expressed as a sum of webs with only black and green edges. 
Similarly by exchanging green and red.
\end{prop}

\begin{proof}
We start by exploding\footnote{We ``explode'' by using \eqref{eq-simpler1} 
(the order does not matter by \eqref{eq-frob}). We 
indicate ``explosions'' with dots.} every red edge. Around 
internal vertices of $u$ with no outgoing green edges we get
\[
\xy
(0,0)*{
\begin{tikzpicture}[scale=.5]
	\draw [very thick, mycolor, directed=.55] (-1,3) to (-1,6);
	\draw [very thick, mycolor, rdirected=.55] (-1,3) to (-3,0);
	\draw [very thick, mycolor, rdirected=.55] (-1,3) to (1,0);
	\node at (-1,6.25) {\tiny $k{+}l$};
	\node at (-3,-.25) {\tiny $k$};
	\node at (1,-.25) {\tiny $l$};
\end{tikzpicture}
};
\endxy
=
\frac{1}{[k]!}\frac{1}{[l]!}\frac{1}{[k+l]!}
\xy
(0,0)*{
\begin{tikzpicture}[scale=.5]
	\draw [very thick, mycolor, directed=.7] (-1,5.5) to (-1,6);
	\draw [very thick, mycolor, directed=.65] (-1,3) to (-1,3.5);
	\draw [very thick, directed=.55] (-1,3.5) to [out=150,in=210] (-1,5.5);
	\draw [very thick, directed=.55] (-1,3.5) to [out=30,in=330] (-1,5.5);
	\draw [very thick, mycolor, rdirected=.65] (-1,3) to (-1.45,2.333);
	\draw [very thick, mycolor, rdirected=.55] (-2.55,.67) to (-3,0);
	\draw [very thick, directed=.55] (-2.55,.67) to [out=116.3,in=176.3] (-1.45,2.333);
	\draw [very thick, directed=.55] (-2.55,.67) to [out=-3.7,in=296.3] (-1.45,2.333);
	\draw [very thick, mycolor, rdirected=.65] (-1,3) to (-.55,2.333);
	\draw [very thick, mycolor, rdirected=.55] (0.55,.67) to (1,0);
	\draw [very thick, directed=.55] (0.55,.67) to [out=187.7,in=243.7] (-.55,2.333);
	\draw [very thick, directed=.55] (0.55,.67) to [out=63.7,in=363.7] (-.55,2.333);
	\draw [dashed] (-3,0.9) rectangle (1,5);
	\node at (-1,6.25) {\tiny $k{+}l$};
	\node at (-3,-.25) {\tiny $k$};
	\node at (1,-.25) {\tiny $l$};
	\node at (-1.8,3.25) {\tiny $k{+}l$};
	\node at (-1.6,2.75) {\tiny $k$};
	\node at (-0.4,2.75) {\tiny $l$};
	\node at (-1.85,4.5) {\tiny $1$};
	\node at (-0.15,4.5) {\tiny $1$};
	\node at (-0.95,4.5) {\tiny $\cdots$};
	\node at (0,1.5) {\tiny \rotatebox{33.7}{$\cdots$}};
	\node at (-0.7,1.15) {\tiny $1$};
	\node at (0.7,1.85) {\tiny $1$};
	\node at (-2,1.5) {\tiny \rotatebox{-33.7}{$\cdots$}};
	\node at (-1.3,1.15) {\tiny $1$};
	\node at (-2.7,1.85) {\tiny $1$};
\end{tikzpicture}
};
\endxy
\]
Note that the marked part above is 
$\cal{CL}^r_{k+l}$ up to a non-zero scalar. 
This can be seen by using (co)associativity \eqref{eq-frob} and 
the expression in \fullref{lem-recursion}.
Thus, we can use \fullref{cor-green} to 
replace $\cal{CL}^r_{k+l}$ by a non-zero sum 
of webs with only black and green edges. Repeating this 
for all purely red internal vertices shows the statement, since 
all outer edges are assumed to be black or green. 
The other statement follows by symmetry.
\end{proof}

Denote by $\CKM$ the subcategory given in \cite[Definition 2.2]{ckm} 
with only upward-pointing strands, tags replaced by (untruncated) 
$N$-labeled edges and additionally allowing $0$-labeled objects.
As a consequence of \fullref{prop-green} we see that 
interpreting webs in $\CKM$ as green webs in $\SuSp$ 
gives a full functor $\iota^{\infty}_{1}$ between these categories. 
In \fullref{lem-faithful} we will see that it is also faithful and 
we get the following corollary.

\begin{cor}\label{cor-green2}
The functor $\iota^{\infty}_{1}\colon\CKM\to\SuSp$, given by coloring webs green, is an inclusion of a full, monoidal subcategory. In particular, $\CKM$ and 
$\AltSp$ are equivalent as monoidal categories.
\end{cor}

\begin{proof}
The functor is well-defined since all relations in 
$\CKM$ hold in $\SuSp$. That $\iota^{\infty}_{1}$ is monoidal is clear, 
fullness follows from \fullref{prop-green} and 
faithfulness from \fullref{lem-faithful}. Thus, 
we see that $\CKM$ and 
$\AltSp$ are monoidally equivalent.
\end{proof}

\subsection{Braidings}\label{sub-braid}

We define now a 
\textit{braided} monoidal structure on $\InSp$.

\begin{defn}\label{defn-crossings1}
Define for $k,l\in\Z_{\geq 0}$ an \textit{elementary crossing} 
depending on four cases. The \textit{monochromatic crossings}
(note the different powers of $q$)
\begin{equation}\label{eq-braidedstru1}
\begin{gathered}
\beta_{k,l}^{g} =
\raisebox{-.05cm}{\xy
(0,0)*{
\begin{tikzpicture}[scale=.3]
	\draw [very thick, green, ->] (-1,-1) to (1,1);
	\draw [very thick, green, ->] (-0.25,0.25) to (-1,1);
	\draw [very thick, green] (0.25,-0.25) to (1,-1);
	\node at (-1,-1.5) {\tiny $k$};
	\node at (1,-1.5) {\tiny $l$};
\end{tikzpicture}
};
\endxy}
=
(-1)^{k+kl}q^{k}\!\!\!\!\sum_{\begin{smallmatrix} j_1,j_2\geq 0 \\ j_1-j_2=k-l \end{smallmatrix}} (-q)^{-j_1}
\xy
(0,0)*{
\begin{tikzpicture}[scale=.3]
	\draw [very thick, green, directed=.55] (-2,-4) to (-2,-2);
	\draw [very thick, green, directed=.55] (-2,-2) to (-2,2);
	\draw [very thick, green, directed=.55] (2,-4) to (2,-2);
	\draw [very thick, green, directed=.55] (2,-2) to (2,2);
	\draw [very thick, green, directed=.55] (-2,-2) to (2,-2);
	\draw [very thick, green, directed=.55] (-2,2) to (-2,4);
	\draw [very thick, green, directed=.55] (2,2) to (2,4);
	\draw [very thick, green, rdirected=.55] (-2,2) to (2,2);
	\node at (-2,-4.5) {\tiny $k$};
	\node at (2,-4.5) {\tiny $l$};
	\node at (-2,4.5) {\tiny $l$};
	\node at (2,4.5) {\tiny $k$};
	\node at (-3.5,0) {\tiny $k{-}j_1$};
	\node at (3.5,0) {\tiny $l{+}j_1$};
	\node at (0,-1.25) {\tiny $j_1$};
	\node at (0,2.75) {\tiny $j_2$};
\end{tikzpicture}
};
\endxy
\\
\beta_{k,l}^{r} =
\raisebox{-.05cm}{\xy
(0,0)*{
\begin{tikzpicture}[scale=.3]
	\draw [very thick, mycolor, ->] (-1,-1) to (1,1);
	\draw [very thick, mycolor, ->] (-0.25,0.25) to (-1,1);
	\draw [very thick, mycolor] (0.25,-0.25) to (1,-1);
	\node at (-1,-1.5) {\tiny $k$};
	\node at (1,-1.5) {\tiny $l$};
\end{tikzpicture}
};
\endxy}
=
(-1)^{k}q^{-k}\sum_{\begin{smallmatrix} j_1,j_2\geq 0 \\ j_1-j_2=k-l \end{smallmatrix}} (-q)^{+j_1}
\xy
(0,0)*{
\begin{tikzpicture}[scale=.3]
	\draw [very thick, mycolor, directed=.55] (-2,-4) to (-2,-2);
	\draw [very thick, mycolor, directed=.55] (-2,-2) to (-2,2);
	\draw [very thick, mycolor, directed=.55] (2,-4) to (2,-2);
	\draw [very thick, mycolor, directed=.55] (2,-2) to (2,2);
	\draw [very thick, mycolor, directed=.55] (-2,-2) to (2,-2);
	\draw [very thick, mycolor, directed=.55] (-2,2) to (-2,4);
	\draw [very thick, mycolor, directed=.55] (2,2) to (2,4);
	\draw [very thick, mycolor, rdirected=.55] (-2,2) to (2,2);
	\node at (-2,-4.5) {\tiny $k$};
	\node at (2,-4.5) {\tiny $l$};
	\node at (-2,4.5) {\tiny $l$};
	\node at (2,4.5) {\tiny $k$};
	\node at (-3.5,0) {\tiny $k{-}j_1$};
	\node at (3.5,0) {\tiny $l{+}j_1$};
	\node at (0,-1.25) {\tiny $j_1$};
	\node at (0,2.75) {\tiny $j_2$};
\end{tikzpicture}
};
\endxy
\end{gathered}
\end{equation}
The 
\textit{mixed crossings} are defined 
via explosion of the strand going over:
\begin{equation}\label{eq-mixedbraiding}
\beta_{k,l}^{m} =
\raisebox{-.05cm}{\xy
(0,0)*{
\begin{tikzpicture}[scale=.3]
	\draw [very thick, green, ->] (-1,-1) to (1,1);
	\draw [very thick, mycolor, ->] (-0.25,0.25) to (-1,1);
	\draw [very thick, mycolor] (0.25,-0.25) to (1,-1);
	\node at (-1,-1.5) {\tiny $k$};
	\node at (1,-1.5) {\tiny $l$};
\end{tikzpicture}
};
\endxy}
=
\frac{1}{[k]!}
\raisebox{-.1cm}{\xy
(0,0)*{
\begin{tikzpicture}[scale=.75]
	\draw [very thick, green] (-1,-1) to (-.5,-.5);
	\draw [very thick, green, ->] (.5,.5) to (1,1);
	\draw [very thick, directed=.85] (-.5,-.5) to [out=105, in=165] (.5,.5);
	\draw [very thick, directed=.85] (-.5,-.5) to [out=-15, in=285] (.5,.5);
	\draw [very thick, mycolor, ->] (-.4,.4) to (-1,1);
	\draw [very thick, mycolor] (.4,-.4) to (1,-1);
	\draw [very thick, mycolor] (-.1,.1) to (.1,-.1);
	\node at (-1,-1.25) {\tiny $k$};
	\node at (1,-1.25) {\tiny $l$};
	\node at (-.7,0) {\tiny $1$};
	\node at (.7,0) {\tiny $1$};
	\node at (-.15,-.15) {\tiny \rotatebox{-45}{$\cdots$}};
	\node at (.15,.15) {\tiny \rotatebox{-45}{$\cdots$}};
\end{tikzpicture}
};
\endxy}
\quad\text{and}\quad
\beta_{k,l}^{\tilde m} =
\raisebox{-.05cm}{\xy
(0,0)*{
\begin{tikzpicture}[scale=.3]
	\draw [very thick, mycolor, ->] (-1,-1) to (1,1);
	\draw [very thick, green, ->] (-0.25,0.25) to (-1,1);
	\draw [very thick, green] (0.25,-0.25) to (1,-1);
	\node at (-1,-1.5) {\tiny $k$};
	\node at (1,-1.5) {\tiny $l$};
\end{tikzpicture}
};
\endxy}
=
\frac{1}{[k]!}
\raisebox{-.1cm}{\xy
(0,0)*{
\begin{tikzpicture}[scale=.75]
	\draw [very thick, mycolor] (-1,-1) to (-.5,-.5);
	\draw [very thick, mycolor, ->] (.5,.5) to (1,1);
	\draw [very thick, directed=.85] (-.5,-.5) to [out=105, in=165] (.5,.5);
	\draw [very thick, directed=.85] (-.5,-.5) to [out=-15, in=285] (.5,.5);
	\draw [very thick, green, ->] (-.4,.4) to (-1,1);
	\draw [very thick, green] (.4,-.4) to (1,-1);
	\draw [very thick, green] (-.1,.1) to (.1,-.1);
	\node at (-1,-1.25) {\tiny $k$};
	\node at (1,-1.25) {\tiny $l$};
	\node at (-.7,0) {\tiny $1$};
	\node at (.7,0) {\tiny $1$};
	\node at (-.15,-.15) {\tiny \rotatebox{-45}{$\cdots$}};
	\node at (.15,.15) {\tiny \rotatebox{-45}{$\cdots$}};
\end{tikzpicture}
};
\endxy}
\end{equation}
where the remaining crossings are of the 
form $\beta_{1,l}^r$ or $\beta_{1,l}^g$
respectively.
\end{defn}

\begin{ex}\label{ex-oneoneok}
The case $k=l=1$ is not ambiguous, since we have
\[
\beta_{1,1}^g=q\left(
\xy
(0,0)*{
\begin{tikzpicture}[scale=.3] 
	\draw [very thick, directed=.55] (1,-2.75) to (1,2.5);
	\draw [very thick, directed=.55] (-1,-2.75) to (-1,2.5);
	\node at (-1,3) {\tiny $1$};
	\node at (1,3) {\tiny $1$};
	\node at (-1,-3.15) {\tiny $1$};
	\node at (1,-3.15) {\tiny $1$};
\end{tikzpicture}
};
\endxy
-q^{-1}
\xy
(0,0)*{
\begin{tikzpicture}[scale=.3]
	\draw [very thick, green, directed=.55] (0,-1) to (0,.75);
	\draw [very thick, directed=.55] (0,.75) to [out=30,in=270] (1,2.5);
	\draw [very thick, directed=.55] (0,.75) to [out=150,in=270] (-1,2.5); 
	\draw [very thick, directed=.55] (1,-2.75) to [out=90,in=330] (0,-1);
	\draw [very thick, directed=.55] (-1,-2.75) to [out=90,in=210] (0,-1);
	\node at (-1,3) {\tiny $1$};
	\node at (1,3) {\tiny $1$};
	\node at (-1,-3.15) {\tiny $1$};
	\node at (1,-3.15) {\tiny $1$};
	\node at (-0.5,0) {\tiny $2$};
\end{tikzpicture}
};
\endxy
\right)
\stackrel{\eqref{eq-dumb}}{=}
-q^{-1}\left(
\xy
(0,0)*{
\begin{tikzpicture}[scale=.3] 
	\draw [very thick, directed=.55] (1,-2.75) to (1,2.5);
	\draw [very thick, directed=.55] (-1,-2.75) to (-1,2.5);
	\node at (-1,3) {\tiny $1$};
	\node at (1,3) {\tiny $1$};
	\node at (-1,-3.15) {\tiny $1$};
	\node at (1,-3.15) {\tiny $1$};
\end{tikzpicture}
};
\endxy
-q
\xy
(0,0)*{
\begin{tikzpicture}[scale=.3]
	\draw [very thick, mycolor, directed=.55] (0,-1) to (0,.75);
	\draw [very thick, directed=.55] (0,.75) to [out=30,in=270] (1,2.5);
	\draw [very thick, directed=.55] (0,.75) to [out=150,in=270] (-1,2.5); 
	\draw [very thick, directed=.55] (1,-2.75) to [out=90,in=330] (0,-1);
	\draw [very thick, directed=.55] (-1,-2.75) to [out=90,in=210] (0,-1);
	\node at (-1,3) {\tiny $1$};
	\node at (1,3) {\tiny $1$};
	\node at (-1,-3.15) {\tiny $1$};
	\node at (1,-3.15) {\tiny $1$};
	\node at (-0.5,0) {\tiny $2$};
\end{tikzpicture}
};
\endxy
\right)
=\beta_{1,1}^r,
\]
as a small calculation shows.
\end{ex}

As shorthand notation, we write $\beta^{\Sup}_{k,l}$ where 
$\Sup$ stands for either $g$, $r$, $m$ or $\tilde m$ from now on.
Note that the sums in \eqref{eq-braidedstru1} 
are finite, because 
webs with negative labels are zero.

\begin{lem}(\textbf{Pitchfork relations})\label{lem-pitchfork}
We have
\[
\raisebox{-.01cm}{\xy
(0,0)*{
\begin{tikzpicture}[scale=.75]
	\draw [very thick, green, ->] (-1,-1) to (1,1);
	\draw [very thick, mycolor] (1,-1) to (.5,-.5);
	\draw [very thick, mycolor] (.5,-.5) to (.15,-.15);
	\draw [very thick, mycolor] (-.15,.15) to (-.3,.3);
	\draw [very thick, ->] (-.8,.2) to (-1.3,.7);
	\draw [very thick, ->] (-.2,.8) to (-.7,1.3);
	\draw [very thick] (-.3,.3) to [out=225, in=-45] (-.8,.2);
	\draw [very thick] (-.3,.3) to [out=45, in=-45] (-.2,.8);
	\node at (-1,-1.25) {\tiny $k$};
	\node at (1,-1.25) {\tiny $l$};
	\node at (-1.15,.25) {\tiny $1$};
	\node at (-.25,1.15) {\tiny $1$};
	\node at (-.65,.65) {\tiny \rotatebox{45}{$\cdots$}};
\end{tikzpicture}
};
\endxy}
=
\raisebox{-.01cm}{\xy
(0,0)*{
\begin{tikzpicture}[scale=.75]
	\draw [very thick, green, ->] (-1,-1) to (1,1);
	\draw [very thick, mycolor] (1,-1) to (.5,-.5);
	\draw [very thick, ->] (-.8,.2) to (-1.3,.7);
	\draw [very thick, ->] (-.2,.8) to (-.7,1.3);
	\draw [very thick] (-.45,-.15) to (-.8,.2);
	\draw [very thick] (.15,.45) to (-.2,.8);
	\draw [very thick] (.5,-.5) to [out=225, in=-45] (-.15,-.45);
	\draw [very thick] (.5,-.5) to [out=45, in=-45] (.45,.15);
	\node at (-1,-1.25) {\tiny $k$};
	\node at (1,-1.25) {\tiny $l$};
	\node at (-1.15,.25) {\tiny $1$};
	\node at (-.25,1.15) {\tiny $1$};
	\node at (-.65,.65) {\tiny \rotatebox{45}{$\cdots$}};
\end{tikzpicture}
};
\endxy}
\quad,\quad
\raisebox{-.01cm}{\reflectbox{\xy
(0,0)*{
\begin{tikzpicture}[scale=.75]
    \draw [very thick, green] (-1,-1) to (-.15,-.15);
	\draw [very thick, green, ->] (.15,.15) to (1,1);
	\draw [very thick, mycolor] (1,-1) to (-.3,.3);
	\draw [very thick, ->] (-.8,.2) to (-1.3,.7);
	\draw [very thick, ->] (-.2,.8) to (-.7,1.3);
	\draw [very thick] (-.3,.3) to [out=225, in=-45] (-.8,.2);
	\draw [very thick] (-.3,.3) to [out=45, in=-45] (-.2,.8);
	\node at (-1,-1.25) {\reflectbox{\tiny $l$}};
	\node at (1,-1.25) {\reflectbox{\tiny $k$}};
	\node at (-1.15,.25) {\reflectbox{\tiny $1$}};
	\node at (-.25,1.15) {\reflectbox{\tiny $1$}};
	\node at (-.65,.65) {\reflectbox{\tiny \rotatebox{135}{$\cdots$}}};
\end{tikzpicture}
};
\endxy}}
=
\raisebox{-.01cm}{\reflectbox{\xy
(0,0)*{
\begin{tikzpicture}[scale=.75]
	\draw [very thick, green] (-1,-1) to (-.45,-.45);
	\draw [very thick, green] (-.2,-.2) to (.2,.2);
	\draw [very thick, green, ->] (.45,.45) to (1,1);
	\draw [very thick, mycolor] (1,-1) to (.5,-.5);
	\draw [very thick, ->] (-.8,.2) to (-1.3,.7);
	\draw [very thick, ->] (-.2,.8) to (-.7,1.3);
	\draw [very thick] (-.15,-.45) to (-.8,.2);
	\draw [very thick] (.45,.15) to (-.2,.8);
	\draw [very thick] (.5,-.5) to [out=225, in=-45] (-.15,-.45);
	\draw [very thick] (.5,-.5) to [out=45, in=-45] (.45,.15);
	\node at (-1,-1.25) {\reflectbox{\tiny $l$}};
	\node at (1,-1.25) {\reflectbox{\tiny $k$}};
	\node at (-1.15,.25) {\reflectbox{\tiny $1$}};
	\node at (-.25,1.15) {\reflectbox{\tiny $1$}};
	\node at (-.65,.65) {\tiny \reflectbox{\rotatebox{135}{$\cdots$}}};
\end{tikzpicture}
};
\endxy}}
\]
Similar with exchanged roles of green and red,
for the monochromatic cases and with merges.
\end{lem}

Note that the pitchfork lemma directly implies that \eqref{eq-mixedbraiding} 
could also be done by exploding the edges going underneath instead of the edges 
going over (or exploding both).

\begin{proof}
The pitchfork lemma with only green colored edges follows 
as in \cite[Lemma 5.3]{moy}. 
By symmetry, the arguments go through for the monochromatic red case as well.

The mixed, left-hand equation is easy to verify by the above, since 
we explode the overcrossing edge and we thus, can directly use the 
monochromatic case. 
It remains to prove the mixed, right-hand equation. 
We only need to check the case $k=2$, the case $k\in\Z_{>2}$ then 
follows easily from this case by using \fullref{lem-green}.
We write
\[
\raisebox{-.01cm}{\reflectbox{\xy
(0,0)*{
\begin{tikzpicture}[scale=.75]
    \draw [very thick, green] (-1,-1) to (-.15,-.15);
	\draw [very thick, green, ->] (.15,.15) to (1,1);
	\draw [very thick, mycolor] (1,-1) to (-.3,.3);
	\draw [very thick, ->] (-.8,.2) to (-1.3,.7);
	\draw [very thick, ->] (-.2,.8) to (-.7,1.3);
	\draw [very thick] (-.3,.3) to [out=225, in=-45] (-.8,.2);
	\draw [very thick] (-.3,.3) to [out=45, in=-45] (-.2,.8);
	\node at (-1,-1.25) {\tiny \reflectbox{$l$}};
	\node at (1,-1.25) {\tiny \reflectbox{$2$}};
	\node at (-1.15,.25) {\tiny \reflectbox{$1$}};
	\node at (-.25,1.15) {\tiny \reflectbox{$1$}};
\end{tikzpicture}
};
\endxy}}
=\frac{1}{[2]}
\raisebox{-.01cm}{\reflectbox{\xy
(0,0)*{
\begin{tikzpicture}[scale=.75]
	\draw [very thick, green] (-1,-1) to (-.45,-.45);
	\draw [very thick, green] (-.2,-.2) to (.2,.2);
	\draw [very thick, green, ->] (.45,.45) to (1,1);
	\draw [very thick, mycolor] (1,-1) to (.25,-.25);
	\draw [very thick, mycolor] (-.25,.25) to (-.5,.5);
	\draw [very thick, ->] (-.95,.35) to (-1.3,.7);
	\draw [very thick, ->] (-.35,.95) to (-.7,1.3);
	\draw [very thick] (-.4,-.2) to [out=135, in=225] (-.25,.25);
	\draw [very thick] (.2,.4) to [out=135, in=45] (-.25,.25);
	\draw [very thick] (-.4,-.2) to (-.2,-.4);
	\draw [very thick] (.2,.4) to (.4,.2);
	\draw [very thick] (.25,-.25) to [out=225, in=-45] (-.2,-.4);
	\draw [very thick] (.25,-.25) to [out=45, in=-45] (.4,.2);
	\draw [very thick] (-.5,.5) to [out=225, in=-45] (-.95,.35);
	\draw [very thick] (-.5,.5) to [out=45, in=-45] (-.35,.95);
	\node at (-1,-1.25) {\tiny \reflectbox{$l$}};
	\node at (1,-1.25) {\tiny \reflectbox{$2$}};
	\node at (-1.15,.25) {\tiny \reflectbox{$1$}};
	\node at (-.25,1.15) {\tiny \reflectbox{$1$}};
	\node at (-.25,.5) {\tiny \reflectbox{$2$}};
	\node at (-.25,-.65) {\tiny \reflectbox{$1$}};
	\node at (.65,.25) {\tiny \reflectbox{$1$}};
\end{tikzpicture}
};
\endxy}}
\stackrel{\eqref{eq-dumb}}{=}
\raisebox{-.01cm}{\reflectbox{\xy
(0,0)*{
\begin{tikzpicture}[scale=.75]
	\draw [very thick, green] (-1,-1) to (-.45,-.45);
	\draw [very thick, green] (-.2,-.2) to (.2,.2);
	\draw [very thick, green, ->] (.45,.45) to (1,1);
	\draw [very thick, mycolor] (1,-1) to (.25,-.25);
	\draw [very thick, ->] (-.95,.35) to (-1.3,.7);
	\draw [very thick, ->] (-.35,.95) to (-.7,1.3);
	\draw [very thick] (-.95,.35) to (-.2,-.4);
	\draw [very thick] (-.35,.95) to (.4,.2);
	\draw [very thick] (.25,-.25) to [out=225, in=-45] (-.2,-.4);
	\draw [very thick] (.25,-.25) to [out=45, in=-45] (.4,.2);
	\node at (-1,-1.25) {\tiny \reflectbox{$l$}};
	\node at (1,-1.25) {\tiny \reflectbox{$2$}};
	\node at (-1.15,.25) {\tiny \reflectbox{$1$}};
	\node at (-.25,1.15) {\tiny \reflectbox{$1$}};
\end{tikzpicture}
};
\endxy}}
-
\frac{1}{[2]}
\raisebox{-.01cm}{\reflectbox{\xy
(0,0)*{
\begin{tikzpicture}[scale=.75]
	\draw [very thick, green] (-1,-1) to (-.45,-.45);
	\draw [very thick, green] (-.2,-.2) to (.2,.2);
	\draw [very thick, green, ->] (.45,.45) to (1,1);
	\draw [very thick, mycolor] (1,-1) to (.25,-.25);
	\draw [very thick, green] (-.25,.25) to (-.5,.5);
	\draw [very thick, ->] (-.95,.35) to (-1.3,.7);
	\draw [very thick, ->] (-.35,.95) to (-.7,1.3);
	\draw [very thick] (-.4,-.2) to [out=135, in=225] (-.25,.25);
	\draw [very thick] (.2,.4) to [out=135, in=45] (-.25,.25);
	\draw [very thick] (-.4,-.2) to (-.2,-.4);
	\draw [very thick] (.2,.4) to (.4,.2);
	\draw [very thick] (.25,-.25) to [out=225, in=-45] (-.2,-.4);
	\draw [very thick] (.25,-.25) to [out=45, in=-45] (.4,.2);
	\draw [very thick] (-.5,.5) to [out=225, in=-45] (-.95,.35);
	\draw [very thick] (-.5,.5) to [out=45, in=-45] (-.35,.95);
	\node at (-1,-1.25) {\reflectbox{\tiny $l$}};
	\node at (1,-1.25) {\reflectbox{\tiny $2$}};
	\node at (-1.15,.25) {\reflectbox{\tiny $1$}};
	\node at (-.25,1.15) {\reflectbox{\tiny $1$}};
	\node at (-.25,.5) {\reflectbox{\tiny $2$}};
	\node at (-.25,-.65) {\reflectbox{\tiny $1$}};
	\node at (.65,.25) {\reflectbox{\tiny $1$}};
\end{tikzpicture}
};
\endxy}}
\]
The rightmost diagram is zero by \fullref{lem-green} and the 
monochromatic pitchfork relations. This 
proves the mixed right-hand equation. The other cases are analogous.
\end{proof}

Let $\vec{k}\in X^{L}_{\geq 0}$ be an object in $\InSp$. 
We define for $i=1,\dots,L-1$ the crossing
$\beta_{i}^{\Sup}\one_{\vec{k}}$ to be the corresponding 
elementary crossing $\beta_{k_i,k_{i+1}}^{\Sup}$ between the strands $i$ and $i+1$ 
and the identity elsewhere. Clearly, it suffices to indicate the 
rightmost $\one_{\vec{k}}$ in a sequence of the $\beta_{i}^{\Sup}\one_{\vec{k}}$.

\begin{lem}\label{lem-braidrel}
The crossings $\beta_{i}^{\Sup}\one_{\vec{k}}$
satisfy the braid relations, that is, they are invertible, 
they satisfy the commutation relations
$\beta_{i}^{\Sup}\beta_{j}^{\Sup}\one_{\vec{k}}
=\beta_{j}^{\Sup}
\beta_{i}^{\Sup}\one_{\vec{k}}$ for $|i-j|>2$ and the 
Reidemeister 3 relations 
$\beta_{i}^{\Sup}\beta_{j}^{\Sup}\beta_{i}^{\Sup}\one_{\vec{k}}
=\beta_{j}^{\Sup}\beta_{i}^{\Sup}\beta_{j}^{\Sup}\one_{\vec{k}}$ 
for $|i-j|=1$.
\end{lem}

The inverses $({\beta_{i}^{\Sup}})^{-1}$ are given as 
in \eqref{eq-braidedstru1}, but with $q\rightarrow q^{-1}$. See also \cite[Section 5]{moy}.

\begin{proof}
By \fullref{lem-pitchfork}, since the black case 
can be verified as in \cite[Section 5]{moy}.
\end{proof}

\begin{rem}\label{rem-JWnew}
Let $S_K$ denote the symmetric group on $K$ letters. Moreover,  
let $w\in S_K$ and let $\beta^{\Sup}_w \in \End_{\InSp}(\vec{K})$ 
be the permutation braid associated to $w$ (this is a well-defined 
assignment by \fullref{lem-braidrel}).
Let $\ell(w)$ be the length of $w$.
Following \cite[Chapter 3, Section 2]{kali}, 
one can show that
\[
\cal{CL}^g_K=q^{\frac{K(K-1)}{2}}\frac{1}{[K]!}\sum_{w\in S_K}(-q)^{-\ell(w)}\beta^{\Sup}_w\quad,
\quad
\cal{CL}^r_K=q^{-\frac{K(K-1)}{2}}\frac{1}{[K]!}\sum_{w\in S_K}q^{\ell(w)}\beta^{\Sup}_w.
\]
The factors $q^{\frac{K(K-1)}{2}}$ 
and $q^{-\frac{K(K-1)}{2}}$ come from our conventions for crossings. 
\end{rem}

Define 
$\beta^{\Sup}_{\vec{k},\vec{l}}$ for objects 
$\vec{k}=(k_1,\dots,k_a)$ and $\vec{l}=(l_1,\dots,l_b)$ via
\[
\beta^{\Sup}_{\vec{k},\vec{l}}=
\xy
(0,0)*{
\begin{tikzpicture}[scale=.3]
	\node at (0,0) {\tiny $k_1$};
	\node at (2,0) {$\dots$};
	\node at (4,0) {\tiny $k_a$};
	\node at (6,0) {\tiny $l_1$};
	\node at (8,0) {$\dots$};
	\node at (10,0) {\tiny $l_b$};
	\node at (0,7.5) {\tiny $l_1$};
	\node at (2,7.5) {$\dots$};
	\node at (4,7.5) {\tiny $l_b$};
	\node at (6,7.5) {\tiny $k_1$};
	\node at (8,7.5) {$\dots$};
	\node at (10,7.5) {\tiny $k_a$};
	\draw[very thick, blue, ->] (0,.75) to (6,6.75);
	\draw[very thick, blue, ->] (4,.75) to (10,6.75);
	\draw[very thick, blue] (6,.75) to (5.25,1.5);
	\draw[very thick, blue] ( 4.75,2) to (3.25,3.5);
	\draw[very thick, blue, ->] (2.75,4) to (0,6.75);
	\draw[very thick, blue] (10,.75) to (7.25,3.5) ;
	\draw[very thick, blue] (6.75,4) to (5.35,5.5);
	\draw[very thick, blue, ->] (4.75,6) to (4,6.75);
\end{tikzpicture}
};
\endxy
\in \Hom_{\InSp}(\vec{k}\otimes\vec{l},\vec{l}\otimes\vec{k}),
\]
where blue stands for all suitable color possibilities.

Recall that a \textit{braided monoidal category} (with an 
underlying strict monoidal category) is 
a pair $(\boldsymbol{\cal{C}},\beta^{\boldsymbol{\cal{C}}}_{\cdot,\cdot})$ 
consisting of a monoidal category 
$\boldsymbol{\cal{C}}$ and a collection of natural isomorphisms
$\beta^{\boldsymbol{\cal{C}}}_{\vec{k},\vec{l}}\colon\vec{k}\otimes\vec{l}\to\vec{l}\otimes\vec{k}$ 
such that 
the \textit{hexagon identities} hold for any objects $\vec{k},\vec{l},\vec{m}$ of $\boldsymbol{\cal{C}}$:
\begin{equation}\label{eq-hexid}
\beta^{\boldsymbol{\cal{C}}}_{\vec{k},\vec{l}\otimes\vec{m}}=(\mathrm{id}_{\vec{l}}\otimes\beta^{\boldsymbol{\cal{C}}}_{\vec{k},\vec{m}})\circ(\beta^{\boldsymbol{\cal{C}}}_{\vec{k},\vec{l}}\otimes \mathrm{id}_{\vec{m}})\quad,\quad
\beta^{\boldsymbol{\cal{C}}}_{\vec{k}\otimes\vec{l},\vec{m}}=(\beta^{\boldsymbol{\cal{C}}}_{\vec{k},\vec{m}}\otimes \mathrm{id}_{\vec{l}})\circ(\mathrm{id}_{\vec{k}}\otimes\beta^{\boldsymbol{\cal{C}}}_{\vec{l},\vec{m}}).
\end{equation}

\begin{prop}\label{prop-braiding}
The pair $(\InSp,\beta^{\Sup}_{\cdot,\cdot})$ is a braided monoidal category.
\end{prop}

\begin{proof}
Since $\InSp$ is a monoidal category and the
$\beta^{\Sup}_{\vec{k},\vec{l}}$ are isomorphisms that clearly 
satisfy \eqref{eq-hexid}, we only need to prove that they are natural. 
That is, we need to show that, for each web
$u\in\Hom_{\InSp}(\vec{k},\vec{l})$ and each other object $\vec{m}=(m_1,\dots,m_c)$ of 
$\InSp$, we have 
(we again use blue as a generic color):
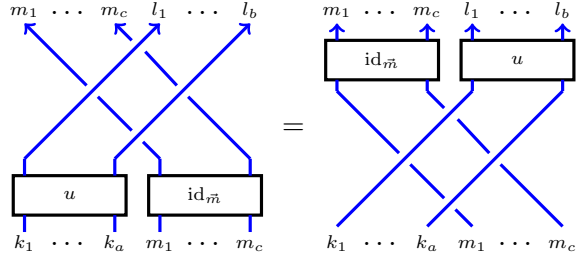
\[
\xy
(0,0)*{
\begin{tikzpicture}[scale=.3]
	\draw[very thick, blue, ->] (0,3.75) to (6,9.75);
	\draw[very thick, blue, ->] (4,3.75) to (10,9.75);
	\draw[very thick, blue] (6,3.75) to (5.25,4.5);
	\draw[very thick, blue] ( 4.75,5) to (3.25,6.5);
	\draw[very thick, blue, ->] (2.75,7) to (0,9.75);
	\draw[very thick, blue] (10,3.75) to (7.25,6.5) ;
	\draw[very thick, blue] (6.75,7) to (5.35,8.5);
	\draw[very thick, blue, ->] (4.75,9) to (4,9.75);
	\draw[very thick] (-.5,1.25) rectangle (4.5,3);
	\draw[very thick] (5.5,1.25) rectangle (10.5,3);
	\draw[very thick, blue] (0,0.5) to (0,1.25);
	\draw[very thick, blue] (4,0.5) to (4,1.25);
	\draw[very thick, blue] (6,0.5) to (6,1.25);
	\draw[very thick, blue] (10,0.5) to (10,1.25);
	\draw[very thick, blue] (0,3) to (0,3.75);
	\draw[very thick, blue] (4,3) to (4,3.75);
	\draw[very thick, blue] (6,3) to (6,3.75);
	\draw[very thick, blue] (10,3) to (10,3.75);
	\node at (0,0) {\tiny $k_1$};
	\node at (2,0) {$\dots$};
	\node at (4,0) {\tiny $k_a$};
	\node at (6,-.1) {\tiny $m_1$};
	\node at (8,0) {$\dots$};
	\node at (10,-.1) {\tiny $m_c$};
	\node at (0,10.25) {\tiny $m_1$};
	\node at (2,10.25) {$\dots$};
	\node at (4,10.25) {\tiny $m_c$};
	\node at (6,10.35) {\tiny $l_1$};
	\node at (8,10.25) {$\dots$};
	\node at (10,10.35) {\tiny $l_b$};
	\node at (2,2.125) {\tiny $u$};
	\node at (8,2.125) {\tiny $\mathrm{id}_{\vec{m}}$};
\end{tikzpicture}
};
\endxy
=
\xy
(0,0)*{
\begin{tikzpicture}[scale=.3]
	\draw[very thick, blue] (0,.75) to (6,6.75);
	\draw[very thick, blue] (4,.75) to (10,6.75);
	\draw[very thick, blue] (6,.75) to (5.25,1.5);
	\draw[very thick, blue] ( 4.75,2) to (3.25,3.5);
	\draw[very thick, blue] (2.75,4) to (0,6.75);
	\draw[very thick, blue] (10,.75) to (7.25,3.5) ;
	\draw[very thick, blue] (6.75,4) to (5.35,5.5);
	\draw[very thick, blue] (4.75,6) to (4,6.75);
	\draw[very thick] (-.5,7.25) rectangle (4.5,9);
	\draw[very thick] (5.5,7.25) rectangle (10.5,9);
	\draw[very thick, blue] (0,6.75) to (0,7.25);
	\draw[very thick, blue] (4,6.75) to (4,7.25);
	\draw[very thick, blue] (6,6.75) to (6,7.25);
	\draw[very thick, blue] (10,6.75) to (10,7.25);
	\draw[very thick, blue, ->] (0,9) to (0,9.75);
	\draw[very thick, blue, ->] (4,9) to (4,9.75);
	\draw[very thick, blue, ->] (6,9) to (6,9.75);
	\draw[very thick, blue, ->] (10,9) to (10,9.75);
	\node at (0,0) {\tiny $k_1$};
	\node at (2,0) {$\dots$};
	\node at (4,0) {\tiny $k_a$};
	\node at (6,-.1) {\tiny $m_1$};
	\node at (8,0) {$\dots$};
	\node at (10,-.1) {\tiny $m_c$};
	\node at (0,10.25) {\tiny $m_1$};
	\node at (2,10.25) {$\dots$};
	\node at (4,10.25) {\tiny $m_c$};
	\node at (6,10.35) {\tiny $l_1$};
	\node at (8,10.25) {$\dots$};
	\node at (10,10.35) {\tiny $l_b$};
	\node at (8,8.125) {\tiny $u$};
	\node at (2,8.125) {\tiny $\mathrm{id}_{\vec{m}}$};
\end{tikzpicture}
};
\endxy
\]
The equality follows from \fullref{lem-pitchfork}. This
proves the statement.
\end{proof}

The braiding $\beta^{\Sup}_{\cdot,\cdot}$ descends to 
the subquotients $\SuSp$, $\AltSp$ and $\SymSp$ and we denote 
all induced braidings also by $\beta^{\Sup}_{\cdot,\cdot}$. 
They are all given by the formulas in \fullref{defn-crossings1}, 
but some diagrams might be zero 
due to \eqref{eq-alt}.

\begin{cor}\label{cor-braiding}
$(\SuSp,\beta^{\Sup}_{\cdot,\cdot})$, 
$(\AltSp,\beta^{\Sup}_{\cdot,\cdot})$ and $(\SymSp,\beta^{\Sup}_{\cdot,\cdot})$, 
with the braiding $\beta^{\Sup}_{\cdot,\cdot}$ induced from $(\InSp,\beta^{\Sup}_{\cdot,\cdot})$, 
are braided monoidal categories.\qed
\end{cor}

Note that $\CKM$ is also a braided monoidal category, see \cite[Corollary 6.2.3]{ckm}. We 
rescale their braiding by multiplying it with $q^{\frac{kl}{N}}$ and we denote the 
resulting braided monoidal category by $(\CKM,\beta^{\Sup}_{\cdot,\cdot})$.
The following corollary is immediate from \fullref{cor-green2}.

\begin{cor}\label{cor-ckmwebs}
The functor 
$\iota_{1}^{\infty}\colon(\CKM,\beta^{\Sup}_{\cdot,\cdot})\to(\SuSp,\beta^{\Sup}_{\cdot,\cdot})$ 
is an inclusion of a full, braided monoidal subcategory.\qed
\end{cor}

\subsection{A collection of diagrammatic idempotents}\label{sub-idem}

Recall that the \textit{Iwahori--Hecke algebra} $H_K(q)$ is the 
$q$-deformation of the symmetric group algebra $\C[S_K]$ on $K$ letters. It is 
generated by $\{H_i\mid s_i\in S_K\}$ for all transpositions $s_i=(i,i+1)\in S_K$ subject to the 
relations
\begin{equation*}
\begin{gathered}
H_i^2=(q-q^{-1})H_i+1,\quad\text{for }i=1,\dots,K-1,
\\
H_iH_j=H_jH_i,\quad\text{for }|i-j|>1\quad,\quad
H_iH_jH_i=H_jH_iH_j,\quad\text{for }|i-j|=1.
\end{gathered}
\end{equation*}

There is a representation $p_K\colon \C_q(B_K)\to H_K(q)$ 
of the group algebra $\C_q(B_K)$ of the braid group $B_K$ with $K$ strands 
given by sending the braid group generators $b_i$
(between the strands $i$ and $i+1$) to $H_i$. 
Thinking of the generators $H_i$ of $H_K(q)$ as 
crossings also makes sense from the perspective of the webs, as the next lemma shows.

\begin{lem}\label{lem-hecke}
Given $K\in \Z_{\geq 0}$, there is an isomorphism of 
$\C_q$-algebras 
\[
\Phi^{\infty}_{q\mathrm{SW}}\colon H_K(q)
\xrightarrow{\cong}\End_{\InSp}(\vec{K}),\quad H_i
\mapsto 
\xy
(0,0)*{
\begin{tikzpicture}[scale=.3]
	\draw [very thick, ->] (-5,-1) to (-5,1);
	\draw [very thick, ->] (-3,-1) to (-3,1);
	\draw [very thick, ->] (-1,-1) to (1,1);
	\draw [very thick, ->] (-0.25,0.25) to (-1,1);
	\draw [very thick] (0.25,-0.25) to (1,-1);
	\draw [very thick, ->] (3,-1) to (3,1);
	\draw [very thick, ->] (5,-1) to (5,1);
	\node at (-5,-1.4) {\tiny $1$};
	\node at (-3,-1.4) {\tiny $1$};
	\node at (-1,-1.4) {\tiny $1$};
	\node at (1,-1.4) {\tiny $1$};
	\node at (3,-1.4) {\tiny $1$};
	\node at (5,-1.4) {\tiny $1$};
	\node at (-5,1.4) {\tiny $1$};
	\node at (-3,1.4) {\tiny $1$};
	\node at (-1,1.4) {\tiny $1$};
	\node at (1,1.4) {\tiny $1$};
	\node at (3,1.4) {\tiny $1$};
	\node at (5,1.4) {\tiny $1$};
	\node at (-4,0) {\tiny $\cdots$};
	\node at (4,0) {\tiny $\cdots$};
\end{tikzpicture}
};
\endxy
\]
\end{lem}

In order to prove \fullref{lem-hecke}, 
which will be used in \fullref{sec-applications}, 
we need \fullref{thm-equi}.

\begin{proof}
A direct computation shows that $\Phi_{q\mathrm{SW}}$ 
is a well-defined $\C_q$-algebra homomorphism. 
In fact, the composite $\Gamma\circ \Phi^{\infty}_{q\mathrm{SW}}$ is the isomorphism induced by quantum Schur--Weyl duality.
To see this, let $V=(\C_q^N)^{\otimes K}$ and 
recall that quantum Schur--Weyl duality states that
\begin{equation}\label{eq-schurweyl}
\Phi^N_{q\mathrm{SW}}\colon H_K(q)\twoheadrightarrow
\End_{\Uun}(V)\quad\text{and}\quad\Phi^N_{q\mathrm{SW}}\colon
H_K(q)\xrightarrow{\cong}\End_{\Uun}(V),\text{ if }N\geq K.
\end{equation}
Here $\Phi^N_{q\mathrm{SW}}$ is the 
$\C_q$-algebra homomorphism induced by 
the action of $H_K(q)$ on the $K$-fold 
tensor product $V$. By \fullref{thm-equi}, 
we will get an isomorphism 
$H_K(q)\cong\End_{\SuSp}(\vec{K})$, if $N\geq K$. By 
using \fullref{prop-green}, 
there is a basis of $\End_{\SuSp}(\vec{K})$ 
for $N\geq K$ given by webs with only 
black edges or green edges with labels at most $K$. 
Since $K$ is fixed, a direct 
comparison shows that $\Phi^{\infty}_{q\mathrm{SW}}$ 
has to be an isomorphism as well.
\end{proof}

Let $K\in\Z_{\geq 0}$ and let $\Lambda^+(K)$ denote the set of all 
\textit{Young diagrams with} $K$ \textit{nodes}, e.g.\
\[
\lambda=(4,3,1,1)\in\Lambda^+(9)\leftrightsquigarrow\lambda=\xy(0,0)*{\begin{Young} & & &\cr & &\cr \cr\cr\end{Young}}\endxy\;,\;
\lambda^{T}=(4,2,2,1)\in\Lambda^+(9)
\leftrightsquigarrow\lambda^{\T}=\xy(0,0)*{\begin{Young} & & &\cr & \cr &\cr\cr\end{Young}}\endxy,
\]
where we use the English notation for our Young diagrams. 
Here we have also displayed the \textit{transpose} Young diagram $\lambda^{\T}$ of $\lambda$.
Next, the following definition is motivated by \cite{gyoja} and \cite{am}. 
(It is best explained via examples, cf. \fullref{ex-idemHecke} 
and \fullref{ex-needsalabel}, which the reader might want to check while reading the definition.)

\begin{defn}(\textbf{Gyoja--Aiston idempotents})
\label{defn-idem}
Given $\lambda\in\Lambda^+(K)$, we associate to it 
a primitive \textit{idempotent} $e_q(\lambda)\in\End_{\InSp}(\vec{K})$. 
First we define two idempotents as tensor products of green or red clasps:
\[
e_{\mathrm{col}}(\lambda)=\cal{CL}^g_{\mathrm{col}_1}\otimes\cdots\otimes
\cal{CL}^g_{\mathrm{col}_c},
\quad\quad
e_{\mathrm{row}}(\lambda)=\cal{CL}^r_{\mathrm{row}_1}\otimes\cdots\otimes
\cal{CL}^r_{\mathrm{row}_r},
\]
where $c$ and $r$ are the number of columns and rows of $\lambda$ respectively, and 
$\mathrm{col}_i$ and $\mathrm{row}_i$ denote the number of nodes 
in the $i^{\text{th}}$ column and row.

Denote by $T^{\rightarrow}_{\lambda}$ 
and by $T^{\downarrow}_{\lambda}$ the two 
tableaux of shape $\lambda$ obtained by filling the 
numbers $1,\dots,K$ into the
Young diagram $\lambda$ in order: $\rightarrow$ means rows before columns and 
$\downarrow$ means columns before rows (both from left to right). 
Pick any shortest presentation of the 
permutation $w(\lambda)\in S_K$ permuting 
$T^{\rightarrow}_{\lambda}$ 
to $T^{\downarrow}_{\lambda}$.
Then we define the \textit{quasi-idempotent associated to} $\lambda$ via
\[
\tilde e_q(\lambda)=
e_{\mathrm{col}}(\lambda)\circ\beta^{\Sup}_{w(\lambda)}
\circ e_{\mathrm{row}}(\lambda)\circ(\beta^{\Sup}_{w(\lambda)})^{-1}.
\]
By \cite[Theorem 4.7]{am} (and the fact that their definition agrees with 
ours by \fullref{lem-hecke} and \fullref{rem-JWnew}), 
there exists a non-zero scalar $a(\lambda)\in\C_q$ 
such that $\tilde e_q(\lambda)^2=a(\lambda)\tilde e_q(\lambda)$. 
Thus, we define 
the \textit{idempotent associated to} $\lambda$ 
to be $e_q(\lambda)=\frac{1}{a(\lambda)}\tilde e_q(\lambda)$. 
\end{defn}

These 
idempotents are primitive and orthogonal 
by \cite[Theorem 4.5]{gyoja} and \cite[Theorem 4.7]{am}.

\begin{ex}\label{ex-idemHecke}
If $K=2$, then there are two primitive idempotents, namely
\[
e_q\left(\sign\right)
=
\xymatrix{
\dfrac{1}{[2]}
\xy
(0,0)*{
\begin{tikzpicture}[scale=.3]
	\draw [very thick, green, directed=.55] (0,-1) to (0,.75);
	\draw [very thick, directed=.55] (0,.75) to [out=30,in=270] (1,2.5);
	\draw [very thick, directed=.55] (0,.75) to [out=150,in=270] (-1,2.5); 
	\draw [very thick, directed=.55] (1,-2.75) to [out=90,in=330] (0,-1);
	\draw [very thick, directed=.55] (-1,-2.75) to [out=90,in=210] (0,-1);
	\node at (-1,3) {\tiny $1$};
	\node at (1,3) {\tiny $1$};
	\node at (-1,-3.15) {\tiny $1$};
	\node at (1,-3.15) {\tiny $1$};
	\node at (-0.6,0) {\tiny $2$};
\end{tikzpicture}
};
\endxy\ar@<2pt>@{->}[rr]^/-.5cm/{\text{green to red}} & & 
\dfrac{1}{[2]}
\xy
(0,0)*{
\begin{tikzpicture}[scale=.3]
	\draw [very thick, mycolor, directed=.55] (0,-1) to (0,.75);
	\draw [very thick, directed=.55] (0,.75) to [out=30,in=270] (1,2.5);
	\draw [very thick, directed=.55] (0,.75) to [out=150,in=270] (-1,2.5); 
	\draw [very thick, directed=.55] (1,-2.75) to [out=90,in=330] (0,-1);
	\draw [very thick, directed=.55] (-1,-2.75) to [out=90,in=210] (0,-1);
	\node at (-1,3) {\tiny $1$};
	\node at (1,3) {\tiny $1$};
	\node at (-1,-3.15) {\tiny $1$};
	\node at (1,-3.15) {\tiny $1$};
	\node at (-0.6,0) {\tiny $2$};
\end{tikzpicture}
};
\endxy
=e_q\left(\triv\right)
\ar@<2pt>@{->}[ll]^/.5cm/{\text{red to green}}
}
\] 
Note that $a(\lambda)=1$ for only one column or only one row Young diagrams $\lambda$. 
\end{ex}

\begin{lem}\label{lem-symmidem}
Exchanging green and red sends $e_q(\lambda)$ 
to $e_q(\lambda^{\T})$ modulo a commutator.
\end{lem}

\begin{proof} 
Note that $e_{\mathrm{col}}(\lambda)$ and $e_{\mathrm{row}}(\lambda)$ differ from $e_{\mathrm{row}}(\lambda^{\T})$ and $e_{\mathrm{col}}(\lambda^{\T})$ respectively only in exchanging the colors green and red. On black crossings the green--red symmetry acts by $\beta^{\Sup}_{1,1}\mapsto -(\beta^{\Sup}_{1,1})^{-1}$, on permutation braids as $\beta^{\Sup}_{w}\mapsto (-1)^{\ell(w)} (\beta^{\Sup}_{w^{-1}})^{-1}$ and on the quasi-idempotent $\tilde e_q(\lambda)$ as:
\begin{align*}
\tilde e_q(\lambda) =
e_{\mathrm{col}}(\lambda)\circ\beta^{\Sup}_{w(\lambda)}
\circ e_{\mathrm{row}}(\lambda)\circ(\beta^{\Sup}_{w(\lambda)})^{-1}&\mapsto
e_{\mathrm{row}}(\lambda^{\T})\circ (\beta^{\Sup}_{w(\lambda)^{-1}})^{-1} 
\circ e_{\mathrm{col}}(\lambda^{\T})\circ \beta^{\Sup}_{w(\lambda)^{-1}}
\\
 &=
e_{\mathrm{row}}(\lambda^{\T})\circ (\beta^{\Sup}_{w(\lambda^{\T})})^{-1} 
\circ e_{\mathrm{col}}(\lambda^{\T})\circ \beta^{\Sup}_{w(\lambda^{\T})}.
\end{align*}
In the first line, the signs from the crossing inversions cancel, and in the second line we use $w(\lambda)^{-1}=w(\lambda^{\T})$. The result agrees with $\tilde e_q(\lambda^{\T})$ up to a commutator. This proves the statement of the lemma for the quasi-idempotents. Applying the green--red symmetry to both sides of the equation $\tilde e_q(\lambda)^2=a(\lambda)\tilde e_q(\lambda)$ shows that $a(\lambda)=a(\lambda^{\T})$ and the lemma follows.
\end{proof}

\begin{ex}\label{ex-needsalabel}
For $\lambda=(3,1)\in\Lambda^+(4)$, we have
\[
\lambda=
\xy(0,0)*{\begin{Young} & & \cr \cr\end{Young}}\endxy
\quad,\quad
T^{\rightarrow}_{\lambda}=
\xy(0,0)*{\begin{Young} 1& 2& 3\cr 4\cr\end{Young}}\endxy
\quad,\quad
T^{\downarrow}_{\lambda}=
\xy(0,0)*{\begin{Young} 1& 3& 4\cr 2\cr\end{Young}}\endxy.
\]
Thus, $w=(243)=(23)(34)\in S_4$ permutes $T^{\rightarrow}_{\lambda}$ to 
$T^{\downarrow}_{\lambda}$. Then
\[
\tilde e_q(\lambda)=
\xy
(0,0)*{
\begin{tikzpicture}[scale=.3]
	\draw [very thick, ->] (-2,5) to (-2,6);
	\draw [very thick, ->] (0,5) to (0,6);
	\draw [very thick] (0,2) to (0,3);
	\draw [very thick, ->] (2,2) to (2,6);
	\draw [very thick, ->] (4,2) to (4,6);
	\draw [very thick] (-2,-1) to (-2,3);
	\draw [very thick] (0,-1) to (0,0);
	\draw [very thick] (2,-1) to (2,0);
	\draw [very thick] (4,-4) to (4,0);
	\draw [very thick] (-2,-7) to (-2,-3);
	\draw [very thick] (0,-4) to (0,-3);
	\draw [very thick] (2,-4) to (2,-3);
	\draw [very thick] (0,-7) to (0,-6);
	\draw [very thick] (2,-7) to (2,-6);
	\draw [very thick] (4,-7) to (4,-6);
	\draw [very thick, green] (-2.3,3) rectangle (0.3,5);
	\draw [very thick] (-0.3,0) rectangle (4.3,2);
	\draw [very thick, mycolor] (-2.3,-3) rectangle (2.3,-1);
	\draw [very thick] (-0.3,-6) rectangle (4.3,-4);
	\node at (-1,4) {\tiny $\cal{CL}^g_2$};
	\node at (2,1) {\tiny $\beta^{\Sup}_{w(\lambda)}$};
	\node at (0,-2) {\tiny $\cal{CL}^r_3$};
	\node at (2,-5) {\tiny $(\beta^{\Sup}_{w(\lambda)})^{-1}$};
	\node at (-2,-7.45) {\tiny $1$};
	\node at (0,-7.45) {\tiny $1$};
	\node at (2,-7.45) {\tiny $1$};
	\node at (4,-7.45) {\tiny $1$};
	\node at (-2,6.45) {\tiny $1$};
	\node at (0,6.45) {\tiny $1$};
	\node at (2,6.45) {\tiny $1$};
	\node at (4,6.45) {\tiny $1$};
\end{tikzpicture}
};
\endxy
\xleftrightarrow{\text{green}\leftrightarrow\text{red}}
\xy
(0,0)*{
\begin{tikzpicture}[scale=.3]
	\draw [very thick, ->] (-2,5) to (-2,6);
	\draw [very thick, ->] (0,5) to (0,6);
	\draw [very thick] (0,2) to (0,3);
	\draw [very thick, ->] (2,2) to (2,6);
	\draw [very thick, ->] (4,2) to (4,6);
	\draw [very thick] (-2,-1) to (-2,3);
	\draw [very thick] (0,-1) to (0,0);
	\draw [very thick] (2,-1) to (2,0);
	\draw [very thick] (4,-4) to (4,0);
	\draw [very thick] (-2,-7) to (-2,-3);
	\draw [very thick] (0,-4) to (0,-3);
	\draw [very thick] (2,-4) to (2,-3);
	\draw [very thick] (0,-7) to (0,-6);
	\draw [very thick] (2,-7) to (2,-6);
	\draw [very thick] (4,-7) to (4,-6);
	\draw [very thick, mycolor] (-2.3,3) rectangle (0.3,5);
	\draw [very thick] (-0.3,0) rectangle (4.3,2);
	\draw [very thick, green] (-2.3,-3) rectangle (2.3,-1);
	\draw [very thick] (-0.3,-6) rectangle (4.3,-4);
	\node at (-1,4) {\tiny $\cal{CL}^r_2$};
	\node at (2,1) {\tiny $\scalebox{.85}{$(\beta^{\Sup}_{w(\lambda)^{-1}})^{-1}$}$};
	\node at (0,-2) {\tiny $\cal{CL}^g_3$};
	\node at (2,-5) {\tiny $\beta^{\Sup}_{w(\lambda)^{-1}}$};
	\node at (-2,-7.45) {\tiny $1$};
	\node at (0,-7.45) {\tiny $1$};
	\node at (2,-7.45) {\tiny $1$};
	\node at (4,-7.45) {\tiny $1$};
	\node at (-2,6.45) {\tiny $1$};
	\node at (0,6.45) {\tiny $1$};
	\node at (2,6.45) {\tiny $1$};
	\node at (4,6.45) {\tiny $1$};
\end{tikzpicture}
};
\endxy
\equiv_{\mathrm{tr}}
\xy
(0,0)*{
\begin{tikzpicture}[scale=.3]
	\draw [very thick, ->] (-2,-1) to (-2,0);
	\draw [very thick, ->] (0,-1) to (0,0);
	\draw [very thick, ->] (2,-1) to (2,0);
	\draw [very thick, ->] (4,-4) to (4,0);
	\draw [very thick] (-2,-7) to (-2,-3);
	\draw [very thick] (0,-4) to (0,-3);
	\draw [very thick] (2,-4) to (2,-3);
	\draw [very thick] (0,-7) to (0,-6);
	\draw [very thick] (2,-10) to (2,-6);
	\draw [very thick] (4,-10) to (4,-6);
	\draw [very thick] (-2,-13) to (-2,-9);
	\draw [very thick] (0,-10) to (0,-9);
	\draw [very thick] (0,-13) to (0,-12);
	\draw [very thick] (2,-13) to (2,-12);
	\draw [very thick] (4,-13) to (4,-12);
	\draw [very thick, green] (-2.3,-3) rectangle (2.3,-1);
	\draw [very thick] (-0.3,-6) rectangle (4.3,-4);
	\draw [very thick, mycolor] (-2.3,-9) rectangle (0.3,-7);
	\draw [very thick] (-0.3,-12) rectangle (4.3,-10);
	\node at (-1,-8) {\tiny $\cal{CL}^r_2$};
	\node at (2,-11) {\tiny $\scalebox{.9}{$(\beta^{\Sup}_{w(\lambda^{\T})})^{-1}$}$};
	\node at (0,-2) {\tiny $\cal{CL}^g_3$};
	\node at (2,-5) {\tiny $\beta^{\Sup}_{w(\lambda^{\T})}$};
	\node at (-2,-13.45) {\tiny $1$};
	\node at (0,-13.45) {\tiny $1$};
	\node at (2,-13.45) {\tiny $1$};
	\node at (4,-13.45) {\tiny $1$};
	\node at (-2,0.45) {\tiny $1$};
	\node at (0,0.45) {\tiny $1$};
	\node at (2,0.45) {\tiny $1$};
	\node at (4,0.45) {\tiny $1$};
\end{tikzpicture}
};
\endxy
=
\tilde e_q(\lambda^{\T}).
\]
Here $\equiv_{\mathrm{tr}}$ means equal modulo a commutator and the scaling factor in this case is $a(\lambda)=\frac{[4]}{[2][3]}=a(\lambda^{\T})$.
\end{ex}

\begin{rem}\label{re-young}
For $N\geq K$, the $H_K(q)$-module $(\C_q^N)^{\otimes K}$ decomposes into 
\[
\bigoplus_{\lambda\in \Lambda^+(K)}(S^{\lambda})^{\oplus m_{\lambda}}
\]
where the $S^{\lambda}$ are the irreducible \textit{Specht modules} for $H_K(q)$ and $m_{\lambda}$ are their multiplicities. The primitive idempotents $e_q(\lambda)$ from \fullref{defn-idem} are quantizations of Young symmetrizers that project onto $S^{\lambda}$. Note that a braid-conjugate of $e_q(\lambda)$
might project onto a different copy of $S^{\lambda}$ in the above decomposition. 
\end{rem}
%
%
%%%%%%%%%%%%%%%%%%%%%%%%%%%%%%%%%%%%%%%%
%%%                                  %%%
%%%       the proof                  %%%
%%%                                  %%%
%%%%%%%%%%%%%%%%%%%%%%%%%%%%%%%%%%%%%%%%
\section{Proofs of the diagrammatic presentations}\label{sec-proof}
%%%%%%%%%%%%%%%%%%%%%%%%%%%%%%%%%%%%%%%%%%%%%%%
This section contains the proof of our main theorems.

\subsection{Super \texorpdfstring{$q$}{q}-Howe duality}\label{sub-superHowe}

Let $m,n\in\Z_{\geq 0}$. We start by recalling the \textit{quantum general linear 
superalgebra} $\Us$ and its \textit{idempotented form} $\Usd$. We 
follow the conventions used in \cite{zh}, but adapt Zhang's notation 
to be closer to the one from \cite{ckm}.

To this end, recall that the $\glmn$-weight 
lattice is isomorphic to $\Z^{m+n}$ and we denote 
the $\glmn$-weights usually by vectors $\vec{k}=(k_1,\dots,k_m,k_{m+1},\dots,k_{m+n})$. For 
$\mathbb{I}=\mathbb{I}_0\cup \mathbb{I}_1$  
with $\mathbb{I}_0=\{1,\dots,m\}$ (even part) and 
$\mathbb{I}_1=\{m+1,\dots,m+n\}$ (odd part) define
\[
|i|=\begin{cases}
0,&\text{if }i\in \mathbb{I}_0=\{1,\dots,m\},\\
1,&\text{if }i\in \mathbb{I}_1=\{m+1,\dots,m+n\}.\\
\end{cases}
\]
The notation $|\cdot|$ means the \textit{super} degree (which is a $\Z/2$-degree).
We use a similar notation for all $\Z/2$-graded spaces where we, by convention, 
always consider degrees modulo $2$ in the following. 
Moreover, let 
$\epsilon_i=(0,\dots,0,1,0,\dots,0)\in \Z^{m+n}$, with 
$1$ being in the $i^{\text{th}}$ 
coordinate, and denote by $\alpha_i=\epsilon_i-\epsilon_{i+1}
=(0,\dots,1,-1,\dots,0)\in\Z^{m+n}$ for 
$i\in\mathbb{I}\!-\!\{m+n\}$ the $i^{\text{th}}$ simple root. Recall that the \textit{super} Euclidean 
inner product on $\Z^{m+n}$ is given by  
$(\epsilon_i,\epsilon_j)_{\Su}=(-1)^{|i|}\delta_{i,j}$.

\begin{defn}\label{defn-glmn} 
Let $m,n\in\Z_{\geq 0}$. The \textit{quantum general linear 
superalgebra} 
$\Us$ is 
the associative, $\Z/2$-graded, unital $\C_q$-algebra 
generated by $L^{\pm 1}_i$, for $i\in\mathbb{I}$, 
and $F_{i}$ and $E_i$, for $i\in\mathbb{I}\!-\!\{m+n\}$, subject to 
the \textit{non-super} 
relations 
\begin{gather*}
L_{i}L_{j}=L_{j}L_{i},\quad 
L_iL_i^{-1}=L_i^{-1}L_i=1,\quad
L_{i}F_{j}=q^{- (\epsilon_{i},\alpha_{j})_{\Su}}F_{j}L_{i},\quad
L_{i}E_{j}=q^{(\epsilon_{i},\alpha_{j})_{\Su}}E_{j}L_{i},
\\
E_{i}F_{j} - F_{j}E_{i} = 
(-1)^{|i|}\delta_{i,j}\dfrac{L_{i}L_{i+1}^{-1}-L_{i}^{-1}L_{i+1}}{q-q^{-1}},\text{ if } i\neq m,
\\
[2]F_{i}F_{j}F_{i}=F_{i}^2F_{j}+F_{j}F_{i}^2,
\quad\text{if } |i-j|=1, \; i\neq m, \qquad
F_{i}F_{j}-F_{j}F_{i}=0, \quad\text{if }|i-j|>1,
\\
[2]E_{i}E_{j}E_{i}=E_{i}^2E_{j}+E_{j}E_{i}^2,
\quad\text{if } |i-j|=1, \; i\neq m, \qquad
E_{i}E_{j}-E_{j}E_{i}=0,\quad\text{if }|i-j|>1,
\end{gather*}
(for suitable $i,j\in\mathbb{I}$) and the \textit{super relations}
\begin{gather*}
\begin{aligned}
F_m^2=0=E_m^2,&\quad E_{m}F_{m} + F_{m}E_{m} = 
\dfrac{L_{m}L_{m+1}^{-1}-L_{m}^{-1}L_{m+1}}{q-q^{-1}},
\\
[2]F_{m}F_{m+1}F_{m-1}F_{m}=&
F_{m}F_{m+1}F_{m}F_{m-1}+F_{m-1}F_{m}F_{m+1}F_{m}
\\
&+F_{m+1}F_{m}F_{m-1}F_{m}+F_{m}F_{m-1}F_{m}F_{m+1},
\\
[2]E_{m}E_{m+1}E_{m-1}E_{m}=&
E_{m}E_{m+1}E_{m}E_{m-1}+E_{m-1}E_{m}E_{m+1}E_{m}\\
&+E_{m+1}E_{m}E_{m-1}E_{m}+E_{m}E_{m-1}E_{m}E_{m+1}.
\end{aligned}
\end{gather*}
Also, $|L_i|=0$ for $i\in\mathbb{I}$, 
$|F_i|=|E_i|=0$ for $i\in\mathbb{I}-\{m\}$ and $|F_m|=|E_m|=1$.
\end{defn}

We recover 
$\Uq(\gln)$ by setting $m=N$ and $n=0$. We write 
$\mathbb{I}_N=\{1,\dots,N\}$ in the following to distinguish it from 
$\mathbb{I}$ as above. Note that $\Uq(\gln)$ 
is concentrated in degree $0$.

The algebra $\Us$ is a $\Z/2$-graded Hopf algebra with coproduct $\Delta$, 
antipode $S$ and the counit $\varepsilon$ given by
\begin{gather*}
\Delta(F_i)=F_i\otimes 1+L^{-1}_iL_{i+1}\otimes F_i,\quad
\Delta(E_i)=E_i\otimes L_iL_{i+1}^{-1}+1\otimes E_i,\quad\Delta(L_i)=L_i\otimes L_i,\\
S(F_i)=-L_iL_{i+1}^{-1}F_i,\; S(E_i)=-E_iL^{-1}_iL_{i+1},\; S(L_i)=L_i^{-1},\quad\varepsilon(F_i)=\varepsilon(E_i)=0,\; \varepsilon(L_i)=1.
\end{gather*}

In the spirit of Lusztig \cite[Chapter 23]{lus}, we now
adjoin for all $\vec{k}\in\Z^{m+n}$ 
idempotents $1_{\vec{k}}$ of super degree $|1_{\vec{k}}|=0$ to $\Us$.
Denote by $I$ the ideal generated by
\begin{gather*}
\begin{aligned}
1_{\vec{k}}1_{\vec{l}} = \delta_{\vec{k},\vec{l}}1_{\vec{k}},\quad &
\quad L_i1_{\vec{k}} = q^{k_i(\epsilon_i,\epsilon_i)_{\Su}}1_{\vec{k}},   
\\
1_{\vec{k}-\alpha_i}F_{i}1_{\vec{k}}=F_{i}1_{\vec{k}} = 1_{\vec{k}-\alpha_i}F_{i},\quad &\quad 1_{\vec{k}+\alpha_i}E_{i}1_{\vec{k}}=E_{i}1_{\vec{k}} = 1_{\vec{k}+\alpha_i}E_{i}.
\end{aligned}
\end{gather*}

\begin{defn}\label{defn-blm} Define by
\[
\Usd=({\textstyle\bigoplus_{\vec{k},\vec{l}\in\Z^{m+n}}1_{\vec{l}}}\,\Us 1_{\vec{k}})/I
\]
the 
\textit{idempotented} quantum general linear 
superalgebra.
\end{defn}

\begin{rem}\label{rem-relations-in-blm}
One can view $\Ud(\glmn)$ as generated by the \textit{divided powers}
\[
F^{(j)}_i=\frac{F^j_i}{[j]!}\quad\text{ and }\quad E^{(j)}_i=\frac{E^j_i}{[j]!},\quad\text{for }i\in\mathbb{I}\!-\!\{m+n\}.
\] 
This allows the definition of an integral 
version of $\Ud(\glmn)$. For simplicity, we work over $\C_q$ in this paper and we do not consider the 
integral version.

The relations in $\Usd$ are obtained from the relations 
of $\Us$. For convenience we list the new versions of the 
\textit{super} relations:
\begin{equation}\label{eq-super1}
\begin{gathered}
F_m^2\one_{\vec{k}}=0=E_m^2\one_{\vec{k}},\quad\quad E_mF_m\one_{\vec{k}}+F_mE_m\one_{\vec{k}}
=[k_m+k_{m+1}]\one_{\vec{k}},\\
\begin{aligned}
[2]F_{m}F_{m+1}F_{m-1}F_{m}\one_{\vec{k}}=&
F_{m}F_{m+1}F_{m}F_{m-1}\one_{\vec{k}}+F_{m-1}F_{m}F_{m+1}F_{m}\one_{\vec{k}}
\\
&+F_{m+1}F_{m}F_{m-1}F_{m}\one_{\vec{k}}+F_{m}F_{m-1}F_{m}F_{m+1}\one_{\vec{k}},
\end{aligned}
\end{gathered}
\end{equation}
the second of which we call the \textit{super commutation relation} 
(the third type of relation holds for $E$'s as well).

It is convenient for us hereinafter to view $\Usd$ as a category whose 
objects are the $\glmn$-weights 
$\vec{k}\in\Z^{m+n}$ and $\Hom_{\Usd}(\vec{k},\vec{l}) = 1_{\vec{l}}\Usd 1_{\vec{k}}$.
\end{rem}

Recall that the vector representation $\C_q^{m|n}$ of $\Us$ has a basis given by
$\{x_i\mid i\in\mathbb{I}\}$ with super degrees $|x_i|=|i|$ for $i\in\mathbb{I}$, 
where the $\Us$-action is defined via
\[
F_i(x_j)=\begin{cases}x_{j+1}, &\text{if }i=j,\\0, &\text{otherwise,}\end{cases}\quad
E_i(x_j)=\begin{cases}x_{j-1}, &\text{if }i=j-1,\\0, &\text{otherwise,}\end{cases}\quad
L_i(x_j)=q^{(\epsilon_i,\epsilon_j)_{\Su}}x_{j}.
\]

We need to consider the \textit{quantum exterior superalgebra} 
$\bV_q^{\bullet}(\C_q^{m|n}\otimes \C_q^N)$. 
Recall that a vector space $V=V_0\oplus V_1$ with a $\Z/2$-grading is 
called a \textit{super} vector space. Here $V_0$ and $V_1$ are its degree $0$ and $1$ parts. These graded parts of $\C_q^{m|n}$ have bases given by $\{x_i\mid i\in\mathbb{I}_0\}$ and $\{x_{i}\mid i\in\mathbb{I}_1\}$ respectively. In contrast, $\C_q^{N}=(\C_q^{N})_0$ is concentrated in degree zero and we denote its basis 
by $\{y_j\mid j\in\mathbb{I}_N\}$.
Additionally, the tensor product $V\otimes W$ of 
two super vector spaces $V$ and $W$ is a super vector space with $v\otimes w$ of 
degree $|v|+|w|$ for two homogeneous elements $v$ and $w$. Specifically, 
$\C_q^{m|n}\otimes\C_q^{N}$ is a super vector space with 
$(\C_q^{m|n}\otimes\C_q^{N})_0$ spanned by 
$\{z_{ij}=x_i\otimes y_j\mid i\in\mathbb{I}_0,\;j\in\mathbb{I}_N\}$ and 
$(\C_q^{m|n}\otimes\C_q^{N})_1$ spanned by 
$\{z_{ij}=x_i\otimes y_j\mid i\in\mathbb{I}_1,\;j\in\mathbb{I}_N\}$. Here $|z_{ij}|=|i|$. 
Note that $(\C_q^{m|n}\otimes\C_q^{N})^{\otimes K}$ 
is a $\Z/2$-graded $\Us\otimes\Uq(\gln)$-module for all $K\in\Z_{\geq 0}$ by using the 
Hopf algebras structures of $\Us$ and $\Uq(\gln)$.

We denote by $\Sym^2_q(\C_q^{m|n}\otimes\C_q^{N})$ 
the \textit{second symmetric super power} as in \cite[(4.1)]{qs}, 
but with $q$ inverted in their formulas.
Armed with this notation, we define the \textit{quantum exterior superalgebra}
\[
\bV^{\bullet}_q(\C_q^{m|n}\otimes\C_q^{N})=
T(\C_q^{m|n}\otimes\C_q^{N})/\Sym^2_q(\C_q^{m|n}\otimes\C_q^{N}),
\]
where $T(\C_q^{m|n}\otimes\C_q^{N})=\bigoplus_{K\in\Z_{\geq 0}}(\C_q^{m|n}\otimes\C_q^{N})^{\otimes K}$ 
denotes the \textit{super 
tensor algebra} of $\C_q^{m|n}\otimes\C_q^{N}$. This is a 
$\Us\otimes\Uq(\gln)$-module and decompose as
\[
\bV^{\bullet}_q(\C_q^{m|n}\otimes\C_q^{N})\cong{\textstyle\bigoplus_{K\in\Z_{\geq 0}}}\bV^K_q(\C_q^{m|n}\otimes\C_q^{N}).
\]
The space $\bV^K_q(\C_q^{m|n}\otimes\C_q^{N})$ is called 
the \textit{degree} $K$ \textit{part} of $\bV^{\bullet}_q(\C_q^{m|n}\otimes\C_q^{N})$.

\begin{rem}\label{rem-others}
We can recover
the degree $K$ part of the quantum exterior 
algebra $\bV^K_q(\C_q^{m}\otimes\C_q^{N})$ by setting $n=0$ 
and, by \cite[Remark 2.1]{sar}, 
the degree $K$ part of the quantum symmetric 
algebra $\Sym^K_q(\C_q^{n}\otimes\C_q^{N})$ by setting $m=0$. 
These were originally
defined in \cite[Definition 2.7]{bz1} and used in \cite[Section 4.2]{ckm} and 
in \cite[Section 2.1]{rt} to study 
skew and symmetric $q$-Howe duality.
\end{rem}

\begin{ex}\label{ex-altsuper2}
Write $z_{\vec{ij}}=z_{i_1j_1}\otimes \cdots \otimes z_{i_Kj_K}$ and 
$z_{i_kj_k}\preceq z_{i_{k+1}j_{k+1}}$ for the anti-lexicographical order 
on the indices of the $z_{ij}$. Then
$\bV^{K}_q(\C_q^{m|n}\otimes\C_q^{N})$ 
has a basis given by (cf. \cite[Lemma 4.1]{qs})
\begin{equation}\label{eq-basis}
\left\{z_{\vec{ij}}
\middle|
z_{i_kj_k}\preceq z_{i_{k+1}j_{k+1}},\;\;
\begin{matrix} 1\leq i_1\leq\cdots\leq i_K\leq m+n,\\
1\leq j_1\leq\cdots\leq j_K\leq N,
\end{matrix}\;\text{ and }\;
\begin{matrix} i_k=i_{k+1},\\
j_k=j_{k+1},
\end{matrix}\Rightarrow
|i_k|=1
\right\}.
\end{equation} 
By setting $m=1$ and $n=0$, we obtain the (usual) basis for 
$\bV^{K}_q\C_q^{N}$ of the form
\begin{equation}\label{eq-basis2}
\left\{y_{i_1}\otimes \cdots \otimes y_{i_K}
\mid
1\leq y_1<\cdots<y_K\leq N
\right\},
\end{equation}
while setting $m=0$ and $n=1$ gives the (usual) 
basis for $\Sym^{K}_q\C_q^{N}$ of the form
\begin{equation}\label{eq-basis3}
\left\{y_{i_1}\otimes \cdots \otimes y_{i_K}
\mid
1\leq y_1\leq \cdots\leq y_K\leq N
\right\}.
\end{equation}
These are precisely the usual (non-super) 
bases, see for example \cite[Section 2.4]{bz1}.
\end{ex}

We call a $\glmn$-weight $\lambda=(\lambda_1,\dots,\lambda_{m+n})\in\Z^{m+n}$ a 
\textit{dominant integral} $\glmn$-weight 
if it is a dominant integral $\glnn{m}\oplus\glnn{n}$-weight. We only need $\lambda$ that are $(m|n)$-hook Young diagrams, i.e. diagrams that fit into a hook-shaped region with one horizontal arm of height $m$ and 
one vertical 
arm of width $n$ (here we use the conventions 
from \cite[Definition 2.10]{cw}). The following figure shows an $(m|n)$-hook Young diagram $\lambda$ and a box-shaped Young diagram that is not an $(m|n)$-hook.

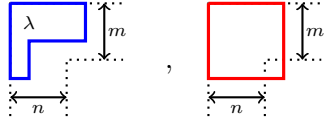
\[
\begin{tikzpicture}[anchorbase]
	\draw [thick, dotted] (0.25,-1) to (0.25,-.25) to (1,-.25);
	\draw [thick, dotted] (-.5,-1) to (-.5,.5) to (1,.5);
	\draw [very thick, blue] (-.5,.5) to (-.5,-.5) to (-.25,-.5) to (-.25,0) to (.5,0) to (.5,.5) to (-.5,.5);
	\node at (-.25,.25) {\tiny $\lambda$};
	\draw [thick, <->] (-.5,-.75) to (0.25,-.75);
	\draw [thick, <->] (.75,.5) to (.75,-.25);
	%\draw [thick, <->] (-.5,.675) to (.25,.675);
%
	\node at (.92,.125) {\tiny $m$};
	\node at (-.125,-.9) {\tiny $n$};
	%\node at (-.125,.85) {\tiny $N$};
\end{tikzpicture}
\quad , \quad
\begin{tikzpicture}[anchorbase]
	\draw [thick, dotted] (0.25,-1) to (0.25,-.25) to (1,-.25);
	\draw [thick, dotted] (-.5,-1) to (-.5,.5) to (1,.5);
	\draw [very thick, red] (-.5,.5) to (-.5,-.5) to  (.5,-.5) to (.5,.5) to (-.5,.5);
	\draw [thick, <->] (-.5,-.75) to (0.25,-.75);
	\draw [thick, <->] (.75,.5) to (.75,-.25);
	%\draw [thick, <->] (-.5,.675) to (.25,.675);
%
	\node at (.92,.125) {\tiny $m$};
	\node at (-.125,-.9) {\tiny $n$};
	%\node at (-.125,.85) {\tiny $N$};
\end{tikzpicture}
\]

We call a dominant 
integral $\glmn$-weight $\lambda$ an 
\textit{$(m|n,N)$-supported} 
$\glmn$\textit{-weight} if it corresponds to an 
$(m|n)$-hook Young diagram with at most $N$ columns. For each such $\lambda$ there exists an irreducible 
$\Us$-module $L_{m|n}(\lambda)$ and an irreducible
$\Uq(\gln)$-module $L_N(\lambda^{\T})$, see e.g. \cite[Section 2.5]{kwon}.

\begin{thm}(\textbf{Super $q$-Howe duality})\label{thm-superHowe} 
We have the following.
\begin{itemize}
\item[(a)] Let $K \in\Z_{\geq 0}$. The actions of $\Us$ and $\Uq(\gln)$ on 
$\bV_q^{K}(\C_q^{m|n}\otimes \C_q^N)$ commute 
and generate each others commutant.
\item[(b)] There exists an isomorphism 
\[
\bV_q^{\bullet}(\C_q^{m|n}\otimes \C_q^N) \cong 
(\bV_q^{\bullet}\C_q^N)^{\otimes m}\otimes(\Sym_q^{\bullet}\C_q^N)^{\otimes n}
\] 
of $\Uq(\gln)$-modules under which
the $\vec{k}$-weight space of 
$\bV_q^{\bullet}(\C_q^{m|n}\otimes \C_q^N)$ 
(considered as a $\Us$-module)
is identified with
\begin{equation}\label{eq-simply}
\bV_q^{\vec{k}_0}\C_q^N\otimes\Sym_q^{\vec{k}_1}\C_q^N=
\bV_q^{k_{1}}\C_q^N\otimes\cdots\otimes\bV_q^{k_{m}}\C_q^N
\otimes\Sym_q^{k_{m+1}}\C_q^N\otimes\cdots\otimes\Sym_q^{k_{m+n}}\C_q^N.
\end{equation}
Here $\vec{k}=(k_1,\dots,k_{m+n})$, $\vec{k}_0=(k_1,\dots,k_m)$ and 
$\vec{k}_1=(k_{m+1},\dots,k_{m+n})$.
\item[(c)] As $\Us\otimes\Uq(\gln)$-modules, we have a decomposition 
of the form
\[
\bV_q^{K}(\C_q^{m|n}\otimes \C_q^N)
\cong{\textstyle\bigoplus_{\lambda}}L_{m|n}(\lambda)\otimes L_N(\lambda^{\T}),
\]
where we sum over all 
$(m|n,N)$-supported 
$\glmn$-weights $\lambda$ whose entries sum up to $K$. 
This induces a decomposition
\[
\bV_q^{\bullet}(\C_q^{m|n}\otimes \C_q^n) 
\cong{\textstyle\bigoplus_{\lambda}}L_{m|n}(\lambda)\otimes L_N(\lambda^{\T}),
\]
where we sum over all 
$(m|n,N)$-supported $\glmn$-weights $\lambda$.
\end{itemize}
\end{thm}

\begin{rem}\label{rem-howe}
Symmetric and skew
Howe duality for the pair $(\mathrm{GL}_m,\mathrm{GL}_N)$ is 
originally due to Howe, see \cite[Section 2 and Section 4]{howe}.
Note that the non-quantum version of \fullref{thm-superHowe} can be found for 
example in \cite[Theorem 3.3]{cw} or \cite[Proposition 2.2]{sar}. Moreover, 
the ``dual'' of \fullref{thm-superHowe}, given by considering $\Uun$ as the 
Howe dual group instead of $\Us$, can be found in \cite[Proposition 4.3]{qs}.
\end{rem}

\begin{proof}
Parts (a) and (c) are proven in \cite[Theorem 2.2]{wz} or in \cite[Theorem 4.2]{qs} 
and only (b) remains to be verified.
For this purpose, we use the bases from \eqref{eq-basis}, \eqref{eq-basis2}
and \eqref{eq-basis3} to define:
\begin{gather*}
T^e_i\colon \bV_q^{k}(\C_q^N)\to \bV_q^{k}(\C_q^{m|n}\otimes \C_q^N),\quad y_{j_1}\otimes\cdots\otimes y_{j_k}\mapsto z_{ij_1}\otimes\dots\otimes z_{ij_k},\quad i\in\mathbb{I}_0,
\\
T^s_i\colon \Sym_q^{k}(\C_q^n)\to \bV_q^{k}(\C_q^{m|n}\otimes \C_q^N),\quad y_{j_1}\otimes\cdots\otimes y_{j_k}\mapsto z_{ij_1}\otimes\dots\otimes z_{ij_k},\quad i\in\mathbb{I}_1.
\end{gather*}
That these maps are well-defined $\Uq(\gln)$-intertwiners follows from the explicit 
description in \fullref{ex-altsuper2}.
Injectivity 
was shown in \cite[Theorem 4.2.2]{ckm} for the first and 
in \cite[Theorem 2.6]{rt} for the second map.
Thus, 
for $\vec{k}\in\Z^{m+n}$ with 
$k_1+\dots+k_{m+n}=K$, we see 
that
\[
T\colon
{\textstyle\bigoplus_{\vec{k}\in\Z^{m+n}_{\geq 0}}}
\bV_q^{\vec{k}_0}\C_q^N\otimes\Sym_q^{\vec{k}_1}\C_q^N\to 
\bV_q^{K}(\C_q^{m|n}\otimes \C_q^N)
\]
given by 
\[
T(v_1\otimes\cdots\otimes v_{m+n})= T^e_1(v_1)\otimes\cdots\otimes T^e_m(v_m)\otimes T^s_{m+1}(v_{m+1})\otimes\cdots\otimes T^s_{m+n}(v_{m+n}),
\]
is a $\Uq(\gln)$-module isomorphism by comparing 
the sizes of the bases from \fullref{ex-altsuper2}. This clearly 
induces the isomorphism of $\Uq(\gln)$-modules we are looking for.

It remains to verify the $\Us$-weight space decomposition 
from \eqref{eq-simply}. To this end, we 
only have to see that the action on 
$\bV_q^{\vec{k}_0}\C_q^N\otimes\Sym_q^{\vec{k}_1}\C_q^N$ 
of the $L_{i^{\prime}}$'s of $\Us$ under 
the inverse of $T$ 
is just a multiplication with $q^{k_i(\epsilon_i,\epsilon_{i^{\prime}})_{\Su}}$. 
The action of 
$\Us$ is given by
\[
L_{i^{\prime}}(z_{ij_1}\otimes\dots\otimes z_{ij_{m+n}})=
L_{i^{\prime}}(z_{ij_1})\otimes\dots\otimes L_{i^{\prime}}(z_{ij_{m+n}})=
q^{k_i(\epsilon_i,\epsilon_{i^{\prime}})_{\Su}} z_{ij_1}\otimes\dots\otimes z_{ij_{m+n}}.
\]
Hence, the $\Us$-weight space decomposition follows.
\end{proof}

By \fullref{thm-superHowe} part (b): 
we get 
linear maps
\[
f_{\vec{k}}^{\vec{l}}\colon 1_{\vec{l}}\,\Usd 1_{\vec{k}}\to\Hom_{\Uun}(\bV_q^{\vec{k}_0}\C_q^N\otimes\Sym_q^{\vec{k}_1}\C_q^N,
\bV_q^{\vec{l}_0}\C_q^N\otimes\Sym_q^{\vec{l}_1}\C_q^N)
\]
for any two $\vec{k},\vec{l}\in\Z_{\geq 0}^{m+n}$ such that 
$\sum_{i=0}^{m+n}k_i=\sum_{i=0}^{m+n}l_i$. 
Using part (a) of \fullref{thm-superHowe}, we see that
the homomorphisms $f_{\vec{k}}^{\vec{l}}$ are all 
surjective. Thus, we get the following.

\begin{cor}\label{cor-functor}
There exists a full functor $\Phi_{\Su}^{m|n}\colon\Usd \to \Repas$, 
which we call 
the \textit{super $q$-Howe functor}, 
given on objects and morphisms by
\begin{gather*}
\Phi_{\Su}^{m|n}(\vec{k})=\bV_q^{\vec{k}_0}\C_q^N\otimes\Sym_q^{\vec{k}_1}\C_q^N,\quad\quad
\Phi_{\Su}^{m|n}(1_{\vec{l}} x 1_{\vec{k}})=f_{\vec{k}}^{\vec{l}}(x).
\end{gather*} 
Everything else is sent to zero.\qed
\end{cor}

\subsection{The sorted equivalences}\label{sub-sortequi}

In this subsection we construct a full and faithful functor
\[
\Gams\colon\SoSuSp\to\Repsas,
\]
where $\SoSuSp$ 
is the sorted web category from \fullref{defn-spid2} and 
$\Repsas$ denotes the full subcategory of $\Repas$ whose objects are sorted as in \eqref{eq-simply}.

As already explained in the introduction, 
we 
essentially define $\Gams$ such that there is a commuting diagram
\begin{equation}\label{eq-commute}
\begin{gathered}
\xymatrix{
\Usd \ar[r]^{\Phi_{\Su}^{m|n}} \ar[dr]_{\Lad} & \Repsas \\
& \SoSuSp \ar[u]_{\Gams}
}
\end{gathered}
\end{equation}
The functor $\Lad$ is a \textit{ladder functor}, whose 
definition is motivated by \cite[Section 5.1]{ckm}.

\begin{lem}\label{lem-ladderfunc}
Let $m,n\in\Z_{\geq 0}$. There exists a functor 
\[
\Lad\colon\Usd\to\SoSuSp
\]
which sends 
a $\glmn$-weight $\vec{k}\in\Z^{m+n}_{\geq 0}$ 
to $((k_1)_g,\dots,(k_m)_g,(k_{m+1})_r,\dots,(k_{m+n})_r)$ in $\SoSuSp$
and all other $\glmn$-weights to the 
zero object. On morphisms $\Lad$ is 
given by
\[
F_i^{(j)} 1_{\vec{k}}\mapsto 
\xy
(0,0)*{
\begin{tikzpicture}[scale=.3]
	\draw [very thick, green, directed=0.55] (-4.5,-2) to (-4.5,2);
	\draw [very thick, green, directed=0.55] (-1.5,-2) to (-1.5,0);
	\draw [very thick, green, directed=0.55] (-1.5,0) to (-1.5,2);
	\draw [very thick, green, directed=0.55] (1.5,-2) to (1.5,0);
	\draw [very thick, green, directed=0.55] (1.5,0) to (1.5,2);
	\draw [very thick, green, directed=.55] (-1.5,0) to (1.5,0);
	\draw [very thick, green, directed=0.55] (4.5,-2) to (4.5,2);
	\draw [very thick, mycolor, directed=0.55] (7.5,-2) to (7.5,2);
	\draw [very thick, mycolor, directed=0.55] (10.5,-2) to (10.5,2);
	\node at (-4.5,-2.5) {\tiny $k_1$};
	\node at (-1.5,-2.5) {\tiny $k_i$};
	\node at (1.5,-2.5) {\tiny $k_{i+1}$};
	\node at (4.5,-2.5) {\tiny $k_m$};
	\node at (7.5,-2.5) {\tiny $k_{m+1}$};
	\node at (10.5,-2.5) {\tiny $k_{m+n}$};
	\node at (-1.65,2.5) {\tiny $k_i\!\! -\!\! j$};
	\node at (1.65,2.5) {\tiny $k_{i+1}\!\! +\!\! j$};
	\node at (-4.5,2.5) {\tiny $k_1$};
	\node at (4.5,2.5) {\tiny $k_m$};
	\node at (7.5,2.5) {\tiny $k_{m+1}$};
	\node at (10.5,2.5) {\tiny $k_{m+n}$};
	\node at (0,0.75) {\tiny $j$};
	\node at (-2.95,0) {$\cdots$};
	\node at (3.05,0) {$\cdots$};
	\node at (9.05,0) {$\cdots$};
\end{tikzpicture}
};
\endxy 
\;\; , \;\;
F_i^{(j)} 1_{\vec{k}}\mapsto 
\xy
(0,0)*{
\begin{tikzpicture}[scale=.3]
	\draw [very thick, green, directed=0.55] (-4.5,-2) to (-4.5,2);
	\draw [very thick, green, directed=0.55] (-1.5,-2) to (-1.5,2);
	\draw [very thick, mycolor, directed=0.55] (1.5,-2) to (1.5,2);
	\draw [very thick, mycolor, directed=.55] (4.5,0) to (7.5,0);
	\draw [very thick, mycolor, directed=.25, directed=.75] (4.5,-2) to (4.5,2);
	\draw [very thick, mycolor, directed=.25, directed=.75] (7.5,-2) to (7.5,2);
	\draw [very thick, mycolor, directed=0.55] (10.5,-2) to (10.5,2);
	\node at (-4.5,-2.5) {\tiny $k_1$};
	\node at (-1.5,-2.5) {\tiny $k_m$};
	\node at (1.5,-2.5) {\tiny $k_{m+1}$};
	\node at (4.5,-2.5) {\tiny $k_i$};
	\node at (7.5,-2.5) {\tiny $k_{i+1}$};
	\node at (10.5,-2.5) {\tiny $k_{m+n}$};
	\node at (-1.65,2.5) {\tiny $k_m$};
	\node at (1.65,2.5) {\tiny $k_{m+1}$};
	\node at (-4.5,2.5) {\tiny $k_1$};
	\node at (4.35,2.5) {\tiny $k_i\!\! -\!\! j$};
	\node at (7.60,2.5) {\tiny $k_{i+1}\!\! +\!\! j$};
	\node at (10.5,2.5) {\tiny $k_{m+n}$};
	\node at (6,0.75) {\tiny $j$};
	\node at (-2.95,0) {$\cdots$};
	\node at (3.05,0) {$\cdots$};
	\node at (9.05,0) {$\cdots$};
\end{tikzpicture}
};
\endxy
\]
for $i\in\mathbb{I}_0-\{m\}$ or $i\in\mathbb{I}_1-\{m+n\}$ respectively, and
\[
F_m 1_{\vec{k}}\mapsto 
\xy
(0,0)*{
\begin{tikzpicture}[scale=.3]
	\draw [very thick, green, directed=0.55] (-4.5,-2) to (-4.5,2);
	\draw [very thick, green, directed=0.55] (-1.5,-2) to (-1.5,2);
	\draw [very thick, green, directed=.25, directed=.75] (1.5,-2) to (1.5,2);
	\draw [very thick, directed=.55] (1.5,0) to (4.5,0);
	\draw [very thick, mycolor, directed=.25, directed=.75] (4.5,-2) to (4.5,2);
	\draw [very thick, mycolor, directed=0.55] (7.5,-2) to (7.5,2);
	\draw [very thick, mycolor, directed=0.55] (10.5,-2) to (10.5,2);
	\node at (-4.5,-2.5) {\tiny $k_1$};
	\node at (-1.5,-2.5) {\tiny $k_{m-1}$};
	\node at (1.5,-2.5) {\tiny $k_{m}$};
	\node at (4.5,-2.5) {\tiny $k_{m+1}$};
	\node at (7.5,-2.5) {\tiny $k_{m+2}$};
	\node at (10.5,-2.5) {\tiny $k_{m+n}$};
	\node at (-4.5,2.5) {\tiny $k_1$};
	\node at (-1.5,2.5) {\tiny $k_{m-1}$};
	\node at (1.25,2.5) {\tiny $k_{m}\!\! -\!\! 1$};
	\node at (4.65,2.5) {\tiny $k_{m+1}\!\! +\!\! 1$};
	\node at (7.65,2.5) {\tiny $k_{m+2}$};
	\node at (10.5,2.5) {\tiny $k_{m+n}$};
	\node at (3,0.75) {\tiny $1$};
	\node at (-2.95,0) {$\cdots$};
	\node at (9.05,0) {$\cdots$};
\end{tikzpicture}
};
\endxy
\]
Similarly, but with reversed horizontal orientations, for the generators 
$E_i^{(j)}\one_{\vec{k}}$ and $E_m\one_{\vec{k}}$.
\end{lem}

\begin{proof}
To show that $\Lad$ is well-defined, it suffices to show that all relations 
in $\Usd$ are satisfied in $\SoSuSp$. For monochromatic relations we 
can copy \cite[Proposition 5.2.1]{ckm}.
\fullref{lem-secondcq} shows 
that the super relations \eqref{eq-super1} 
hold in $\SoSuSp$.
\end{proof}

\begin{defn}(\textbf{The diagrammatic presentation functor $\Gams$})\label{defn-mainfunctor}
We define a functor $\Gams\colon\SoSuSp\to\Repsas$ as follows.
\begin{itemize}
\item On objects: to each 
$\vec{k}=((k_1)_g,\dots,(k_m)_g,(k_{m+1})_r,\dots,(k_{m+n})_r)$, we assign
\[
\Gams(\vec{k})=\bV_q^{\vec{k}_0}\C_q^N\otimes\Sym_q^{\vec{k}_1}\C_q^N,
\]
where $\vec{k}_0=(k_1,\dots,k_m)$ and $\vec{k}_1=(k_{m+1},\dots,k_{m+n})$. 
Moreover, we 
send the empty tuple to the trivial $\Uun$-module $\C_q$ 
and the zero object to the $\Uun$-module $0$.

\item On morphisms:
we use the functor $\Phi_{\Su}^{m|n}$ 
from \fullref{cor-functor} to define $\Gams$ on the generating 
trivalent vertices in $\SoSuSp$ (here we assume that the diagrams are the 
identities outside of the illustrated part). For this, let
$i\in\mathbb{I}$ and we use the notation $k=k_i,l=k_{i+1}$ and $(k,l)=(k_1,\dots,k_i=k,k_{i+1}=l,\dots,k_{m+n})$.
\begin{equation}\label{eq-imgreen}
\begin{gathered}
\Gams \left(
\xy
(0,0)*{
\begin{tikzpicture}[scale=.3]
	\draw [very thick, green, directed=0.55] (0, .75) to (0,2.5);
	\draw [very thick, green, directed=0.55] (1,-1) to [out=90,in=330] (0,.75);
	\draw [very thick, green, directed=0.55] (-1,-1) to [out=90,in=210] (0,.75); 
	\node at (0, 3) {\tiny $k\! +\! l$};
	\node at (-1,-1.5) {\tiny $k$};
	\node at (1,-1.5) {\tiny $l$};
\end{tikzpicture}
};
\endxy
\right)
=
\Phi_{\Su}^{m|n}(E_i^{(l)}1_{(k,l)})
,\quad\quad
\Gams \left(
\xy
(0,0)*{
\begin{tikzpicture}[scale=.3]
	\draw [very thick, green, directed=0.55] (0,-1) to (0,.75);
	\draw [very thick, green, directed=0.55] (0,.75) to [out=30,in=270] (1,2.5);
	\draw [very thick, green, directed=0.55] (0,.75) to [out=150,in=270] (-1,2.5); 
	\node at (0, -1.5) {\tiny $k\! +\! l$};
	\node at (-1,3) {\tiny $k$};
	\node at (1,3) {\tiny $l$};
\end{tikzpicture}
};
\endxy
\right)
=
\Phi_{\Su}^{m|n}(F_i^{(l)}1_{(k+l,0)}),
\\
\Gams \left(
\xy
(0,0)*{
\begin{tikzpicture}[scale=.3]
	\draw [very thick, mycolor, directed=0.55] (0, .75) to (0,2.5);
	\draw [very thick, mycolor, directed=0.55] (1,-1) to [out=90,in=330] (0,.75);
	\draw [very thick, mycolor, directed=0.55] (-1,-1) to [out=90,in=210] (0,.75); 
	\node at (0, 3) {\tiny $k\! +\! l$};
	\node at (-1,-1.5) {\tiny $k$};
	\node at (1,-1.5) {\tiny $l$};
\end{tikzpicture}
};
\endxy
\right)
=
\Phi_{\Su}^{m|n}(F_i^{(k)}1_{(k,l)})
,\quad\quad
\Gams \left(
\xy
(0,0)*{
\begin{tikzpicture}[scale=.3]
	\draw [very thick, mycolor, directed=0.55] (0,-1) to (0,.75);
	\draw [very thick, mycolor, directed=0.55] (0,.75) to [out=30,in=270] (1,2.5);
	\draw [very thick, mycolor, directed=0.55] (0,.75) to [out=150,in=270] (-1,2.5); 
	\node at (0, -1.5) {\tiny $k\! +\! l$};
	\node at (-1,3) {\tiny $k$};
	\node at (1,3) {\tiny $l$};
\end{tikzpicture}
};
\endxy
\right)
=
\Phi_{\Su}^{m|n}(E_i^{(k)}1_{(0,k+l)}).
\end{gathered}
\end{equation}
\end{itemize}
Note that these definitions include the mixed case, where we either have $l = 1$ (and colored black) or $k = 1 $ (and colored black) and we use the odd generators $F_m$ and $E_m$.
\end{defn}

\begin{rem}\label{rem-otherchoices}
There are certain choices for the 
images of monochromatic merges and splits, but 
these choices do not matter, see \cite[Remark 2.18]{rt}. In contrast, there is 
no other choice for the mixed merges and splits. For example, take $l=1$ 
in the top left in \eqref{eq-imgreen}. 
The green edge labeled $k+1$ should represent $\bV_q^{k+1}\C_q^N$. Thus, we 
have to see the top boundary of the left-hand side as $\one_{(k+1,0)}$ and not as 
$\one_{(0,k+1)}$, which determines our choices. Similarly for the other 
mixed generators. For example, if $m=n=1$, and $k=1$ or $l=1$, then
\[
\Gams\left(
\xy
(0,0)*{
\begin{tikzpicture}[scale=.3]
	\draw [very thick, green, directed=0.55] (0, .75) to (0,2.5);
	\draw [very thick, directed=0.55] (1,-1) to [out=90,in=330] (0,.75);
	\draw [very thick, directed=0.55] (-1,-1) to [out=90,in=210] (0,.75); 
	\node at (0, 3) {\tiny $2$};
	\node at (-1,-1.5) {\tiny $1$};
	\node at (1,-1.5) {\tiny $1$};
\end{tikzpicture}
};
\endxy
\right)
=
\Phi_{\Su}^{1|1}(E_11_{(1,1)})
\neq
\Phi_{\Su}^{1|1}(F_11_{(1,1)})
=
\Gams \left(
\xy
(0,0)*{
\begin{tikzpicture}[scale=.3]
	\draw [very thick, mycolor, directed=0.55] (0, .75) to (0,2.5);
	\draw [very thick, directed=0.55] (1,-1) to [out=90,in=330] (0,.75);
	\draw [very thick, directed=0.55] (-1,-1) to [out=90,in=210] (0,.75); 
	\node at (0, 3) {\tiny $2$};
	\node at (-1,-1.5) {\tiny $1$};
	\node at (1,-1.5) {\tiny $1$};
\end{tikzpicture}
};
\endxy
\right).
\]
\end{rem}

\begin{lem}\label{lem-mainfunctor}
$\Gams$ is a well-defined functor $\Gams\colon\SoSuSp\to\Repsas$ making the diagram \eqref{eq-commute} commutative.
\end{lem}

\begin{proof}
First we note that $\Gams\circ \Lad= \Phi_{\Su}^{m|n}$ on generators $F^{(j)}_i\one_{\vec{k}}$ and $F_m\one_{\vec{k}}$ (and analogously for $E$'s) with $i\in\mathbb{I}-\{m\},j\in\Z_{\geq 0}$ and $\vec{k}\in\Z^{m+n}$. This follows from the definition of $\Gams$ via $\Phi_{\Su}^{m|n}$ and the observation that ladders can be written as compositions of merges and splits, see also \cite[Lemma 2.20]{rt}.

We need to check that the images of the relations from $\SoSuSp$ under $\Gams$ 
hold in $\Repsas$. \fullref{cor-functor} guarantees that all relations in $\Repsas$ are induced via $\Phi_{\Su}^{m|n}$ from relations in $\Usd$ and the fact that $\Phi_{\Su}^{m|n}$ kills certain $\glmn$-weights. It remains to check that the relations in $\SoSuSp$ are, likewise, induced via $\Lad$ from relations in $\Usd$. For the monochromatic and isotopy relations, this follows as in \cite[Lemma 2.20]{rt}.

The dumbbell relation \eqref{eq-dumb} can be recovered from $\Usd$ as follows. Without loss of generality we work with $m=n=1$:
\begin{gather*}
[2]
\xy
(0,0)*{
\begin{tikzpicture}[scale=.3] 
	\draw [very thick, directed=.55] (1,-2.75) to (1,2.5);
	\draw [very thick, directed=.55] (-1,-2.75) to (-1,2.5);
	\node at (-1,3) {\tiny $1$};
	\node at (1,3) {\tiny $1$};
	\node at (-1,-3.15) {\tiny $1$};
	\node at (1,-3.15) {\tiny $1$};
\end{tikzpicture}
};
\endxy
=
\Upsilon^{1|1}_{\Su}([2]\one_{(1,1)})
%=\Phi_{\Su}^{1|1}([2]\one_{(1,1)})\\
%=\Phi_{\Su}^{1|1}(FE\one_{(1,1)}+EF\one_{(1,1)})
=\Upsilon^{1|1}_{\Su}(FE\one_{(1,1)}+EF\one_{(1,1)})
=
\xy
(0,0)*{
\begin{tikzpicture}[scale=.3]
	\draw [very thick, green, directed=.55] (0,-1) to (0,.75);
	\draw [very thick, directed=.55] (0,.75) to [out=30,in=270] (1,2.5);
	\draw [very thick, directed=.55] (0,.75) to [out=150,in=270] (-1,2.5); 
	\draw [very thick, directed=.55] (1,-2.75) to [out=90,in=330] (0,-1);
	\draw [very thick, directed=.55] (-1,-2.75) to [out=90,in=210] (0,-1);
	\node at (-1,3) {\tiny $1$};
	\node at (1,3) {\tiny $1$};
	\node at (-1,-3.15) {\tiny $1$};
	\node at (1,-3.15) {\tiny $1$};
	\node at (-0.5,0) {\tiny $2$};
\end{tikzpicture}
};
\endxy
+
\xy
(0,0)*{
\begin{tikzpicture}[scale=.3]
	\draw [very thick, mycolor, directed=.55] (0,-1) to (0,.75);
	\draw [very thick, directed=.55] (0,.75) to [out=30,in=270] (1,2.5);
	\draw [very thick, directed=.55] (0,.75) to [out=150,in=270] (-1,2.5); 
	\draw [very thick, directed=.55] (1,-2.75) to [out=90,in=330] (0,-1);
	\draw [very thick, directed=.55] (-1,-2.75) to [out=90,in=210] (0,-1);
	\node at (-1,3) {\tiny $1$};
	\node at (1,3) {\tiny $1$};
	\node at (-1,-3.15) {\tiny $1$};
	\node at (1,-3.15) {\tiny $1$};
	\node at (-0.5,0) {\tiny $2$};
\end{tikzpicture}
};
\endxy
\end{gather*}
Relation \eqref{eq-alt} is a consequence of killing $\glmn$-weights $\vec{k}=(k_1,\dots, k_{m+n})$, one of whose first $m$ entries is larger than $N$.
\end{proof}

\begin{lem}\label{lem-faithful}
The functor $\iota^{\infty}_{1}\colon\CKM\to\SuSp$ is faithful.
\end{lem}

\begin{proof}
By \fullref{lem-mainfunctor} and a 
comparison of definitions, we have a commuting diagram
\begin{equation*}
\begin{gathered}
\xymatrix{
\Repe{N}  \ar[r]^{\iota^{\mathrm{es}}_{\mathrm{e}}}& \Repsas\\
\CKM \ar[u]^{\Gamma_{\mathrm{CKM}}}\ar[r]_{\iota^{\infty}_{1}} & \SoSuSp \ar[u]_{\Gams}
}
\end{gathered}
\end{equation*}
where $\Gamma_{\mathrm{CKM}}$ is the functor considered 
in \cite[Section 3.2]{ckm} and 
$\iota^{\mathrm{es}}_{\mathrm{e}}$ is the evident embedding of 
a full subcategory. $\Gamma_{\mathrm{CKM}}$ is 
faithful by \cite[Theorem 3.3.1]{ckm} and thus,  
$\iota^{\infty}_{1}$ is faithful as well.
\end{proof}

\begin{rem}\label{rem-additive}
Let $\Mat(\SoSuSp)$ be the \textit{additive closure} of $\SoSuSp$: 
objects are finite, formal direct sums of the objects of $\SoSuSp$ 
and morphisms are matrices (whose entries are morphisms from $\SoSuSp$).
We can extend $\Gams$ additively to a functor
\[
\Gams\colon\Mat(\SoSuSp)\to\Repsas.
\]
Similarly for $\Gam$ later on.
\end{rem}

\begin{prop}\label{prop-sortequi}
The functor $\Gams\colon\SoSuSp\to\Repsas$ gives rise to an equivalence of categories 
$\Gams\colon\Mat(\SoSuSp)\to\Repsas$.
\end{prop}

\begin{proof}
Since 
$\Gams\colon\Mat(\SoSuSp)\to\Repsas$ is well-defined
by \fullref{lem-mainfunctor} and \fullref{rem-additive}, it remains to show 
that $\Gams$ is essentially surjective, full and faithful.

\textbf{Essentially surjective.} This follows 
directly from the definitions of $\Gams$, $\SoSuSp$, 
its additive closure $\Mat(\SoSuSp)$ and $\Repsas$.

\textbf{Full.}: It suffices to verify fullness 
for morphisms between objects of the form 
$\vec{k}\in X^{m+n}$, where $X^{m+n}=(X_b\cup X_g)^m\cup (X_b\cup X_r)^n$. That it holds is clear from diagram \eqref{eq-commute} since $\Phi_{\Su}^{m|n}$ is full 
by \fullref{cor-functor}.

\textbf{Faithful.} Again it suffices to verify faithfulness 
for morphisms between objects of the form 
$\vec{k}\in X^{m+n}$.
Given any web $u\in\Hom_{\SoSuSp}(\vec{k},\vec{l})$ for 
$\vec{k}\in X^{m+n}$ and 
$\vec{l}\in X^{m^{\prime}+n^{\prime}}$, we can 
compose $u$ from the bottom and the top with 
merges and splits respectively 
to obtain
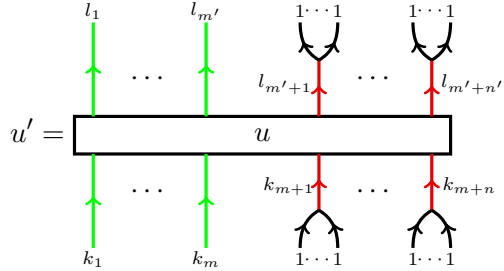
\[
u^{\prime}=
\xy
(0,0)*{
\begin{tikzpicture} [scale=1]
\draw[very thick] (-1.5,1.25) rectangle (3.5,.75);
\draw[very thick, green, directed=0.55] (-1.25,-.5) to (-1.25,.75);
\draw[very thick, green, directed=0.55] (-1.25,1.25) to (-1.25,2.5);
\draw[very thick, green, directed=0.55] (.25,-.5) to (.25,.75);
\draw[very thick, green, directed=0.55] (.25,1.25) to (.25,2.5);
\draw[very thick, mycolor, directed=0.55] (1.75,0) to (1.75,.75);
\draw[very thick, mycolor, directed=0.55] (1.75,1.25) to (1.75,2);
\draw[very thick, mycolor, directed=0.55] (3.25,0) to (3.25,.75);
\draw[very thick, mycolor, directed=0.55] (3.25,1.25) to (3.25,2);
\node at (-0.5,1.75) {$\cdots$};
\node at (-0.5,.25) {$\cdots$};
\node at (2.5,1.75) {$\cdots$};
\node at (2.5,.25) {$\cdots$};
\node at (1,1) {$u$};
\node at (3.8,1.675) {\tiny $l_{m^{\prime}+n^{\prime}}$};
\node at (3.75,0.325) {\tiny $k_{m+n}$};
\node at (-1.25,2.65) {\tiny $l_1$};
\node at (-1.25,-.65) {\tiny $k_1$};
\node at (0.25,2.65) {\tiny $l_{m^{\prime}}$};
\node at (0.25,-.65) {\tiny $k_m$};
\node at (1.3,1.675) {\tiny $l_{m^{\prime}+1}$};
\node at (1.35,0.325) {\tiny $k_{m+1}$};
%\draw [very thick, directed=.55] (-1,-.5) to [out=90,in=330] (-1.25,0);
%\draw [very thick, rdirected=.65] (-1.25,0) to [out=210,in=90] (-1.5,-.5);
%\draw [very thick, rdirected=.65] (0.25,0) to [out=330,in=90] (0.5,-.5);
%\draw [very thick, rdirected=.65] (0.25,0) to [out=210,in=90] (0,-.5);
\draw [very thick, rdirected=.65] (1.75,0) to [out=330,in=90] (2,-.5);
\draw [very thick, rdirected=.65] (1.75,0) to [out=210,in=90] (1.5,-.5);
\draw [very thick, rdirected=.65] (3.25,0) to [out=330,in=90] (3.5,-.5);
\draw [very thick, rdirected=.65] (3.25,0) to [out=210,in=90] (3,-.5);
%\draw [very thick, directed=.55] (-1.25,2) to [out=30,in=270] (-1,2.5);
%\draw [very thick, directed=.55] (-1.25,2) to [out=150,in=270] (-1.5,2.5);
%\draw [very thick, directed=.55] (0.25,2) to [out=30,in=270] (0.5,2.5);
%\draw [very thick, directed=.55] (0.25,2) to [out=150,in=270] (0,2.5);
\draw [very thick, directed=.55] (1.75,2) to [out=30,in=270] (2,2.5);
\draw [very thick, directed=.55] (1.75,2) to [out=150,in=270] (1.5,2.5);
\draw [very thick, directed=.55] (3.25,2) to [out=30,in=270] (3.5,2.5);
\draw [very thick, directed=.55] (3.25,2) to [out=150,in=270] (3,2.5);
%\node at (-1.25,2.65) {\tiny $\cdots$};
%\node at (-1.25,-.65) {\tiny $\cdots$};
%\node at (0.25,2.65) {\tiny $\cdots$};
%\node at (0.25,-.65) {\tiny $\cdots$};
\node at (1.75,2.65) {\tiny $\cdots$};
\node at (1.75,-.65) {\tiny $\cdots$};
\node at (3.25,2.65) {\tiny $\cdots$};
\node at (3.25,-.65) {\tiny $\cdots$};
%\node at (-1,2.65) {\tiny $1$};
%\node at (-1.5,2.65) {\tiny $1$};
%\node at (-1,-.65) {\tiny $1$};
%\node at (-1.5,-.65) {\tiny $1$};
%\node at (0,2.65) {\tiny $1$};
%\node at (0.5,2.65) {\tiny $1$};
%\node at (0,-.65) {\tiny $1$};
%\node at (0.5,-.65) {\tiny $1$};
\node at (1.5,2.65) {\tiny $1$};
\node at (2,2.65) {\tiny $1$};
\node at (1.5,-.65) {\tiny $1$};
\node at (2,-.65) {\tiny $1$};
\node at (3,2.65) {\tiny $1$};
\node at (3.5,2.65) {\tiny $1$};
\node at (3,-.65) {\tiny $1$};
\node at (3.5,-.65) {\tiny $1$};
\end{tikzpicture}
};
\endxy
\]
Recall that exploding edges is, 
by \eqref{eq-simpler1}, a 
reversible operation. Hence, we have
\[
\Gams(u)=\Gams(v)
\quad\text{if and only if}\quad
\Gams(u^{\prime})=\Gams(v^{\prime}),
\] 
which together with \fullref{cor-green2} reduces the verification of faithfulness to the case where all web edges are black or green. Such webs lie in $\iota^{\infty}_{1}(\CKM)$ and faithfulness follows as in the proof of \fullref{lem-faithful}.
\end{proof}

\subsection{Proofs of the equivalences}\label{sub-proof}

\begin{rem}\label{rem-rmatrix}
Recall that the \textit{universal} $R$\textit{-matrix} for $\gln$ gives a braiding on $\Repas$ as follows 
(see e.g. \cite[Chapter XI, Sections 2 and 7]{tur}).
For any pair of $\Uun$-modules $V,W$ in $\Repas$ let 
$\mathrm{Perm}_{V,W}\colon V\otimes W\to W\otimes V$ be the permutation 
$\mathrm{Perm}_{V,W}(v\otimes w)=w\otimes v$ and define
$\beta^R_{V,W}=\mathrm{Perm}_{V,W}\circ R$.
We scale $\beta^R_{V,W}$ as follows:
\[
\tilde\beta^R_{V,W}=q^{-\frac{kl}{N}} \beta^R_{V,W}
\]
whenever $V$ and $W$ are exterior or symmetric power $\Uun$-modules of exponent $k$ and $l$ respectively. This induces a scaling 
$\tilde\beta^R_{V,W}$ of $\beta^R_{V,W}$ for all 
$\Uun$-modules $V,W\in\Repas$.
Then $(\Repas,\tilde\beta^R_{\cdot,\cdot})$ is a braided monoidal category.
\end{rem}

The goal of this subsection is to finally prove our main theorems.
To this end, we extend \eqref{eq-commute} to a diagram 
\begin{equation}\label{eq-commute2}
\begin{gathered}
\xymatrix{
\Usd \ar[r]^{\Phi_{\Su}^{m|n}} \ar[dr]_{\Lad} & \Repsas\ar@{^{(}->}[r]^{\iota_{\mathrm{alg}}}  & \Repas\\
& \SoSuSp \ar[u]_{\Gams}\ar@{^{(}->}[r]_{\iota_{\mathrm{dia}}} & \SuSp\ar[u]_{\Gam}
}
\end{gathered}
\end{equation}
where $\iota_{\mathrm{alg}}$ and $\iota_{\mathrm{dia}}$ are the evident 
inclusions of full subcategories. We will define the functor $\Gam$ such that 
the diagram \eqref{eq-commute2} commutes.

\begin{defn}(\textbf{The diagrammatic presentation functor $\Gam$})\label{defn-mainfunctor2}
We define a functor $\Gam\colon\SuSp\to\Repas$ as follows.
\begin{itemize}
\item On objects: $\Gam$ sends an object $\vec{k}\in X^{L}$ of $\SuSp$ to the tensor product of exterior and symmetric powers of $\C_q^N$ specified by the entries of $\vec{k}$: green and red integers encode exterior and symmetric powers respectively, and a black entry $1$ corresponds to $\C_q^N$ itself.

\item On morphisms: For an object $\vec{k}\in X^L$ let $w(\vec{k})\in S_L$ be a shortest length permutation that sorts green integers in $\vec{k}$ to the left of red integers. We define $\Gam$ on an arbitrary web $u\in\Hom_{\SuSp}(\vec{k},\vec{l})$ by pre-composing and 
post-composing with elementary crossings and the universal $R$-matrix intertwiners:
\[
\Gam(u)=(\tilde\beta^{R}_{w(\vec{l})})^{-1}\circ
\Gams(\beta^{\Sup}_{w(\vec{l})}\circ u\circ(\beta^{\Sup}_{w(\vec{k})})^{-1})
\circ\tilde\beta^{R}_{w(\vec{k})}.
\]
\end{itemize}
Clearly, $\Gam$ restricts to $\Gams$.
\end{defn}

\begin{lem}\label{lem-mainagain}
$\Gam\colon\SuSp\to\Repas$ is a monoidal 
functor making \eqref{eq-commute2} commutative. 
\end{lem}

\begin{proof}
By \fullref{lem-mainfunctor} and the fact that $\beta^{\Sup}_{\cdot,\cdot}$ and 
$\tilde\beta^{R}_{\cdot,\cdot}$ are braidings (see \fullref{prop-braiding} 
and \fullref{rem-rmatrix}), we see that $\Gam$ is well-defined. That $\Gam$ is monoidal and makes \eqref{eq-commute2} commutative is clear from its construction.
\end{proof}

\begin{prop}\label{prop-braidfunctor}
The functor $\Gam\colon(\SuSp,\beta^{\Sup}_{\cdot,\cdot})\to(\Repas,\tilde\beta^R_{\cdot,\cdot})$ 
is a functor of braided monoidal categories.
\end{prop}

\begin{proof}
By \fullref{lem-mainagain}, it remains to verify 
\begin{equation*}\label{eq-commute3}
\Gam(\beta^{\Sup}_{\vec{k}\otimes\vec{l}})=\tilde\beta^R_{\Gam(\vec{k}),\Gam(\vec{l})},\quad
\text{for all objects }\vec{k},\vec{l}\text{ of }\SuSp.
\end{equation*}
The green--red symmetry and the fact that the mixed crossings are defined via the monochromatic crossings, together with \fullref{cor-ckmwebs}, reduce this problem to the situation studied in \cite[Theorem 6.2.1 and Lemma 6.2.2]{ckm}. It remains to show
\[
\Gam(\beta^{g}_{1,1})=\Gam(\beta^{r}_{1,1})=\Gams(\beta^{g}_{1,1})=
\Gams(\beta^{r}_{1,1})=\tilde\beta^R_{\C_q^N,\C_q^N}.
\] 
This follows since $\Gams(\beta^{g}_{1,1})=\Gams(\beta^{r}_{1,1})$ acts on 
\[
\C_q^N\otimes\C_q^N\cong\bV_q^2(\C_q^N)\oplus\Sym_q^2(\C_q^N)
\]
as 
$-q^{-1}$ on the first summand and as $q$ on the second 
(see \fullref{ex-oneoneok}).
\end{proof}

\begin{thm}(\textbf{The diagrammatic presentations})\label{thm-equi}
The functor 
\[
\Gam\colon(\Mat(\SuSp),\beta^{\Sup}_{\cdot,\cdot})\to(\Repas,\tilde\beta^R_{\cdot,\cdot})
\] 
is an equivalence of braided monoidal categories.
\end{thm}

\begin{proof}
By \fullref{prop-braidfunctor}, $\Gam$
extends to a braided monoidal functor on the additive closure 
and it remains to show that $\Gam$ 
is essentially surjective, full and faithful.

\textbf{Essentially surjective.} This follows directly from the definitions. 
See also \fullref{rem-additive}.

\textbf{Full and faithful.} As before, it suffices to verify this on morphisms between objects of the form 
$\vec{k}\in X^{L}$. Consider the commuting diagram
\[
\begin{gathered}
\xymatrix{
\Repsas  & \Repas \ar[l]_/-0.3em/{\omega_R}\\
\SoSuSp \ar[u]^{\Gams} & \SuSp\ar[u]_{\Gam}\ar[l]^/-0.3em/{\omega_{\Sup}}
}
\end{gathered}
\]
where $\omega_{R}$ and 
$\omega_{\Sup}$ are the functors that order $f\in\Hom_{\Repas}(\Gam(\vec{k}),\Gam(\vec{l}))$ and
webs $u\in\Hom_{\SuSp}(\vec{k},\vec{l})$ by using the $R$-matrix 
braiding $\tilde\beta^R_{\cdot,\cdot}$ and the braiding $\beta^{\Sup}_{\cdot,\cdot}$ respectively, via a permutation of shortest length. 
Since sorting is invertible we get:
\begin{align*}
\dim(\Hom_{\Repas}(\Gam(\vec{k}),\Gam(\vec{l})))&=\dim(\Hom_{\Repsas}(\Gams(\omega_{\Sup}(\vec{k})),
\Gams(\omega_{\Sup}(\vec{l}))))\\
&=\dim(\Hom_{\SoSuSp}
(\omega_{\Sup}(\vec{k}),\omega_{\Sup}(\vec{l})))\\
&=\dim(\Hom_{\SuSp}(\vec{k},\vec{l})),
\end{align*}
where the second equality follows 
from \fullref{prop-sortequi}.
\end{proof}

\begin{rem}\label{rem-pivotal} 
For now we restrict ourselves to working with webs with only upward-oriented edges. Downward-oriented edges, as for example in \cite{ckm}, can be used to represent the duals of the $\Uun$-modules $\bV_q^k\C_q^N$ and $\Sym_q^l\C_q^N$. With respect to such an enriched web calculus, 
the statement of \fullref{thm-equi} extends to an equivalence of pivotal categories, see \cite[Section 6]{qs} and \fullref{rem-duals}.
\end{rem}

Let $\Hee$ denotes the monoidal, $\C_q$-linear category obtained 
from the collection $H_{\In}(q)$ of 
Iwahori--Hecke algebras as follows. The objects $e,e^{\prime}$ of $\Hee$ are tensor products of Iwahori--Hecke algebra idempotents corresponding to
$e_{\mathrm{col}}(\lambda)$ and $e_{\mathrm{row}}(\lambda)$ 
(as in \fullref{defn-idem}) under the isomorphism in \fullref{lem-hecke}. The morphism spaces are given by
$\Hom_{\Hee}(e,e^{\prime})=e^{\prime}H_{\In}(q)e$. The category $\Hee$ is braided 
with braiding $\tilde\beta^{H}_{\cdot,\cdot}$ induced from $H_{\In}(q)$.

\begin{thm}(\textbf{The diagrammatic presentation})\label{thm-equi2}
For large $N$ the functors $\Gam$ stabilize to a functor 
\[
\Gami\colon(\Mat(\InSp),\beta^{\Sup}_{\cdot,\cdot})\to(\Mat(\Hee),\tilde\beta^{H}_{\cdot,\cdot}),
\] 
which is an equivalence of braided monoidal categories.
\end{thm}

\begin{proof}
By Schur--Weyl duality \eqref{eq-schurweyl} and by the construction of 
the categories $\SuSp$ as quotients of $\InSp$, we have quotient functors $\pi_\infty^N$ and $\pi^N$ for $N\in\Z_{\geq 0}$ such that
\begin{equation}\label{eq-last}
\begin{gathered}
\xymatrix{
\Mat(\Hee)\ar[r]^{\pi^N} & \Repas\\
\Mat(\InSp)\ar[r]_{\pi_{\infty}^N}\ar[u]^{\Gami} & \Mat(\SuSp)\ar[u]_{\Gam}
}
\end{gathered}
\end{equation}
commutes. Here the functor $\Gami$ is an idempotented version of the inverse of the isomorphism $\Phi^{\infty}_{q\mathrm{SW}}$ from \fullref{lem-hecke}.

Fix two objects $\vec{k}\in X^L$ and 
$\vec{l}\in X^L$ of $\InSp$ and suppose that $N$ is greater than the sum of the integer values of the entries of $\vec{k}$ (ignoring their colors). Then, 
by \eqref{eq-schurweyl}, \fullref{thm-equi}, the commutativity of \eqref{eq-last} and the fullness of $\pi^N_{\In}$, we have
\begin{align*}
\dim(\Hom_{\Hee}(\Gami(\vec{k}),\Gami(\vec{l})))&=
\dim(\Hom_{\Repas}(\pi^N(\Gami(\vec{k})),\pi^N(\Gami(\vec{l}))))\\
%&= \dim(\Hom_{\SuSp}(\Gam^{-1}(\pi^N(\Gami(\vec{k}))),\Gam^{-1}(\pi^N(\Gami(\vec{l})))))\\
&= \dim(\Hom_{\SuSp}(\pi_\infty^N(\vec{k}),\pi_\infty^N(\vec{l})))\\
&= \dim(\Hom_{\InSp}(\vec{k},\vec{l})).
\end{align*}
$\Gami$ 
is clearly essentially surjective 
and a braided monoidal functor and the theorem follows.
\end{proof}
%
%
%%%%%%%%%%%%%%%%%%%%%%%%%%%%%%%%%%%%%%%%
%%%                                  %%%
%%%       Applications               %%%
%%%                                  %%%
%%%%%%%%%%%%%%%%%%%%%%%%%%%%%%%%%%%%%%%%
\section{Applications}\label{sec-applications}
%%%%%%%%%%%%%%%%%%%%%%%%%%%%%%%%%%%%%%%%%%%%%%%
In this section we write $\cal{L}_D$ for \textit{diagrams 
of framed, oriented links} $\cal{L}$, $b^K_D$ for \textit{diagrams 
of braids in} $K$\textit{-strands} and $\overline{b}^K_D$ for \textit{closures} of such braid diagrams. 
We consider labelings of the connected components of $\cal{L}$ and of braids by Young diagrams $\lambda^i$. If $\cal{L}$ is a $d$-component link, then we write $\cal{L}(\vec{\lambda})$ for its labeling by a vector of Young diagrams $\vec{\lambda}=(\lambda^1,\dots,\lambda^d)$, and use an analogous notation for labeled link and braid diagrams. If not mentioned otherwise, then all appearing links and related concepts are assumed to be framed and oriented from now on.

Let $\cal{L}_D(\vec{\lambda}) = \overline{b}^K_D(\vec{\lambda})$ be a diagram of a framed, oriented, labeled link given as a braid closure. The following process associates to $b^K_D(\vec{\lambda})$ an element $p_{K^{\prime}}(\tilde b^{K^{\prime}}_D)e_q(\vec{\lambda})$
of $H_{K^{\prime}}(q)\cong\End_{\InSp}(\vec{K}^{\prime})$:
\[
\raisebox{.085cm}{\xy
(0,0)*{
\begin{tikzpicture}[scale=.3]
	\draw [very thick, ->] (0,-1) to (0,1);
	\node at (-0.5,0) {\tiny $\lambda^i$};
	\node at (0,-1.5) {\tiny $\lambda^i\in\Lambda^+(K_i)$};
\end{tikzpicture}
};
\endxy}\quad
\xrightarrow{\text{cable}}
\quad
\raisebox{.085cm}{\xy
(0,0)*{
\begin{tikzpicture}[scale=.3]
	\draw [very thick, ->] (-2,-1) to (-2,1);
	\draw [very thick, ->] (-1,-1) to (-1,1);
	\draw [very thick, ->] (1,-1) to (1,1);
	\draw [very thick, ->] (2,-1) to (2,1);
	\node at (0.1,0) {\tiny $\cdots$};
	\node at (0,-1.5) {\tiny $K_i$\text{-strands}};
\end{tikzpicture}
};
\endxy}\quad
\xrightarrow{p_{K_i}(\cdot)}
\quad
p_{K_i}\left(
\raisebox{.085cm}{\xy
(0,0)*{
\begin{tikzpicture}[scale=.3]
	\draw [very thick, ->] (-2,-1) to (-2,1);
	\draw [very thick, ->] (-1,-1) to (-1,1);
	\draw [very thick, ->] (1,-1) to (1,1);
	\draw [very thick, ->] (2,-1) to (2,1);
	\node at (0.1,0) {\tiny $\cdots$};
	\node at (0,-1.5) {\tiny $K_i$\text{-strands}};
\end{tikzpicture}
};
\endxy}
\right)e_q(\lambda^i)
=
\xy
(0,0)*{
\begin{tikzpicture}[scale=.3]
	\draw [very thick, ->] (-2,-1) to (-2,1);
	\draw [very thick, ->] (-1,-1) to (-1,1);
	\draw [very thick, ->] (1,-1) to (1,1);
	\draw [very thick, ->] (2,-1) to (2,1);
	\draw [very thick] (-2.1,-3) rectangle (2.1,-1);
	\draw [very thick] (-2,-5) to (-2,-3);
	\draw [very thick] (-1,-5) to (-1,-3);
	\draw [very thick] (1,-5) to (1,-3);
	\draw [very thick] (2,-5) to (2,-3);
	\node at (0.1,0) {\tiny $\cdots$};
	\node at (0.1,-4) {\tiny $\cdots$};
	\node at (0,-2) {\tiny $e_q(\lambda^i)$};
	\node at (-2,1.4) {\tiny $1$};
	\node at (-1,1.4) {\tiny $1$};
	\node at (1,1.4) {\tiny $1$};
	\node at (2,1.4) {\tiny $1$};
	\node at (-2,-5.4) {\tiny $1$};
	\node at (-1,-5.4) {\tiny $1$};
	\node at (1,-5.4) {\tiny $1$};
	\node at (2,-5.4) {\tiny $1$};
\end{tikzpicture}
};
\endxy
\]
where the last equality is by \fullref{lem-hecke} and
we write $p_{K_i}$ for the 
Iwahori--Hecke algebra representation of the braid group on $K_i$ strands. The first step replaces strands labeled by a Young diagram $\lambda^i$ with $K_i$ nodes in the braid diagram $b^K_D$ by $K_i$ parallel strands. This results in a new braid $\tilde b^{K^{\prime}}_D$ where $K^{\prime}$ indicates the number of strands. In the second step this cabled braid is interpreted as an element of 
the Iwahori--Hecke algebra, or equivalently, as a web in $\InSp$, with an idempotent $e_q(\lambda^i)$ placed on the cable of each previously $\lambda^i$ labeled strand.

\subsection{The colored HOMFLY--PT polynomial via \texorpdfstring{$\InSp$}{inftywebgr}}\label{sub-homfly}
In this subsection we work over the ground field $\Ca=\C_q(a)$, 
with $a$ being a generic parameter. We will use the $\Ca$-valued \textit{Jones--Ocneanu trace} $\tr(\cdot)$ on the direct sum of all Iwahori--Hecke algebras $H_{\In}(q)\!=\!\bigoplus_{K\in\Z_{\geq 0}}\! H_K(q)$.
The definition of $\tr(\cdot)$ 
can be found in \cite[Section 5]{jones2} 
(which can be easily adapted to our notation). We will use it in the form of the following lemma. 

\begin{lem}\label{lem-diatrace}
Given a web $u\in\End_{\InSp}(\vec{K})$,
\[
\tr(u)=
\xy
(0,0)*{
\begin{tikzpicture}[scale=.3]
	\draw [very thick] (-1.25,-.5) rectangle (1.25,1.5);
	\draw [very thick, ->] (-1,1.5) to (-1,2.5);
	\draw [very thick, ->] (1,1.5) to (1,2.5);
	\draw [very thick, directed=0.3] (-1,-1.5) to (-1,-.5);
	\draw [very thick, directed=0.3] (1,-1.5) to (1,-.5);
	\draw [very thick, directed=0.55] (3,2.5) to (3,-1.5);
	\draw [very thick, directed=0.55] (5,2.5) to (5,-1.5);
	\draw [very thick] (1,2.5) to [out=90,in=180] (2,3.5) to [out=0,in=90] (3,2.5);
	\draw [very thick] (-1,2.5) to [out=90,in=180] (2,5.5) to [out=0,in=90] (5,2.5);
	\draw [very thick] (1,-1.5) to [out=270,in=180] (2,-2.5) to [out=0,in=270] (3,-1.5);
	\draw [very thick] (-1,-1.5) to [out=270,in=180] (2,-4.5) to [out=0,in=270] (5,-1.5);
	\node at (0,.5) {\tiny $u$};
	\node at (2,3) {\tiny $1$};
	\node at (2,6) {\tiny $1$};
	\node at (2,-2) {\tiny $1$};
	\node at (2,-5) {\tiny $1$};
	\node at (2,4.75) {\tiny $\vdots$};
	\node at (0.1,2.75) {\tiny $\cdots$};
	\node at (4,2.75) {\tiny $\cdots$};
	\node at (2,-3.25) {\tiny $\vdots$};
	\node at (0.1,-1.8) {\tiny $\cdots$};
	\node at (4,-1.8) {\tiny $\cdots$};
\end{tikzpicture}
};
\endxy\in\Ca,
\]
where the closed diagram can be 
evaluated
by using the relations in $\InSp$ and additionally 
\begin{equation}\label{eq-newmoves}
\xy
(0,0)*{
\begin{tikzpicture} [scale=.3]
\draw[very thick] (0,-1) circle (1.5cm);
\draw[very thick, ->] (-1.5,-.95) to (-1.5,-.94);
\node at (-2.15,-1) {\tiny $1$};
\end{tikzpicture}}
\endxy
=\frac{a-a^{-1}}{q-q^{-1}}
,\quad\quad
\xy
(0,0)*{
\begin{tikzpicture}[scale=.3]
	\draw [very thick, green, directed=.55] (0,-1) to (0,.75);
	\draw [very thick, directed=.55] (0,.75) to [out=30,in=270] (1,2.5);
	\draw [very thick, directed=.55] (0,.75) to [out=150,in=270] (-1,2.5); 
	\draw [very thick, directed=.55] (1,-2.75) to [out=90,in=330] (0,-1);
	\draw [very thick, directed=.55] (-1,-2.75) to [out=90,in=210] (0,-1);
	\draw [very thick] (-1,-4) to (-1,-2.75);
	\draw [very thick] (-1,2.5) to (-1,3.75);
	\draw [very thick] (1,2.5) to [out=90, in=180] (2,3.25) to [out=0, in=90] (3,2.5);
	\draw [very thick] (1,-2.75) to [out=270, in=180] (2,-3.5) to [out=0, in=270] (3,-2.75);
	\draw [very thick, directed=0.55] (3,2.5) to (3,-2.75);
	\node at (-1,4.25) {\tiny $1$};
	\node at (-1,-4.5) {\tiny $1$};
	\node at (-0.5,0) {\tiny $2$};
	\node at (2,3.75) {\tiny $1$};
\end{tikzpicture}
};
\endxy
=
\frac{aq^{-1}-a^{-1}q}{q-q^{-1}}
\xy
(0,0)*{
\begin{tikzpicture}[scale=.3]
	\draw [very thick, directed=0.55] (-1,-4) to (-1,3.75);
	\node at (-1,4.25) {\tiny $1$};
	\node at (-1,-4.5) {\tiny $1$};
\end{tikzpicture}
};
\endxy
\end{equation}
\end{lem}

\begin{proof}
By \fullref{prop-green} and \fullref{cor-green}: 
any given web $u\in\End_{\InSp}(\vec{K})$ can be expressed 
using black or green edges with labels at most $2$. 
Using \fullref{lem-hecke}
and additionally \cite[Section 4.2]{ras}, where
Rasmussen's singular crossings correspond to 
green dumbbells with label $2$, 
provides the statement. Note that 
Rasmussen's relations II and III are already part of our 
diagrammatic calculus.
\end{proof}

\begin{defn}(\textbf{The colored HOMFLY--PT polynomial})\label{defn-colHomfly}
Let $\cal{L}_D(\vec{\lambda}) = \overline{b}^K_D(\vec{\lambda})$ be a diagram of a framed, oriented, labeled link $\cal{L}(\vec{\lambda})$ given as a braid closure.

The \textit{colored HOMFLY--PT polynomial} 
of $\cal{L}(\vec{\lambda})$, 
denoted by 
$\cal{P}^{a,q}(\cal{L}(\vec{\lambda}))$, 
is defined via
\[
\cal{P}^{a,q}(\cal{L}(\vec{\lambda}))=\tr(p_{K^{\prime}}(\tilde b^{K^{\prime}}_D)e_q(\vec{\lambda}))\in\Ca,
\]
where
$e_q(\vec{\lambda})$ is a tensor product of $e_q(\lambda^i)$'s, as described above.
\end{defn}

This polynomial is independent of all choices involved and an invariant of framed, oriented, colored links. Up to different conventions, this is shown for example in \cite[Corollary 4.5]{xsz}.

\begin{rem}\label{rem-norm}
In fact, \fullref{defn-colHomfly} gives the framing 
dependent, 
unnormalized version of the colored HOMFLY--PT polynomial. As usual, 
the polynomial can be normalized 
by fixing the value of the unknot to be $1$ (instead of 
$\frac{a-a^{-1}}{q-q^{-1}}$ as in our convention) and 
one can get rid of the 
framing dependence by scaling with a factor coming 
from Reidemeister 1 moves, see for example \cite[Definition 6.1]{jones2}.
We suppress these distinctions in the following.
\end{rem}

Note that \fullref{lem-diatrace} provides a method to calculate the colored HOMFLY--PT polynomials $\cal{P}^{a,q}(\cdot)$ using the web category $\InSp$.

\begin{prop}(\textbf{The colored HOMFLY--PT symmetry})\label{prop-homfly}
We have
\[
\cal{P}^{a,q}(\cal{L}(\vec{\lambda}))\;=\;(-1)^{c}\;
\cal{P}^{a,q^{-1}}(\cal{L}(\vec{\lambda}^{\phantom{.}\T})),
\]
where 
$\vec{\lambda}^{\phantom{.}\T}=((\lambda^1)^{\T},\dots,(\lambda^d)^{\T})$ and $c$ is the sum of the number of nodes in the $\lambda^i$ for $1\leq i\leq d$.
\end{prop}

This symmetry is not new: it can be deduced from \cite[Section 9]{luk} and has been studied in \cite{lp} 
and \cite[Proposition 4.4]{gm}. 
In our framework it follows directly from the green--red symmetry in $\InSp$.

\begin{proof}
We only give a proof for the case of knots $\cal{K}$. The proof for links is analogous, but the notation is more involved. We denote by $I_{\mathrm{gr}}$ the involution on $\InSp$ given by the green--red symmetry, and by $I_q$ the involution on $\C_{a,q}$ which inverts the variable $q$.

\textbf{Claim.} For $u\in\End_{\InSp}(\vec{K})$ we have
\begin{equation}\label{eq-claim}
\mathrm{tr}(u)=(-1)^K I_q(\mathrm{tr}(I_q(I_{\mathrm{gr}}(u)))).
\end{equation} 
It suffices to prove $\mathrm{tr}(u)=(-1)^K I_q(\mathrm{tr}(I_{\mathrm{gr}}(u)))$ where $u$ is a primitive web (a morphism that consists of a single web with coefficient $1$, which is thus invariant 
under $I_q$). In \fullref{lem-diatrace} we have met evaluation relations for monochromatic green webs of edge label at most 2, but clearly analogous relations can be derived for red and mixed webs. In fact, all necessary evaluation relations are invariant under $I_{\mathrm{gr}}$ and $I_q$, except the two relations in \eqref{eq-newmoves}. The circle relation is $I_{\mathrm{gr}}$ invariant, but acquires a sign under $I_{q}$. The following computation shows that the green and red bubble relations also respect \eqref{eq-claim}:
\begin{equation*}\label{eq-notagain}
\xy
(0,0)*{
\begin{tikzpicture}[scale=.3]
	\draw [very thick, mycolor, directed=.55] (0,-1) to (0,.75);
	\draw [very thick, directed=.55] (0,.75) to [out=30,in=270] (1,2.5);
	\draw [very thick, directed=.55] (0,.75) to [out=150,in=270] (-1,2.5); 
	\draw [very thick, directed=.55] (1,-2.75) to [out=90,in=330] (0,-1);
	\draw [very thick, directed=.55] (-1,-2.75) to [out=90,in=210] (0,-1);
	\draw [very thick] (-1,-4) to (-1,-2.75);
	\draw [very thick] (-1,2.5) to (-1,3.75);
	\draw [very thick] (1,2.5) to [out=90, in=180] (2,3.25) to [out=0, in=90] (3,2.5);
	\draw [very thick] (1,-2.75) to [out=270, in=180] (2,-3.5) to [out=0, in=270] (3,-2.75);
	\draw [very thick, directed=0.55] (3,2.5) to (3,-2.75);
	\node at (-1,4.25) {\tiny $1$};
	\node at (-1,-4.5) {\tiny $1$};
	\node at (-0.5,0) {\tiny $2$};
	\node at (2,3.75) {\tiny $1$};
\end{tikzpicture}
};
\endxy
\;\;
\stackrel{\eqref{eq-dumb}}{=}
\;\;
[2]
\xy
(0,0)*{
\begin{tikzpicture}[scale=.3] 
	\draw [very thick, directed=.55] (1,-2.75) to (1,2.5);
	\draw [very thick, directed=.55] (-1,-2.75) to (-1,2.5);
	\draw [very thick] (-1,-4) to (-1,-2.75);
	\draw [very thick] (-1,2.5) to (-1,3.75);
	\draw [very thick] (1,2.5) to [out=90, in=180] (2,3.25) to [out=0, in=90] (3,2.5);
	\draw [very thick] (1,-2.75) to [out=270, in=180] (2,-3.5) to [out=0, in=270] (3,-2.75);
	\draw [very thick, directed=0.55] (3,2.5) to (3,-2.75);
	\node at (-1,4.25) {\tiny $1$};
	\node at (-1,-4.5) {\tiny $1$};
	\node at (2,3.75) {\tiny $1$};
\end{tikzpicture}
};
\endxy
\;\;
-
\xy
(0,0)*{
\begin{tikzpicture}[scale=.3]
	\draw [very thick, green, directed=.55] (0,-1) to (0,.75);
	\draw [very thick, directed=.55] (0,.75) to [out=30,in=270] (1,2.5);
	\draw [very thick, directed=.55] (0,.75) to [out=150,in=270] (-1,2.5); 
	\draw [very thick, directed=.55] (1,-2.75) to [out=90,in=330] (0,-1);
	\draw [very thick, directed=.55] (-1,-2.75) to [out=90,in=210] (0,-1);
	\draw [very thick] (-1,-4) to (-1,-2.75);
	\draw [very thick] (-1,2.5) to (-1,3.75);
	\draw [very thick] (1,2.5) to [out=90, in=180] (2,3.25) to [out=0, in=90] (3,2.5);
	\draw [very thick] (1,-2.75) to [out=270, in=180] (2,-3.5) to [out=0, in=270] (3,-2.75);
	\draw [very thick, directed=0.55] (3,2.5) to (3,-2.75);
	\node at (-1,4.25) {\tiny $1$};
	\node at (-1,-4.5) {\tiny $1$};
	\node at (-0.5,0) {\tiny $2$};
	\node at (2,3.75) {\tiny $1$};
\end{tikzpicture}
};
\endxy
\;\;
=
\frac{aq-a^{-1}q^{-1}}{q-q^{-1}}
\xy
(0,0)*{
\begin{tikzpicture}[scale=.3]
	\draw [very thick, directed=0.55] (-1,-4) to (-1,3.75);
	\node at (-1,4.25) {\tiny $1$};
	\node at (-1,-4.5) {\tiny $1$};
\end{tikzpicture}
};
\endxy
\;\;
\xleftrightarrow{q\leftrightarrow q^{-1}}
\;\;
-
\!\!
\xy
(0,0)*{
\begin{tikzpicture}[scale=.3]
	\draw [very thick, green, directed=.55] (0,-1) to (0,.75);
	\draw [very thick, directed=.55] (0,.75) to [out=30,in=270] (1,2.5);
	\draw [very thick, directed=.55] (0,.75) to [out=150,in=270] (-1,2.5); 
	\draw [very thick, directed=.55] (1,-2.75) to [out=90,in=330] (0,-1);
	\draw [very thick, directed=.55] (-1,-2.75) to [out=90,in=210] (0,-1);
	\draw [very thick] (-1,-4) to (-1,-2.75);
	\draw [very thick] (-1,2.5) to (-1,3.75);
	\draw [very thick] (1,2.5) to [out=90, in=180] (2,3.25) to [out=0, in=90] (3,2.5);
	\draw [very thick] (1,-2.75) to [out=270, in=180] (2,-3.5) to [out=0, in=270] (3,-2.75);
	\draw [very thick, directed=0.55] (3,2.5) to (3,-2.75);
	\node at (-1,4.25) {\tiny $1$};
	\node at (-1,-4.5) {\tiny $1$};
	\node at (-0.5,0) {\tiny $2$};
	\node at (2,3.75) {\tiny $1$};
\end{tikzpicture}
};
\endxy
\end{equation*} We note that in the computation of 
$\mathrm{tr}(u)$ via \fullref{lem-diatrace} strands can only be removed by circle moves and bubble moves. Both of these acquire a sign under $I_q$, which causes the factor $(-1)^K$ in \eqref{eq-claim}. This proves the claim.

Let $b_D^K$ be a braid diagram that closes to a diagram of $\cal{K}$ and suppose that $\cal{K}$ is labeled by a Young diagram $\lambda$ with $L$ nodes. Let $\tilde b_D^{KL}$ be the $L$-fold cable of the braid diagram $b_D^K$. 

Now we have
\begin{align*}
\cal{P}^{a,q}(\cal{K}(\lambda)) 
=
\mathrm{tr}(p_{KL}(\tilde b_D^{KL}) e_q(\lambda)^{\otimes K}) 
&= 
(-1)^{K L} I_q(\mathrm{tr}(I_q(I_{\mathrm{gr}}(p_{KL}(\tilde b_D^{KL}) e_q(\lambda)^{\otimes K})))) 
\\
&=
(-1)^{KL +\mathrm{cr}L^2} I_q(\mathrm{tr}(p_{KL}(\tilde b_D^{KL}) e_q(\lambda^{\T})^{\otimes K}))
\\
&=
(-1)^{L}
\cal{P}^{a,q^{-1}}(\cal{K}({\lambda}^{\T})),
\end{align*}
where $\mathrm{cr}$ is the number of crossings of $b_D^K$. Here we 
have used \eqref{eq-claim} and that $I_qI_{\mathrm{gr}}$ acts as $-1$ on black 
crossings, see \fullref{ex-oneoneok}, while sending $e_q(\lambda)$ to $e_q(\lambda^{\T})$ plus a commutator (which is zero in the trace), see \fullref{lem-symmidem}. Moreover, $(-1)^{KL +\mathrm{cr}L^2}= (-1)^L$
since $cr \equiv K-1 \bmod 2$ as $b_D^K$ closes into a knot.
\end{proof}

\subsection{The colored \texorpdfstring{$\sln$}{sln}-link polynomials via the categories \texorpdfstring{$\SuSp$}{Nwebgr}}\label{sub-linkpoly}

Recall that the \textit{colored Reshetikhin--Turaev} $\sln$\textit{-link 
polynomial} $\cal{RT}^{q^N,q}(\cal{L}(\vec{\lambda}))$ 
are determined by the corresponding colored 
HOMFLY--PT polynomials $\cal{P}^{a,q}(\cal{L}(\vec{\lambda}))$ by specializing $a=q^N$. Alternatively, they can be computed 
directly inside the categories $\SuSp$ from a framed, oriented, labeled link diagram as follows:
\begin{itemize}
\item First we replace all $\lambda$-labeled strands in the link diagram by cables equipped with the diagrammatic idempotent $e_q(\lambda)$, written in monochromatic green webs.
\item The resulting diagram will contain downward-oriented green edges of label $k$, which we replace by upward-oriented green edges of label $N-k$. Simultaneously, caps and cups are replaced by splits and merges
\[
\raisebox{-0.375cm}{
\xy
(0,0)*{
\begin{tikzpicture} [scale=1]
\draw[very thick, green, ->] (0,0) to [out=90,in=180] (.25,.5) to [out=0,in=90] (.5,0);
\node at (0,-.15) {\tiny $k$};
\node at (.5,-.15) {\tiny $N\! -\! k$};
\end{tikzpicture}
}
\endxy}
=
\xy
(0,0)*{
\begin{tikzpicture}[scale=.3]
	\draw [very thick, green, dotted, directed=.55] (0, .75) to (0,2.5);
	\draw [very thick, green, directed=.55] (1,-1) to [out=90,in=330] (0,.75);
	\draw [very thick, green, directed=.55] (-1,-1) to [out=90,in=210] (0,.75); 
	\node at (0, 3) {\tiny $N$};
	\node at (-1,-1.5) {\tiny $k$};
	\node at (1,-1.5) {\tiny $N\! -\! k$};
\end{tikzpicture}
};
\endxy
\quad,\quad
\raisebox{0.525cm}{\xy
(0,0)*{
\begin{tikzpicture} [scale=1]
\draw[very thick, green, ->] (0,1) to [out=270,in=180] (.25,.5) to [out=0,in=270] (.5,1);
\node at (0,1.15) {\tiny $k$};
\node at (.5,1.15) {\tiny $N\! -\! k$};
\end{tikzpicture}
}
\endxy}
=
\xy
(0,0)*{
\begin{tikzpicture}[scale=.3]
	\draw [very thick, green, dotted, directed=.55] (0,-1) to (0,.75);
	\draw [very thick, green, directed=.55] (0,.75) to [out=30,in=270] (1,2.5);
	\draw [very thick, green, directed=.55] (0,.75) to [out=150,in=270] (-1,2.5); 
	\node at (0, -1.5) {\tiny $N$};
	\node at (-1,3) {\tiny $k$};
	\node at (1,3) {\tiny $N\!-\! k$};
\end{tikzpicture}
};
\endxy
\] 
\item The result is a morphism in $\SuSp$ between objects consisting only of entries $0$ and $N_g$. It follows from \fullref{thm-equi} that this $\Hom$-space is one-dimensional. Thus, the framed, oriented, labeled link diagram determines a polynomial, which is the desired colored 
Reshetikhin--Turaev $\sln$-link 
polynomial.
\end{itemize}

Recall from \fullref{rem-whatwenotdo} that this approach relies on the fact that $\repas$ contains the duality isomorphisms $\bV_q^k\C_q^N\cong(\bV_q^{N-k}\C_q^N)^*$. 
In \fullref{rem-duals} we sketch how to include duals in diagrammatic presentations of $\Repas$ and $\RepMN$ and, thus, to compute the corresponding Reshetikhin--Turaev $\gln$ or $\glmn$-link invariants.
%%%%%%%%%%%%%%%%%%%%%%%%%%%%%%%%%%%%%%%%
%%%                                  %%%
%%%       Real super stuff           %%%
%%%                                  %%%
%%%%%%%%%%%%%%%%%%%%%%%%%%%%%%%%%%%%%%%%
\section{Generalization to webs for \texorpdfstring{$\glMN$}{gl(N|M)}}\label{sec-gen}
%%%%%%%%%%%%%%%%%%%%%%%%%%%%%%%%%%%%%%%%%%%%%%%
We now give a diagrammatic presentation 
of $\RepMN$, the (additive closure of the) braided monoidal 
category of 
$\Usu$-modules tensor 
generated by 
the exterior $\bV_q^k\C_q^{N|M}$ and 
the symmetric $\Sym_q^l\C_q^{N|M}$ powers of the vector representation $\C_q^{N|M}$ of $\Usu$.
The diagrammatic presentation is given by the following quotient of $\InSp$.

\begin{defn}\label{defn-spidnext}
The category $\MNSp$ 
is the quotient category obtained from $\InSp$ by imposing the 
\textit{not-a-hook relation}, that is
\begin{equation*}\label{eq-idemdead}
e_q(\mathrm{box}_{N+1,M+1})=0\comm{\quad,\quad\text{for }\lambda_{N,M}=
\raisebox{-0.215cm}{\xy
(0,0)*{\underbrace{\begin{Young} &\cr &\cr\end{Young}}_{= M+1}};
\endxy}\biggr\}{\scriptstyle = N+1}},
\end{equation*}
where $\mathrm{box}_{N+1,M+1}$ is the box-shaped Young diagram with $N+1$ rows and $M+1$ columns.
\end{defn}

Note that $\MNSp$ inherits the braiding $\beta^{\Sup}_{\cdot,\cdot}$ from $\InSp$.

\begin{ex}\label{ex-general1}
If we take $M=0$, then $\mathrm{box}_{N+1,1}$ is a column Young diagram with $N+1$ nodes and the corresponding not-a-hook relation is just the exterior relation \eqref{eq-alt}. In this case we have that $N|0\text{-}\!\textbf{Web}_{\mathrm{gr}}$ is $\SuSp$ and $\mathfrak{gl}_{N|0}\text{-}\textbf{Mod}_{\mathrm{es}}$ is isomorphic to $\Repas$.
\end{ex}

\begin{ex}\label{ex-general2}
If we take $M=N=1$, then we have
\[
\tilde e_q(\mathrm{box}_{2,2})=\tilde e_q\left(\;\xy(0,0)*{\begin{Young} &\cr &\cr\end{Young}};\endxy\;\right)= \frac{1}{[2]^4}
\xy
(0,0)*{
\begin{tikzpicture}[scale=.3] 
	\draw [very thick, green, directed=.75] (-2,1.25) to (-2,2);
	\draw [very thick] (-1,0.75) to (-2,1.25);
	\draw [very thick] (-3,0.75) to (-2,1.25);
	\draw [very thick, directed=.65] (-2,2) to (-1,2.5);
	\draw [very thick, directed=.65] (-2,2) to (-3,2.5);
	\draw [very thick, mycolor, directed=.75] (-2,-2.25) to (-2,-1.5);
	\draw [very thick] (-2,-1.5) to (-1,-1);
	\draw [very thick] (-2,-1.5) to (-3,-1);
	\draw [very thick] (-1,-2.75) to (-2,-2.25);
	\draw [very thick] (-3,-2.75) to (-2,-2.25);
	\draw [very thick, green, directed=.75] (2,1.25) to (2,2);
	\draw [very thick] (3,0.75) to (2,1.25);
	\draw [very thick] (1,0.75) to (2,1.25);
	\draw [very thick, directed=.65] (2,2) to (3,2.5);
	\draw [very thick, directed=.65] (2,2) to (1,2.5);
	\draw [very thick, mycolor, directed=.75] (2,-2.25) to (2,-1.5);
	\draw [very thick] (2,-1.5) to (3,-1);
	\draw [very thick] (2,-1.5) to (1,-1);
	\draw [very thick] (3,-2.75) to (2,-2.25);
	\draw [very thick] (1,-2.75) to (2,-2.25);
	\draw [very thick, directed=.55] (-3,-1) to (-3,0.75);
	\draw [very thick, directed=.55] (3,-1) to (3,0.75);
	\draw [very thick, ->] (-1,-1) to (1,0.75);
	\draw [very thick, ->] (-0.3333,0.166666) to (-1,0.75);
	\draw [very thick] (1,-1) to (0.416666,-0.416666);
	\draw [very thick, ->] (1,-4.5) to (-1,-2.75);
	\draw [very thick, ->] (0.416666,-3.3333) to (1,-2.75);
	\draw [very thick] (-1,-4.5) to (-0.3333,-3.916666);
	\draw [very thick, directed=.55] (-3,-4.5) to (-3,-2.75);
	\draw [very thick, directed=.55] (3,-4.5) to (3,-2.75);
	\node at (-3,3) {\tiny $1$};
	\node at (-1,3) {\tiny $1$};
	\node at (1,3) {\tiny $1$};
	\node at (3,3) {\tiny $1$};
	\node at (-3,-5) {\tiny $1$};
	\node at (-1,-5) {\tiny $1$};
	\node at (1,-5) {\tiny $1$};
	\node at (3,-5) {\tiny $1$};
	\node at (-2.55,1.625) {\tiny $2$};
	\node at (-2.55,-1.875) {\tiny $2$};
	\node at (2.55,1.625) {\tiny $2$};
	\node at (2.55,-1.875) {\tiny $2$};
\end{tikzpicture}
};
\endxy
=- \frac{1}{[2]^4}
\xy
(0,0)*{
\begin{tikzpicture}[scale=.3]
	\draw [very thick] (1,-1) to (0,-.5);
	\draw [very thick] (-1,-1) to (0,-.5);
	\draw [very thick] (0,.25) to (1,0.75);
	\draw [very thick] (0,.25) to (-1,0.75);
	\draw [very thick, green, directed=.75] (0,-0.5) to (0,.25);
	\draw [very thick, green, directed=.75] (-2,1.25) to (-2,2);
	\draw [very thick, directed=.15] (-1,0.75) to (-2,1.25);
	\draw [very thick] (-3,0.75) to (-2,1.25);
	\draw [very thick, directed=.65] (-2,2) to (-1,2.5);
	\draw [very thick, directed=.65] (-2,2) to (-3,2.5);
	\draw [very thick, mycolor, directed=.75] (-2,-2.25) to (-2,-1.5);
	\draw [very thick, ->] (-2,-1.5) to (-1,-1);
	\draw [very thick] (-2,-1.5) to (-3,-1);
	\draw [very thick] (-1,-2.75) to (-2,-2.25);
	\draw [very thick] (-3,-2.75) to (-2,-2.25);
	\draw [very thick, green, directed=.75] (2,1.25) to (2,2);
	\draw [very thick] (3,0.75) to (2,1.25);
	\draw [very thick, directed=.15] (1,0.75) to (2,1.25);
	\draw [very thick, directed=.65] (2,2) to (3,2.5);
	\draw [very thick, directed=.65] (2,2) to (1,2.5);
	\draw [very thick, mycolor, directed=.75] (2,-2.25) to (2,-1.5);
	\draw [very thick] (2,-1.5) to (3,-1);
	\draw [very thick, ->] (2,-1.5) to (1,-1);
	\draw [very thick] (3,-2.75) to (2,-2.25);
	\draw [very thick] (1,-2.75) to (2,-2.25);
	\draw [very thick, directed=.55] (-3,-1) to (-3,0.75);
	\draw [very thick, directed=.55] (3,-1) to (3,0.75);
	\draw [very thick, ->] (1,-4.5) to (-1,-2.75);
	\draw [very thick, ->] (0.416666,-3.3333) to (1,-2.75);
	\draw [very thick] (-1,-4.5) to (-0.3333,-3.916666);
	\draw [very thick, directed=.55] (-3,-4.5) to (-3,-2.75);
	\draw [very thick, directed=.55] (3,-4.5) to (3,-2.75);
	\node at (-3,3) {\tiny $1$};
	\node at (-1,3) {\tiny $1$};
	\node at (1,3) {\tiny $1$};
	\node at (3,3) {\tiny $1$};
	\node at (-3,-5) {\tiny $1$};
	\node at (-1,-5) {\tiny $1$};
	\node at (1,-5) {\tiny $1$};
	\node at (3,-5) {\tiny $1$};
	\node at (-2.55,1.625) {\tiny $2$};
	\node at (-2.55,-1.875) {\tiny $2$};
	\node at (2.55,1.625) {\tiny $2$};
	\node at (2.55,-1.875) {\tiny $2$};
	\node at (-.55,-0.125) {\tiny $2$};
\end{tikzpicture}
};
\endxy
=
\frac{1}{[2]^4}
\xy
(0,0)*{
\begin{tikzpicture}[scale=.3]
	\draw [very thick] (1,-1) to (0,-.5);
	\draw [very thick] (-1,-1) to (0,-.5);
	\draw [very thick] (0,.25) to (1,0.75);
	\draw [very thick] (0,.25) to (-1,0.75);
	\draw [very thick, mycolor, directed=.75] (0,-0.5) to (0,.25);
	\draw [very thick, green, directed=.75] (-2,1.25) to (-2,2);
	\draw [very thick, directed=.15] (-1,0.75) to (-2,1.25);
	\draw [very thick] (-3,0.75) to (-2,1.25);
	\draw [very thick, directed=.65] (-2,2) to (-1,2.5);
	\draw [very thick, directed=.65] (-2,2) to (-3,2.5);
	\draw [very thick, mycolor, directed=.75] (-2,-2.25) to (-2,-1.5);
	\draw [very thick, ->] (-2,-1.5) to (-1,-1);
	\draw [very thick] (-2,-1.5) to (-3,-1);
	\draw [very thick] (-1,-2.75) to (-2,-2.25);
	\draw [very thick] (-3,-2.75) to (-2,-2.25);
	\draw [very thick, green, directed=.75] (2,1.25) to (2,2);
	\draw [very thick] (3,0.75) to (2,1.25);
	\draw [very thick, directed=.15] (1,0.75) to (2,1.25);
	\draw [very thick, directed=.65] (2,2) to (3,2.5);
	\draw [very thick, directed=.65] (2,2) to (1,2.5);
	\draw [very thick, mycolor, directed=.75] (2,-2.25) to (2,-1.5);
	\draw [very thick] (2,-1.5) to (3,-1);
	\draw [very thick, ->] (2,-1.5) to (1,-1);
	\draw [very thick] (3,-2.75) to (2,-2.25);
	\draw [very thick] (1,-2.75) to (2,-2.25);
	\draw [very thick, directed=.55] (-3,-1) to (-3,0.75);
	\draw [very thick, directed=.55] (3,-1) to (3,0.75);
	\draw [very thick, ->] (1,-4.5) to (-1,-2.75);
	\draw [very thick, ->] (0.416666,-3.3333) to (1,-2.75);
	\draw [very thick] (-1,-4.5) to (-0.3333,-3.916666);
	\draw [very thick, directed=.55] (-3,-4.5) to (-3,-2.75);
	\draw [very thick, directed=.55] (3,-4.5) to (3,-2.75);
	\node at (-3,3) {\tiny $1$};
	\node at (-1,3) {\tiny $1$};
	\node at (1,3) {\tiny $1$};
	\node at (3,3) {\tiny $1$};
	\node at (-3,-5) {\tiny $1$};
	\node at (-1,-5) {\tiny $1$};
	\node at (1,-5) {\tiny $1$};
	\node at (3,-5) {\tiny $1$};
	\node at (-2.55,1.625) {\tiny $2$};
	\node at (-2.55,-1.875) {\tiny $2$};
	\node at (2.55,1.625) {\tiny $2$};
	\node at (2.55,-1.875) {\tiny $2$};
	\node at (-.55,-0.125) {\tiny $2$};
\end{tikzpicture}
};
\endxy
\]
It is easy to see that $e_q(\mathrm{box}_{2,2})=0$ is equivalent to the relations \cite[(3.3.13a) and (3.3.13b)]{sarth}, \cite[Section 3.6]{g} and \cite[Corollary 6.18]{qs} which are used to describe the ``purely exterior'' representation category $\mathfrak{gl}_{1|1}\text{-}\textbf{Mod}_{\mathrm{e}}$. This category could be presented as monochromatic green subcategory of $1|1\text{-}\!\textbf{Web}_{\mathrm{gr}}$, defined analogously as in \fullref{defn-spid3}.
\end{ex}

To prove that $\MNSp$ gives a diagrammatic presentation of $\RepMN$, we use a version of super $q$-Howe duality between $\Us$ and $\Usu$. For this, we say a dominant integral $\glmn$-weight $\lambda$ is 
\textit{$(m|n,M|N)$-supported} if it corresponds to a Young diagram 
which is simultaneously an $(m|n)$-hook as well as an $(M|N)$-hook\footnote{This is really intended to be $(M|N)$.}.

\begin{thm}(\textbf{Super $q$-Howe duality, super--super version})\label{thm-supersuperHowe} 
We have the following.
\begin{itemize}
\item[(a)] Let $K \in\Z_{\geq 0}$. The actions of $\Us$ and $\Usu$ on 
$\bV_q^{K}(\C_q^{m|n}\otimes \C_q^{N|M})$ commute 
and generate each others commutant.
\item[(b)] There exists an isomorphism 
\[
\bV_q^{\bullet}(\C_q^{m|n}\otimes \C_q^{N|M}) \cong 
(\bV_q^{\bullet}\C_q^{N|M})^{\otimes m}\otimes(\Sym_q^{\bullet}\C_q^{N|M})^{\otimes n}
\] 
of $\Usu$-modules under which
the $\vec{k}$-weight space of 
$\bV_q^{\bullet}(\C_q^{m|n}\otimes \C_q^{N|M})$ 
(considered as a $\Us$-module)
is identified with
\begin{equation*}
\begin{gathered}
\bV_q^{\vec{k}_0}\C_q^{N|M}\otimes\Sym_q^{\vec{k}_1}\C_q^{N|M} =\\
\bV_q^{k_{1}}\C_q^{N|M}\otimes\cdots\otimes\bV_q^{k_{m}}\C_q^{N|M}
\otimes\Sym_q^{k_{m+1}}\C_q^{N|M}\otimes\cdots\otimes\Sym_q^{k_{m+n}}\C_q^{N|M}.
\end{gathered}
\end{equation*}
Here $\vec{k}=(k_1,\dots,k_{m+n})$, $\vec{k}_0=(k_1,\dots,k_m)$ and 
$\vec{k}_1=(k_{m+1},\dots,k_{m+n})$.
\item[(c)] As $\Us\otimes\Usu$-modules, we have a decomposition 
of the form
\[
\bV_q^{K}(\C_q^{m|n}\otimes \C_q^{N|M})
\cong{\textstyle\bigoplus_{\lambda}}L_{m|n}(\lambda)\otimes L_{N|M}(\lambda^{\T}),
\]
where we sum over all 
$(m|n,M|N)$-supported 
$\glmn$-weights $\lambda$ whose entries sum up to $K$. 
This induces a decomposition
\[
\bV_q^{\bullet}(\C_q^{m|n}\otimes \C_q^{N|M}) 
\cong{\textstyle\bigoplus_{\lambda}}L_{m|n}(\lambda)\otimes L_{N|M}(\lambda^{\T}),
\]
where we sum over all 
$(m|n,M|N)$-supported $\glmn$-weights $\lambda$.
\end{itemize}
\end{thm}

\begin{proof}
As before, (a) and (c) are proven in \cite[Theorem 4.2]{qs} 
and only (b) remains to be verified. This works similarly
as in the proof of \fullref{thm-superHowe} and is left to the reader. For a non-quantized version see \cite[Proposition 2.2]{sar}.
\end{proof}

In the statement of this theorem, $\bV_q^{k}\C_q^{N|M}$, $\Sym_q^{l}\C_q^{N|M}$ and $\bV_q^{K}(\C_q^{m|n}\otimes \C_q^{N|M})$ are defined similarly as 
in \fullref{sub-superHowe}, see 
also \cite[Section 3]{qs}. As before we then get:

\begin{cor}\label{cor-functorMN}
There exists a full functor $\Phi_{\Su}^{m|n}\colon\Usd \to \RepMN$, 
which we again call 
the \textit{super $q$-Howe functor}, 
given on objects and morphisms by
\begin{gather*}
\Phi_{\Su}^{m|n}(\vec{k})=\bV_q^{\vec{k}_0}\C_q^{N|M}\otimes\Sym_q^{\vec{k}_1}\C_q^{N|M}
,\quad\quad
\Phi_{\Su}^{m|n}(1_{\vec{l}}\,x 1_{\vec{k}})=f_{\vec{k}}^{\vec{l}}(x).
\end{gather*} 
Everything else is sent to zero.\qed
\end{cor}

In what follows, we denote by $\Usd^{\geq 0}$ the quotient of $\Usd$ obtained by 
killing all $\glmn$-weights with negative entries. 
\begin{cor}\label{cor-something} 
The super $q$-Howe functor $\Phi_{\Su}^{m|n}$ 
from \fullref{cor-functorMN} induces an algebra epimorphism (denoted by the same symbol) as in the diagram
\[
\begin{gathered}
\xymatrix{
\Usd^{\geq 0} \ar@{->}[r]^/-0.3cm/{\cong} \ar@{->>}[d]_{\Phi_{\Su}^{m|n}}&
{\textstyle\bigoplus_{\substack{
(m|n)\text{-} \\ \text{hooks }\lambda}
}}\,
\End_{\C_q}(L_{m|n}(\lambda)) \ar@{}[d]^(0.1){}="a"^(0.9){}="b" \ar@{->>}^{\pi}"a";"b"\\
\End_{\Usu}(\bV_q^{\bullet}(\C_q^{m|n}\otimes \C_q^{N|M})) \ar@{->}[r]_/0.2cm/{\cong}& 
{\textstyle\bigoplus_{\substack{
(m|n,M|N)\text{-} \\ \text{supported }\lambda
}}
}\,
\End_{\C_q}(L_{m|n}(\lambda))\\
}
\end{gathered}
\] 
Under Artin--Wedderburn decompositions, it corresponds to an algebra epimorphism $\pi$, which acts on the summand $\End_{\C_q}(L_{m|n}(\lambda))$ either as an isomorphism or as zero, depending on whether the Young diagram $\lambda$ is $(m|n, M|N)$-supported or not.
\end{cor}
\begin{proof}
First, note that by \fullref{thm-equi2}, $\Usd^{\geq 0}$ is isomorphic to $\Hee_{m+n}^{\mathrm{sort}}$, the sorted version of $\Hee$ with exactly $m$ exterior strands and $n$ symmetric strands. The Artin--Wedderburn decomposition in the top row of the diagram is then given in \cite[Theorem 5.1]{mit}. The bottom Artin--Wedderburn decomposition follows 
directly from part (c) of \fullref{thm-supersuperHowe}. 
\end{proof}

\begin{rem}\label{rem-something}
We obtain from \fullref{cor-something} an alternative proof of the 
presentation of the $q$-Schur superalgebra $S_q(N|M,K)\cong\End_{H_K(q)}((\C_q^{N|M})^{\otimes K})$ from \cite[Theorem 3.13.1]{etk}.
\end{rem}

\begin{lem}\label{lem-ideal}
Under the correspondence
\[
\begin{gathered}
\xymatrix{
\Hee_{m+n}^{\mathrm{sort}} \ar@{<->}[r]^/-0.25cm/{\cong}&
\Usd^{\geq 0} \ar@{<->}[r]^/-0.5cm/{\cong} &
{\textstyle\bigoplus_{\substack{
(m|n)\text{-} \\ \text{hooks }\lambda}}
}\,
\End_{\C_q}(L_{m|n}(\lambda)) 
}
\end{gathered}
\] 
the kernel of the super $q$-Howe functor 
$\Phi_{\Su}^{m|n}$ from \fullref{cor-functorMN} is given by the tensor ideal $I_{\mathrm{box}}$ in $\Hee^{\mathrm{sort}}_{m+n}$ generated by the primitive idempotent $e_q(\mathrm{box}_{N+1,M+1})$.
\end{lem}

\begin{proof}
From the right isomorphism we know that the kernel of $\Phi_{\Su}^{m|n}$ is generated by all $e_q(\lambda^{\T})$ where $\lambda$ is an $(m|n)$-hook, but not an $(M|N)$-hook. Every such $\lambda$ corresponds to a simple $\Usu$-module which appears in a tensor product $L_{N|M}((\mathrm{box}_{N+1,M+1})^\T) \otimes (\C_q^{N|M})^{\otimes K}$ for some $K\in \Z_{\geq 0}$. Accordingly, $e_q(\lambda^T)$ is contained in the ideal $I_{\mathrm{box}}$.
\end{proof}

\begin{prop}\label{prop-supersorted} 
There is an equivalence of categories
\[ \Mat(\SoMNSp) \cong \RepMNs .\]
\end{prop}
\begin{proof}
\fullref{lem-ideal} shows that the sorted web category $\MNw$, in which webs have $m$ green and $n$ red boundary points both on the bottom and on the top, is equivalent to $\End_{\Usu}(\bV_q^{\bullet}(\C_q^{m|n}\otimes \C_q^{N|M}))$, considered as a category. Via the $\Usd$-weight space decomposition in part (b) of \fullref{thm-supersuperHowe}, $\MNw$ gives a presentation of the morphism spaces in  $\RepMNs$ between objects of the form 
\[\bV_q^{k_{1}}\C_q^{N|M}\otimes\cdots\otimes\bV_q^{k_{m}}\C_q^{N|M}
\otimes\Sym_q^{k_{m+1}}\C_q^{N|M}\otimes\cdots\otimes\Sym_q^{k_{m+n}}\C_q^{N|M}.\]
Any object in $\RepMNs$ is a formal sum of such objects, for suitable $m,n\in \Z_{\geq 0}$, and the conclusion follows.
\end{proof}

\begin{rem}\label{rem-superbraiding}
Recall that $\RepMN$ is a braided monoidal category, where the braiding $\beta^R_{\cdot,\cdot}$ is given by the \textit{universal} $R$\textit{-matrix for} $\glMN$, see \cite{ya}. As before, we use a rescaled braiding $\tilde \beta^R_{\cdot,\cdot}$, where we follow the conventions from \cite[(3.12)]{qs} except that we substitute $q$ by $q^{-1}$ in their formulas. In particular, $\tilde \beta^R_{\C_q^{N|M},\C_q^{N|M}}$ acts as $-q^{-1}$ on $\bV^2_q \C_q^{N|M}$ and as $q$ on $\Sym^2_q \C_q^{N|M}$. 
\end{rem}

\begin{thm} (\textbf{The diagrammatic presentation}) There is an equivalence of braided monoidal categories
\[ (\Mat(\MNSp),\beta^{\Sup}_{\cdot,\cdot}) \cong (\RepMN,\beta^{R}_{\cdot,\cdot}) .\]
\end{thm}
\begin{proof} The equivalence from \fullref{prop-supersorted} can be extended to a monoidal functor between the categories $\Mat(\MNSp)$ and $\RepMN$ as in \fullref{defn-mainfunctor2}. We can also copy the proof of \fullref{prop-braidfunctor}, where we use \fullref{rem-superbraiding} to prove that this functor respects the braiding. Equivalence via this 
functor follows then as in \fullref{thm-equi}.
\end{proof}

\begin{rem}\label{rem-duals}
In \cite[Section 6]{qs} the authors show how to extend a diagrammatic presentation of $\mathfrak{gl}_{N|M}\text{-}\textbf{Mod}_{\mathrm{e}}$ to diagrammatically encode the full subcategory of $\Usu$-modules tensor generated by exterior powers and their duals. Graphically, this involves the introduction of additional objects corresponding to the duals of exterior powers, downward-oriented edges (to represent identity morphisms on duals) and cap and cup webs (which represent co-evaluation and evaluation morphisms). Additional web relations including analogues of \eqref{eq-newmoves} are introduced to encode basic relationships between exterior powers and their duals. The extension of the diagrammatic presentation to include duals is then tautological and \cite[Theorem 6.5]{qs} and \cite[Proposition 6.16]{qs} show that the extended presentation functor is fully faithful.

They further show in \cite[Proposition 6.15]{qs} that their graphical calculus allows the computation of the Reshetikhin--Turaev $\glMN$-tangle invariants for tangles labeled with exterior powers of the vector representation. 
\end{rem}

The same \textit{spiderization strategy}---with minimal changes in proofs---gives an extension of our diagrammatic presentation $\MNSp$ of $\RepMN$ to one for the full subcategory of $\Usu$-modules tensor generated by exterior and symmetric powers and their duals. This spiderized green--red web category directly allows the computation of Reshetikhin--Turaev $\glMN$-tangle invariants for tangles labeled with exterior as well as symmetric powers of the vector representation. The cabling strategy from \fullref{sec-applications} can then be used to compute these invariants with respect to arbitrary irreducible representations.

Lastly, we have a direct consequence of the discussion in this 
section and \fullref{prop-homfly}. It is based on the facts that $\MNSp$ is defined as a quotient of $\InSp$ and that the spiderization in \cite[Section 6]{qs} respects the specialization $a=q^{N-M}$ of the relations \eqref{eq-newmoves}, which are sufficient to compute colored HOMFLY--PT polynomials of braid closures. 

\begin{cor}\label{cor-supersym} We have:
\begin{enumerate}
\item The Reshetikhin--Turaev $\glMN$-tangle invariant of a labeled tangle depends only on $N-M$. In the case of a labeled link, it agrees with the specialization $a=q^{N-M}$ of the corresponding colored HOMFLY--PT polynomial.

\item The green--red symmetry on $\InSp$ descends to a symmetry between 
the categories $\MNSp$ and $M|N\text{-}\!\textbf{Web}_{\mathrm{gr}}$. Hence, there is a symmetry between the representation categories of $\Usu$ and $\textbf{U}_q(\mathfrak{gl}_{M|N})$ that transposes Young diagrams indexing irreducibles. 

\item The symmetry of HOMFLY--PT polynomials described in \fullref{prop-homfly} is a stabilized version of the symmetry between colored Reshetikhin--Turaev $\glMN$-link invariants and $\mathfrak{gl}_{M|N}$-link invariants which transposes Young diagrams and inverts $q$.\qed
\end{enumerate}
\end{cor}

This confirms decategorified analogues of predictions about relationships between colored HOMFLY--PT homology and conjectural colored $\glMN$-link homologies, see \cite{ggs}.
%%%%%%%%%%%%%%%%%%%% End Matter %%%%%%%%%%%%%%%%%%%% 

%%%%%%%%%%%%%%%%%%%%%%%%%%%%%%%%%%%%%%%%
%%%                                  %%%
%%%            Bibliography          %%%
%%%                                  %%%
%%%%%%%%%%%%%%%%%%%%%%%%%%%%%%%%%%%%%%%%

\bibliographystyle{plainurl}
\bibliography{super-webs}
%%%%%%%%%%%%%%%%%%% end of paper %%%%%%%%%%%%%%%%%%%%%%%%%
\end{document}